\documentclass[10pt,a4paper]{amsart}

\numberwithin{equation}{section}
\allowdisplaybreaks 

\usepackage{mathtools,amssymb,eucal,mathrsfs,setspace,graphicx,stmaryrd} 
\usepackage[noadjust]{cite} 
\usepackage[margin=20mm,foot=10mm]{geometry}

\usepackage{array}
\newcolumntype{C}{>{$}c<{$}} 

\usepackage[largesc,theoremfont]{newtxtext}      
\usepackage[libertine,cmbraces,varbb]{newtxmath} 
\begingroup
  \makeatletter
  \@for\theoremstyle:=definition,remark,plain\do{%
    \expandafter\g@addto@macro\csname th@\theoremstyle\endcsname{%
      \addtolength\thm@preskip{.5\baselineskip plus .2\baselineskip minus .2\baselineskip}
      \addtolength\thm@postskip{.5\baselineskip plus .2\baselineskip minus .2\baselineskip}
    }%
  }
\endgroup

\usepackage{enumitem}
\setitemize{leftmargin=*}   
\setenumerate{leftmargin=*, 
  label=\textup{(\alph*)},  
	ref=\textup{\alph*}}      

\usepackage{tikz}
\usetikzlibrary{arrows}

\usepackage[colorlinks=true,citecolor=red,linkcolor=blue]{hyperref} 

\usepackage[capitalise,noabbrev]{cleveref} 

\newcommand{\eps}{\varepsilon}

\newcommand{\pd}{\partial}     

\renewcommand{\ge}{\geqslant} 

\renewcommand{\cong}{\simeq} 

\newcommand{\cc}{\mathsf{c}}   
\newcommand{\dd}{\mathrm{d}}   
\newcommand{\ee}{\mathsf{e}}   
\newcommand{\ii}{\mathsf{i}}   
\newcommand{\kk}{\mathsf{k}}   
\newcommand{\qq}{\mathsf{q}}   
\newcommand{\uu}{\mathsf{u}}   
\providecommand{\vv}{\mathsf{v}}\renewcommand{\vv}{\mathsf{v}}   
\newcommand{\yy}{\mathsf{y}}   
\newcommand{\zz}{\mathsf{z}}   

\newcommand{\rrr}{\mathsf{r}}     
\newcommand{\sss}{\mathsf{s}}     
\newcommand{\ttt}{\mathsf{t}}     

\newcommand{\fracuv}{\frac{\uu}{\vv}} 
\newcommand{\tfracuv}{\tfrac{\uu}{\vv}}

\newcommand{\wun}{\vvmathbb{1}}  

\DeclarePairedDelimiter{\brac}{\lparen}{\rparen}   
\DeclarePairedDelimiter{\sqbrac}{\lbrack}{\rbrack} 
\DeclarePairedDelimiter{\set}{\lbrace}{\rbrace}
\newcommand{\st}{\mspace{5mu} {:} \mspace{5mu}}    
\DeclarePairedDelimiter{\abs}{\lvert}{\rvert}

\newcommand{\no}[1]{\mathopen{:} #1 \mathclose{:}} 

\DeclarePairedDelimiterX{\comm}[2]{\lbrack}{\rbrack}{#1 , #2}  
\DeclarePairedDelimiterX{\acomm}[2]{\lbrace}{\rbrace}{#1 , #2} 
\DeclarePairedDelimiterX{\inner}[2]{\langle}{\rangle}{#1 \delimsize\vert #2} 
\DeclarePairedDelimiterX{\bilin}[2]{\lparen}{\rparen}{#1 , #2} 
\DeclarePairedDelimiterX{\super}[2]{\lparen}{\rparen}{#1 \delimsize\vert \mathopen{} #2} 

\newcommand{\lra}{\mathrel{\longrightarrow}}
\newcommand{\ira}{\mathrel{\hookrightarrow}}    

\newcommand{\lma}{\overset{\sfsymb^{1/2}}{\longmapsto}} 

\newcommand{\dses}[5]{0 \lra #1 \overset{#2}{\lra} #3 \overset{#4}{\lra} #5 \lra 0} 

\newcommand{\fld}[1]{\mathbb{#1}}    
\newcommand{\alg}[1]{\mathfrak{#1}}  
\newcommand{\grp}[1]{\mathsf{#1}}    
\newcommand{\Mod}[1]{\mathcal{#1}}   
\newcommand{\VOA}[1]{\mathsf{#1}}    
\newcommand{\categ}[1]{\mathscr{#1}} 

\newcommand{\ZZ}{\fld{Z}}
\newcommand{\NN}{\ZZ_{\ge 0}} 

\newcommand{\RR}{\fld{R}}
\newcommand{\CC}{\fld{C}}

\newcommand{\SLG}[2]{\grp{#1}_{#2}}            

\newcommand{\finite}[1]{\overline{#1}}
\newcommand{\affine}[1]{\widehat{#1}}

\newcommand{\SLA}[2]{\alg{#1}_{#2}}                      
\newcommand{\AKMA}[2]{\affine{\alg{#1}}_{#2}}            

\newcommand{\sltwo}{\SLA{sl}{2}}
\newcommand{\slthree}{\SLA{sl}{3}}

\newcommand{\aslthree}{\AKMA{sl}{3}}

\newcommand{\rlat}{\grp{Q}}                                          
\newcommand{\pwlat}[1]{\grp{P}^{#1}_{\ge}}                           
\newcommand{\wvec}{\rho}                                             
\newcommand{\fwt}[1]{\omega_{#1}}                                    

\newcommand{\wgrp}{\grp{S}_3}                                        
\newcommand{\wref}[1]{w_{#1}}                                        
\newcommand{\awgrp}{\affine{\grp{S}}_3}                              

\newcommand{\surv}[2]{\Sigma_{#1,#2}}                                
\newcommand{\survuv}{\surv{\uu}{\vv}}
\newcommand{\infwts}[2]{\Gamma_{#1,#2}}                              
\newcommand{\infwtsuv}{\infwts{\uu}{\vv}}

\newcommand{\rspar}[1]{\set[\big]{\begin{smallmatrix} #1 \end{smallmatrix}}}   
\newcommand{\rsparrs}{\rspar{r_0 & r_1 & r_2 \\ s_0 & s_1 & s_2}}
\newcommand{\rsnewpar}[1]{(#1)}
\newcommand{\rsnewparrs}{\rsnewpar{\rrr,\sss}}

\newcommand{\wtpar}[1]{\Gamma \rspar{#1}}                            
\newcommand{\wtparrs}{\wtpar{r_0 & r_1 & r_2 \\ s_0 & s_1 & s_2}}
\newcommand{\wtnewpar}[1]{\Gamma(#1)}                                
\newcommand{\wtnewparrs}{\wtnewpar{\rrr,\sss}}

\newcommand{\outaut}{\nabla}                                         

\newcommand{\bpsymb}{\VOA{BP}}
\newcommand{\Wsymb}{\VOA{W}}
\newcommand{\Wthree}{\Wsymb_3}

\newcommand{\uaffvoa}[2]{\VOA{V}^{#1}\brac{#2}}                   
\newcommand{\saffvoa}[2]{\VOA{L}_{#1}\brac{#2}}                   

\newcommand{\uslvoa}[1]{\uaffvoa{#1}{\slthree}}                   
\newcommand{\ubpvoa}[1]{\bpsymb^{#1}}                             
\newcommand{\sslvoa}[1]{\saffvoa{#1}{\slthree}}                   
\newcommand{\sbpvoa}[1]{\bpsymb_{#1}}                             

\newcommand{\virminmod}[2]{\VOA{M}\brac{#1,#2}}                   
\newcommand{\slminmod}[2]{\VOA{A}_2\brac{#1,#2}}                  
\newcommand{\bpminmod}[2]{\bpsymb\brac{#1,#2}}                    

\newcommand{\uslvoak}{\uslvoa{\kk}}                               
\newcommand{\ubpvoak}{\ubpvoa{\kk}}                               
\newcommand{\sbpvoak}{\sbpvoa{\kk}}                               
\newcommand{\bpminmoduv}{\VOA{BP}\brac{\uu,\vv}}                  

\newcommand{\uWvoa}[1]{\Wthree^{#1}}                       
\newcommand{\Wminmod}[2]{\Wthree \brac{#1,#2}}                
\newcommand{\uWvoak}{\uWvoa{\kk}}                                 
\newcommand{\Wminmoduv}{\Wminmod{\uu}{\vv}}               

\newcommand{\halflattice}{\Pi}                           

\newcommand{\bpcc}[1]{\cc^{\bpsymb}_{#1}}                         
\newcommand{\Wcc}[1]{\cc^{\Wthree}_{#1}}                          
\newcommand{\lcc}[1]{\cc^{\halflattice}_{#1}}                     

\newcommand{\Wcck}{\Wcc{\kk}}

\newcommand{\bpccuv}{\bpcc{\uu,\vv}}
\newcommand{\Wccuv}{\Wcc{\uu,\vv}}
\newcommand{\lccuv}{\lcc{\uu,\vv}}

\newcommand{\modify}[1]{\widetilde{#1}}                           
\newcommand{\tbpcc}[1]{\modify{\cc}^{\bpsymb}_{#1}}               
\newcommand{\tbpccuv}{\tbpcc{\uu,\vv}}

\newcommand{\twist}{\textup{tw}}                          

\newcommand{\conjsymb}{\gamma}                            
\newcommand{\sfsymb}{\sigma}                              

\newcommand{\conjmod}[1]{\conjsymb\bigl(#1\bigr)}         
\newcommand{\sfmod}[2]{\sfsymb^{#1}\bigl(#2\bigr)}        

\newcommand{\irrsymb}{\Mod{H}}

\newcommand{\relsymb}{\Mod{R}}
\newcommand{\Wmodsymb}{\Mod{W}}

\newcommand{\slsymb}{\Mod{L}}

\newcommand{\ihw}[1]{\irrsymb_{#1}}                            
\newcommand{\twihw}[1]{\irrsymb^{\twist}_{#1}}                 
\newcommand{\twrhw}[1]{\relsymb^{\twist}_{#1}}                 

\newcommand{\fslihw}[1]{\finite{\slsymb}_{#1}}                 
\newcommand{\slihw}[1]{\slsymb_{#1}}                           

\newcommand{\Wihw}[1]{\Wmodsymb_{#1}}                          

\newcommand{\lmod}[1]{\halflattice_{#1}}                       

\newcommand{\tilderhw}[1]{\modify{\relsymb}_{#1}}              

\newcommand{\ihwpar}[1]{\irrsymb \rspar{#1}}                   
\newcommand{\ihwparrs}{\irrsymb \rsparrs}
\newcommand{\twihwpar}[1]{\irrsymb^{\twist} \rspar{#1}}        
\newcommand{\twihwparrs}{\irrsymb^{\twist} \rsparrs}
\newcommand{\twrhwpar}[2]{\relsymb^{\twist}_{#1} \rspar{#2}}   
\newcommand{\twrhwparrs}[1]{\relsymb^{\twist}_{#1} \rsparrs}
\newcommand{\trhwpar}[1]{\modify{\relsymb} \rspar{#1}}         

\newcommand{\ihwnewpar}[1]{\irrsymb(#1)}                       
\newcommand{\ihwnewparrs}{\ihwnewpar{\rrr,\sss}}
\newcommand{\twihwnewpar}[1]{\irrsymb^{\twist}(#1)}            
\newcommand{\twihwnewparrs}{\twihwnewpar{\rrr,\sss}}
\newcommand{\twrhwnewpar}[2]{\relsymb^{\twist}_{#1}(#2)}       
\newcommand{\twrhwnewparrs}[1]{\twrhwnewpar{#1}{\rrr,\sss}}
\newcommand{\trhwnewpar}[1]{\modify{\relsymb}(#1)}             

\newcommand{\Wihwparrs}{\Wmodsymb \rsparrs}
\newcommand{\Wihwnewpar}[1]{\Wmodsymb \rsnewpar{#1}}           
\newcommand{\Wihwnewparrs}{\Wmodsymb \rsnewparrs}

\newcommand{\wcat}[1]{\categ{W}_{#1}}             
\newcommand{\twcat}[1]{\categ{W}^{\twist}_{#1}}   
\newcommand{\wcatuv}{\wcat{\uu,\vv}}
\newcommand{\twcatuv}{\twcat{\uu,\vv}}

\DeclareMathOperator{\tr}{tr}
\DeclareMathOperator{\chmap}{ch}              
\DeclareMathOperator{\tchmap}{\widetilde{ch}} 

\newcommand{\Gr}[1]{\sqbrac[\big]{#1}}        
\newcommand{\traceover}[1]{\tr_{\raisebox{-2pt}{$\scriptstyle #1$}}} 

\newcommand{\chnoargs}[1]{\chmap \Gr{#1}}     
\newcommand{\tchnoargs}[1]{\tchmap \Gr{#1}}     

\newcommand{\charac}[4]{\chnoargs{#1} \left( #2 \middle\bracevert #3 \middle\bracevert #4 \right)}

\newcommand{\bpch}[4]{\chnoargs{#1} \left( #2 \middle\bracevert #3 \middle\bracevert #4 \right)} 
\newcommand{\Wch}[2]{\chnoargs{#1} \left( #2 \right)}                                            
\newcommand{\lch}[3]{\chnoargs{#1} \left( #2 \middle\bracevert #3 \right)}                       

\newcommand{\bpchm}[1]{\bpch{#1}{\theta}{\zeta}{\tau}} 
\newcommand{\Wchm}[1]{\Wch{#1}{\tau}}
\newcommand{\lchm}[1]{\lch{#1}{\zeta}{\tau}}

\newcommand{\bpopf}[5]{\chnoargs{#1} \left( #2 \middle\bracevert #3 \middle\bracevert #4 \mathbin{;} #5 \right)} 
\newcommand{\Wopf}[3]{\chnoargs{#1} \left( #2 \mathbin{;} #3 \right)}                                            

\newcommand{\bpopfm}[1]{\bpopf{#1}{\theta}{\zeta}{\tau}{u}} 
\newcommand{\Wopfm}[1]{\Wopf{#1}{\tau}{u}}

\newcommand{\tbpch}[4]{\tchnoargs{#1} \left( #2 \middle\bracevert #3 \middle\bracevert #4 \right)} 
\newcommand{\tbpchm}[1]{\tbpch{#1}{\theta}{\zeta}{\tau}}
\newcommand{\tbpopf}[5]{\tchnoargs{#1} \left( #2 \middle\bracevert #3 \middle\bracevert #4 \mathbin{;} #5 \right)}
\newcommand{\tbpopfm}[1]{\tbpopf{#1}{\theta}{\zeta}{\tau}{u}}

\newcommand{\slthreecharacnoarg}[1]{\chi_{#1}}
\newcommand{\slthreecharac}[2]{\chi_{#1}(#2)}
\newcommand{\chsympoly}[4]{h_{#1}\left(#2, #3, #4\right)}

\newcommand{\modfont}[1]{\mathsf{#1}}
\newcommand{\mods}{\modfont{S}}                        
\newcommand{\modt}{\modfont{T}}                        
\newcommand{\modc}{\modfont{C}}                        

\newcommand{\Smatrix}[2]{\mods_{#1}^{#2}}

\newcommand{\WSmatrix}[2]{\Smatrix{#1,#2}{\Wthree}}

\newcommand{\vacSmatrix}[1]{\Smatrix{\textup{vac.}}{#1}}
\newcommand{\vacWSmatrix}[1]{\WSmatrix{\textup{vac.}}{#1}}

\newcommand{\fuse}{\mathbin{\times}}                                              
\newcommand{\grfuse}{\mathbin{\boxtimes}}                                         

\newcommand{\fuscoeff}[3]{\genfrac{(}{)}{0pt}{0}{#3}{#1 \quad #2}}      
\newcommand{\Wfuscoeff}[3]{\mathcal{N}_{#1, #2}^{\ooalign{$\scriptstyle \Wthree$\cr\hphantom{$\scriptstyle #1,$}$\scriptstyle #3$\cr}}} 
\newcommand{\slfuscoeff}[4]{\mathcal{N}_{#2, #3}^{#1\ #4}}              
\newcommand{\sltencoeff}[3]{\mathsf{N}_{#1, #2}^{\hphantom{#1, #2} #3}} 

\newcommand{\Gfusion}[2]{\Gr{#1} \grfuse \Gr{#2}}
\newcommand{\fusion}[2]{#1 \fuse #2}

\newcommand{\cft}{conformal field theory}
\newcommand{\cfts}{conformal field theories}

\newcommand{\lcfts}{logarithmic conformal field theories}

\newcommand{\hw}{highest-weight}
\newcommand{\hwv}{\hw\ vector}
\newcommand{\hwvs}{\hwv s}
\newcommand{\hwm}{\hw\ module}
\newcommand{\hwms}{\hwm s}
\newcommand{\rhw}{relaxed highest-weight}
\newcommand{\rhwv}{\rhw\ vector}
\newcommand{\rhwvs}{\rhwv s}
\newcommand{\rhwm}{\rhw\ module}
\newcommand{\rhwms}{\rhwm s}

\newcommand{\vo}{vertex operator}
\newcommand{\voa}{\vo\ algebra}
\newcommand{\voas}{\voa s}

\newcommand{\ope}{operator product expansion}
\newcommand{\opes}{\ope s}
\newcommand{\qhr}{quantum hamiltonian reduction} 
\newcommand{\qhrs}{\qhr s}
\newcommand{\emt}{energy-momentum tensor}

\newcommand{\lhs}{left-hand side}

\newcommand{\rhs}{right-hand side}

\newcommand{\fdim}{finite-dimensional}
\newcommand{\infdim}{infinite-dimensional}

\newcommand{\bp}{Bershadsky--Polyakov}

\theoremstyle{plain}
\newtheorem{theorem}{Theorem}[section]
\newtheorem{corollary}[theorem]{Corollary}
\newtheorem{lemma}[theorem]{Lemma}
\newtheorem{proposition}[theorem]{Proposition}

\newtheorem{mthm}{Main Theorem}

\newtheorem{conjecture}{Conjecture}
\newtheorem{definition}[theorem]{Definition}

\Crefname{conjecture}{Conjecture}{Conjectures}

\DeclareRobustCommand{\SkipTocEntry}[5]{} 

\makeatletter
\newcommand\@dotsep{4.5}
\def\@tocline#1#2#3#4#5#6#7{\relax
  \ifnum #1>\c@tocdepth 
  \else
    \par \addpenalty\@secpenalty\addvspace{#2}%
    \begingroup \hyphenpenalty\@M
    \@ifempty{#4}{%
      \@tempdima\csname r@tocindent\number#1\endcsname\relax
    }{%
      \@tempdima#4\relax
    }%
    \parindent\z@ \leftskip#3\relax
    \advance\leftskip\@tempdima\relax
    \rightskip\@pnumwidth plus1em \parfillskip-\@pnumwidth
    #5\leavevmode\hskip-\@tempdima #6\relax
    \leaders\hbox{$\m@th
      \mkern \@dotsep mu\hbox{.}\mkern \@dotsep mu$}\hfill
    \hbox to\@pnumwidth{\@tocpagenum{#7}}\par
    \nobreak
    \endgroup
  \fi}
\makeatother

\begin{document}

\title[Modularity and fusion rules for the Bershadsky--Polyakov algebras]{Modularity of Bershadsky--Polyakov minimal models}

\author[Z~Fehily]{Zachary Fehily}
\address[Zachary Fehily]{
School of Mathematics and Statistics \\
University of Melbourne \\
Parkville, Australia, 3010.
}
\email{zfehily@student.unimelb.edu.au}

\author[D~Ridout]{David Ridout}
\address[David Ridout]{
School of Mathematics and Statistics \\
University of Melbourne \\
Parkville, Australia, 3010.
}
\email{david.ridout@unimelb.edu.au}

\begin{abstract}
	The Bershadsky--Polyakov algebras are the original examples of nonregular W-algebras, obtained from the affine vertex operator algebras associated with $\mathfrak{sl}_3$ by quantum hamiltonian reduction.  In \cite{FehCla20}, we explored the representation theories of the simple quotients of these algebras when the level $\mathsf{k}$ is nondegenerate-admissible.  Here, we combine these explorations with Adamovi\'c's inverse quantum hamiltonian reduction functors to study the modular properties of Bershadsky--Polyakov characters and deduce the associated Grothendieck fusion rules.  The results are not dissimilar to those already known for the affine vertex operator algebras associated with $\mathfrak{sl}_2$, except that the role of the Virasoro minimal models in the latter is here played by the minimal models of Zamolodchikov's $\mathsf{W}_3$ algebras.
\end{abstract}

\maketitle

\onehalfspacing

\vfill

\setcounter{tocdepth}{2}
\tableofcontents

\vfill

\clearpage

\section{Introduction} \label{sec:intro}

\subsection{Background} \label{sec:back}

The \bp\ algebras $\ubpvoak$, $\kk \in \CC$, are among the simplest and best known nonregular W-algebras \cite{PolGau90, BerCon91}.  They may be characterised \cite{KacQua03} as the subregular (or minimal) \qhrs\ of the level-$\kk$ universal affine vertex algebras $\uaffvoa{\kk}{\slthree}$.  This paper is a sequel to \cite{FehCla20} in which the representation theory of $\ubpvoak$ and its simple quotient $\sbpvoak$ was investigated.  Here, we are interested in the characters, modular transformations and fusion rules of the simple quotient $\sbpvoak$ when $\kk$ is a nondegenerate admissible level.

When $\kk + \frac{3}{2} \in \ZZ_{\ge 0}$, $\sbpvoak$ is known to be rational and $C_2$-cofinite \cite{AraRat13,AraAss15}.  For these levels, which are admissible, modular transformations of characters and fusion rules are in principle known \cite{AraRat19}.  For admissible levels with $\vv > 2$, these being the \emph{nondegenerate} admissible levels, $\sbpvoak$ is nonrational in the category of weight modules.  This was shown \cite{FehCla20} by combining the untwisted and twisted Zhu algebras of $\ubpvoak$ with Arakawa's results \cite{AraRep05} on minimal \qhrs.  A classification of simple weight $\sbpvoak$-modules, with \fdim\ weight spaces, was obtained along with a construction of certain nonsemisimple weight $\sbpvoak$-modules.

Nonrationality is a significant obstacle to computing modular transformations and fusion rules, which is essential data for constructing \lcfts.  A general framework for computing this data for many classes of nonrational \voas\ was proposed in \cite{CreLog13,RidVer14} and is usually referred to as the \emph{standard module formalism}.

This formalism has been shown to reproduce the (Grothendieck) fusion rules in many examples \cite{CreRel11,CreMod12,CreMod13,RidMod13,RidBos14,MorBou15} that can be independently verified \cite{GabFus01,FucNon03,EbeVir06,GabFro07,RidLog07,WooFus09,RidFus10,TsuTen12,AdaFus19,AllBos20,CreTen20b}.  Otherwise, applications of the standard module formalism consistently pass the usual consistency tests, for example that the Grothendieck fusion coefficients are nonnegative integers \cite{BabTak12,CanFus15,RidAdm17,KawAdm21}.

At nondegenerate admissible levels, the representation theory of $\sbpvoak$ shares many features with that of the simple affine \voa\ $\saffvoa{\kk}{\sltwo}$.  In particular, a key result in the standard module analysis for $\saffvoa{\kk}{\sltwo}$ is the fact that Virasoro minimal model characters appear as factors of the standard characters \cite{CreMod13,KawRel18}.  Consequently, the modular S-transforms and Grothendieck fusion rules of $\saffvoa{\kk}{\sltwo}$ are also naturally expressed in terms of their Virasoro minimal model analogues.  This was subsequently explained by Adamovi\'{c} \cite{AdaRea17} using a construction of the standard modules from simple Virasoro modules.  Such a construction amounts to a functorial version of the old \emph{inverse} \qhr\ introduced by Semikhatov \cite{SemInv94} and is believed to generalise widely.

In \cite{AdaRea20}, inverse \qhr\ was generalised to the \bp\ algebras, with the role of the Virasoro minimal models being played by Zamolodchikov's $\Wthree$ minimal models \cite{ZamInf85}.  This generalisation leads to the expectation that \bp\ characters, modularity and fusion can be understood in terms of $\Wthree$ minimal model data, using the standard module formalism.  In this paper, we confirm this expectation and thereby provide further evidence for the claim that Adamovi\'{c}'s inverse \qhr\ functors are a fundamental tool for analysing the representation theory of general W-algebras.  Further evidence will be presented in a forthcoming article \cite{FehSub21} that will study inverse \qhr\ for the subregular W-algebras associated with $\SLA{sl}{n}$, $\saffvoa{\kk}{\sltwo}$ and $\sbpvoak$ being the subregular algebras for $n=2$ and $3$, respectively.

\subsection{Results} \label{sec:results}

Assume that $\kk \in \CC$ is nondegenerate-admissible, meaning that it defines parameters $\uu,\vv \in \ZZ_{\ge3}$ by \eqref{eq:fraclevels}.  As shown in \cite{FehCla20}, the weight $\sbpvoak$-modules then include simple \hwms\ $\ihw{\lambda}$, $\lambda \in \survuv$, and generically simple twisted \rhwms\ $\twrhw{[j],[\lambda]}$, $[j] \in \CC/\ZZ$ and $[\lambda] \in \infwtsuv / \ZZ_3$, along with their images under the spectral flow functors $\sfsymb{\ell}$, $\ell \in \frac{1}{2} \ZZ$.  Here, $\survuv$ and $\infwtsuv$ are certain finite sets of $\aslthree$-weights defined in \cref{sec:bpsimple}, as is the $\ZZ_3$-action on $\infwtsuv$.  The standard modules are the $\tilderhw{[j],[\lambda]}^{\ell} = \sfmod{\ell+1/2}{\twrhw{[j-\kappa],[\lambda]}}$, where $\kappa = \frac{1}{6}(2\kk+3)$ and $[j]$ is restricted to lie in $\RR/\ZZ$.
\begin{mthm}[\cref{prop:stopfs}]
	Let $\kk$ be nondegenerate-admissible.  Then, the characters of the standard modules are usually linearly dependent and one has to instead consider one-point functions.  These have the form
	\begin{equation}
		\tbpopfm{\tilderhw{[j],[\lambda]}^{\ell}}
		= \ee^{2\pi\ii\kappa\brac[\big]{\theta - \ell(\ell+1)\tau}} \frac{\Wopfm{\Wihw{[\lambda]}}}{\eta(\tau)^2}
			\sum_{m \in \ZZ} \ee^{2\pi\ii m(j+2\kappa\ell)} \delta(\zeta+\ell\tau-m),
	\end{equation}
	where $\Wihw{[\lambda]}$ is a simple module for the level-$\kk$ $\Wthree$ minimal model.  Moreover, there exist choices for $u$ such that these standard one-point functions are linearly independent.
\end{mthm}

As the standard modules are parametrised by a continuous label $[j] \in \RR/\ZZ$ (as well as discrete labels $\ell$ and $[\lambda]$), the S-transform of a given standard one-point function will not be a weighted sum of one-point functions, but rather a weighted \emph{integral}.  Again, the $\Wthree$ minimal model S-matrix makes a conspicuous appearance.
\begin{mthm}[\cref{thm:stopfS}]
	Let $\kk$ be nondegenerate-admissible.  Then, the S-transform of the one-point function of $\tilderhw{[j],[\lambda]}^{\ell}$ is given by
	\begin{align}
		&\tbpopf{\tilderhw{[j],[\lambda]}^{\ell}}{\theta-\frac{\zeta^2}{\tau}-\frac{\zeta}{\tau}+\zeta}{\frac{\zeta}{\tau}}{-\frac{1}{\tau}}{\frac{u}{\tau^{\Delta_u}}} \\
		&\mspace{200mu} = \frac{\abs{\tau}}{-\ii \tau} \sum_{\ell' \in \ZZ} \int_{\RR/\ZZ} \sum_{[\lambda'] \in \infwtsuv / \ZZ_3}
			\Smatrix{\ell,[j],[\lambda]}{\ell',[j'],[\lambda']} \tbpopfm{\tilderhw{[j'],[\lambda']}^{\ell'}} \, \dd [j'], \notag
	\end{align}
	where $\Delta_u$ is the conformal weight of $u$ and the entries of the ``S-matrix'' (integral kernel) are
	\begin{equation}
		\Smatrix{\ell,[j],[\lambda]}{\ell',[j'],[\lambda']} = \WSmatrix{[\lambda]}{[\lambda']} \ee^{-2 \pi \ii \brac[\big]{2 \kappa \ell \ell' + \ell (j'-\kappa) + (j-\kappa) \ell'}}.
	\end{equation}
\end{mthm}

The vacuum module $\ihw{\kk\fwt{0}}$ is not a standard module, but like all simple weight $\sbpvoak$-modules it admits an infinite (one-sided convergent) resolution by standard modules (\cref{prop:type3reso}).  The Euler--Poincar\'{e} principle then allows us to calculate its modular S-transform.
\begin{mthm}[\cref{cor:vacSmatrix}]
	Let $\kk$ be nondegenerate-admissible.  Then, the S-transform of the one-point function of the vacuum module is given by
	\begin{align}
		&\tbpopf{\ihw{\kk\fwt{0}}}{\theta-\frac{\zeta^2}{\tau}-\frac{\zeta}{\tau}+\zeta}{\frac{\zeta}{\tau}}{-\frac{1}{\tau}}{\frac{u}{\tau^{\Delta_u}}} \\
		&\mspace{200mu} = \frac{\abs{\tau}}{-\ii \tau} \sum_{\ell' \in \ZZ} \int_{\RR/\ZZ} \sum_{[\lambda'] \in \infwtsuv / \ZZ_3}
			\vacSmatrix{\ell',[j'],[\lambda']} \tbpopfm{\tilderhw{[j'],[\lambda']}^{\ell'}} \, \dd [j'], \notag
	\end{align}
	where the entries of the ``vacuum S-matrix" are given by
	\begin{equation}
		\vacSmatrix{\ell',[j'],[\lambda']}
		= \vacWSmatrix{[\lambda']} \frac{\ee^{2\pi\ii \kappa \ell'} \ee^{\pi\ii (j'-\kappa)}}
			{2\cos \brac[\big]{3\pi (j'-\kappa)} - \sum_{i \in \ZZ_3} 2\cos \brac[\big]{\pi a_i(j',\lambda')}}.
	\end{equation}
	Here, $a_i(j,\lambda) = (j- \kappa) + 2j^{\twist}\left( \outaut^i(\lambda) \right)$ and $j^{\twist}$ is defined in \eqref{eq:bptwhwdata}.
\end{mthm}

Having established the modular S-transforms of the standard modules and the vacuum module, one can now apply the (conjectural) standard Verlinde formula \eqref{eq:SVF} to compute predicted Grothendieck fusion rules for the standard modules.  This is quite a nontrivial calculation, requiring several obscure identities involving $\Wthree$ minimal model fusion coefficients, but the result is as follows.
\begin{mthm}[\cref{thm:standardfusion}]
	Let $\kk$ be nondegenerate-admissible.  Then, the Grothendieck fusion rules of the standard modules are
	\begin{multline}
	\Gfusion{\tilderhw{[j],[\lambda]}^{\ell}}{\tilderhw{[j'],[\lambda']}^{\ell'}}
    = \sum_{[\lambda''] \in \infwtsuv / \ZZ_3} \Wfuscoeff{[\lambda]}{[\lambda']}{[\lambda'']}
	    \brac*{\Gr{\tilderhw{[j+j'-4\kappa],[\lambda'']}^{\ell+\ell'+2}} +  \Gr{\tilderhw{[j+j'+2\kappa],[\lambda'']}^{\ell+\ell'-1}}} \\
    + \sum_{[\lambda''] \in \infwtsuv / \ZZ_3} \sum_{i \in \ZZ_3} \brac*{
	    \Wfuscoeff{[\lambda]}{[\wtnewpar{\rrr', \sss'-\fwt{i}+\fwt{i+1}}]}{[\lambda'']} \Gr{\tilderhw{[j+j'-2\kappa],[\lambda'']}^{\ell+\ell'+1}}
	    + \Wfuscoeff{[\lambda]}{[\wtnewpar{\rrr', \sss'+\fwt{i}-\fwt{i+1}}]}{[\lambda'']} \Gr{\tilderhw{[j+j'],[\lambda'']}^{\ell+\ell'}}},
	\end{multline}
	where we parametrise $\lambda'$ as $\wtnewpar{\rrr',\sss'}$ as in \eqref{eq:rspar}.
\end{mthm}

As every simple weight $\sbpvoak$-module may be resolved in terms of standard modules, this result implies the Grothendieck fusion rules for arbitrary simple weight modules.  These general results are doubtlessly unpleasant and we do not attempt to derive them in full generality.  Instead, we note an interesting generalisation of an observation of \cite{CreMod13} for $\saffvoa{\kk}{\sltwo}$.
\begin{mthm}[\cref{prop:fusringiso}]
	If $\kk$ is nondegenerate-admissible, then the simple \hwms\ $\ihw{\lambda}$, with $\lambda = \wtnewparrs$ and $\sss = [\vv-2,-1,0]$, span a subring of the \emph{fusion ring} of $\sbpvoak$ that is isomorphic to the fusion ring of the rational affine \voa\ $\sslvoa{\uu-3}$.
\end{mthm}

\subsection{Outline} \label{sec:outline}

We start by describing various properties of the three families of \voas\ that are involved in the inverse \qhr\ exploited in this paper.  The first is of course the simple \bp\ algebras $\sbpvoak$, reviewed in \cref{sec:bpminmod}.  Of particular importance throughout is $\sbpvoak$ when $\kk$ is a nondegenerate admissible level, denoted by $\bpminmoduv$.  After introducing spectral flow automorphisms and appropriate categories of $\bpminmoduv$-modules, we recall the classification results of \cite{FehCla20} and detail the structure of the spectral flow orbits of the \hw\ $\bpminmoduv$-modules.

\cref{sec:voaemb} is devoted to the other two \voa\ families.  We begin, in \cref{sec:w3voa}, with an account of the representation theory of the $\Wthree$ minimal model vertex operator algebra $\Wminmoduv$.  As $\Wminmoduv$ is rational \cite{AraRat15}, it has finitely many simple modules and all are \hw.  The final \voa\ needed is the half-lattice vertex algebra $\Pi$ described in \cref{sec:halflatvoa}.  There, we quickly review the construction of this vertex algebra, before choosing a conformal structure and defining certain ``relaxed'' $\Pi$-modules that will prove crucial for inverse \qhr.

This \lcnamecref{sec:voaemb} concludes by summarising the relationships between $\bpminmoduv$, $\Wminmoduv$ and $\Pi$, as well as their modules.  In particular, for any nondegenerate admissible level, there exists an embedding $\bpminmoduv \hookrightarrow \Wminmoduv \otimes \Pi$ (\cref{thm:embedding}).  Moreover, \cref{prop:identification} explains how to construct every simple relaxed $\bpminmoduv$-module, as classified in \cite{FehCla20}, as tensor products of $\Wminmoduv$- and relaxed $\Pi$-modules.  These results are due to \cite{AdaRea20}.

With this representation-theoretic review in hand, we commence our modularity study in \cref{sec:char}.  The results fit perfectly within the framework of the standard module formalism of \cite{CreLog13,RidVer14} with spectral flows of relaxed $\bpminmoduv$-modules playing the role of the standard modules.  A convenient technical step taken here is to modify the conformal structure of $\bpminmoduv$ so that one can avoid having to compute with twisted modules.  With this done, \cref{sec:stchar} describes how to compute the characters of standard $\bpminmoduv$-modules.  These are upgraded to linearly independent one-point functions in \cref{sec:stopfs}.  The modular S-matrix for the standard one-point functions is finally computed in \cref{sec:stmod}.

The standard module formalism also details how to extend this modularity to the simple \hw\ $\bpminmoduv$-modules.  However, the details turn out to be quite involved.  To minimise these complications, we temporarily restrict to minimal models with $\vv=3$ in \cref{sec:bpu3}.  These models nevertheless exemplify the general structure and, subject to \cref{conj:SVF} (the standard Verlinde formula for nonrational \voas), the Grothendieck fusion rules of all simple weight modules are computed (Theorem \ref{thm:FRstxstv=3}).  We conclude by identifying the simple currents of $\bpminmod{\uu}{3}$.  \cref{sec:bpu3ex} illustrates the general results for $\bpminmod{4}{3}$ and $\bpminmod{5}{3}$.

Finally, \cref{sec:bpuv} is devoted to attacking the general $\bpminmod{\uu}{\vv}$ minimal models.  \cref{sec:bpres} sets up the resolutions and character formulae for all \hw\ $\bpminmod{\uu}{\vv}$-modules and the modular S-matrix for the simplest class of these is obtained in \cref{thm:type3modularity}.  The standard Grothendieck fusion rules are then computed in \cref{sec:bpfus} and simple currents are identified.  All these calculations are quite involved and several necessary facts about $\Wminmoduv$ S-matrices and fusion coefficients are recalled (and derived) in \cref{sec:W3data}.  Finally, these general results are illustrated with the example $\bpminmod{3}{4}$ in \cref{sec:bpex}.

\addtocontents{toc}{\SkipTocEntry}
\subsection*{Acknowledgements}

ZF's research is supported by an Australian Government Research Training Program (RTP) Scholarship. DR's research is supported by the Australian Research Council Discovery Projects DP160101520 and DP210101502, as well as an Australian Research Council Future Fellowship FT200100431.

\section{\bp\ minimal models} \label{sec:bpminmod}

The level-$\kk$ universal \bp\ algebra $\ubpvoak$, $\kk \ne -3$, is the subregular (and minimal) W-algebra obtained from the universal level-$\kk$ affine \voa\ $\uslvoak$ by \qhr\ \cite{PolGau90,BerCon91}.  It is not simple if and only if $\kk \in \CC$ satisfies \cite{GorSim07}
\begin{equation} \label{eq:fraclevels}
	\kk +3 = \fracuv, \quad \text{for some}\ \uu \in \ZZ_{\ge2}, \vv \in \ZZ_{\ge1}\ \text{and}\ \gcd \set{\uu,\vv} = 1.
\end{equation}
For such $\kk$, we shall denote the simple quotient of $\ubpvoak$ by $\bpminmoduv$ and will refer to it as a \emph{\bp\ minimal model} \voa.  If $\uu \ge 3$ in \eqref{eq:fraclevels}, then $\kk$ is said to be \emph{admissible}.  If, in addition, $\vv \ge 3$, then $\kk$ is \emph{nondegenerate-admissible}.  In this \lcnamecref{sec:bpminmod}, we recall the representation theory of the nondegenerate-admissible-level minimal models $\bpminmoduv$, following \cite{FehCla20}.

\subsection{\bp\ \voas} \label{sec:bpvoa}

We begin with the well known presentation of the universal \bp\ algebra $\ubpvoak$ \cite{PolGau90,BerCon91}.
\begin{definition}
	For each level $\kk \ne -3$, the \emph{universal Bershadsky-Polyakov \voa} $\ubpvoak$ is the vertex algebra strongly and freely generated by fields $J(z)$, $G^\pm(z)$ and $L(z)$ with identity field $\wun$ and the following \opes:
	\begin{equation} \label{ope:bp}
		\begin{gathered}
	    L(z)L(w) \sim -\frac{(2\kk+3)(3\kk+1) \wun}{2 (\kk+3) (z-w)^4} + \frac{2L(w)}{(z-w)^2} + \frac{\partial L(w)}{(z-w)}, \\
	    L(z)J(w) \sim \frac{J(w)}{(z-w)^2} + \frac{\partial J(w)}{(z-w)}, \qquad
	    L(z)G^\pm(w) \sim \frac{\frac{3}{2}G^\pm(w)}{(z-w)^2} + \frac{\partial G^\pm(w)}{(z-w)},  \\
	    J(z)J(w) \sim  \frac{(2\kk+3) \, \wun}{3(z-w)^2}, \qquad
	    J(z)G^\pm(w) \sim  \frac{\pm G^\pm (w)}{(z-w)}, \qquad
	    G^\pm(z)G^\pm(w) \sim 0, \\
			G^+(z)G^-(w) \sim \frac{(\kk+1)(2\kk+3) \wun}{(z-w)^3} + \frac{3(\kk+1) J(w)}{(z-w)^2}
				+ \frac{3 \no{JJ}(w) + \frac{3}{2} (\kk+1) \partial J(w) - (\kk+3) L(w)}{z-w}.
		\end{gathered}
	\end{equation}
\end{definition}
\noindent For later use, it will be convenient to introduce the following reparametrisation of the level:
\begin{equation} \label{eq:defkappa}
	\kappa = \frac{2\kk+3}{6}.
\end{equation}
From \eqref{ope:bp}, the central charge of the minimal model \voa\ $\bpminmoduv$ is given by
\begin{equation} \label{eq:bpcc}
	\bpccuv = -\frac{(2\kk+3)(3\kk+1)}{\kk+3} = -\frac{(2\uu-3\vv)(3\uu-8\vv)}{\uu\vv} = 1 - \frac{6(\uu-2\vv)^2}{\uu\vv}.
\end{equation}
Arakawa has proven that the minimal models $\bpminmod{\uu}{2}$, with $\uu \ge 3$, are rational and $C_2$-cofinite \cite{AraAss15,AraRat13}.

The \emt\ $L(z)$ is expanded into modes in the usual way: $L(z) = \sum_{n \in \ZZ} L_n z^{-n-2}$.  In general, we shall expand the homogeneous fields of $\bpminmoduv$ as follows:
\begin{equation}
	A(z) = \sum_{n \in \ZZ - \Delta_A + \eps_A} A_n z^{-n-\Delta_A}.
\end{equation}
Here, $\Delta_A$ is the conformal weight ($L_0$-eigenvalue) of $A(z)$ and $\eps_A = \frac{1}{2}$, if $\Delta_A \in \ZZ+\frac{1}{2}$ and $A(z)$ is acting on a twisted $\ubpvoak$-module, and $\eps_A = 0$ otherwise.  Note that \eqref{ope:bp} specifies $\Delta_J = 1$ and $\Delta_{G^+} = \Delta_{G^-} = \frac{3}{2}$.

Conjugation is an automorphism of the \voa\ $\bpminmoduv$, defined on the modes of the generating fields $J(z)$, $L(z)$ and $G^\pm(z)$ by
\begin{equation} \label{eq:defconj}
	\conjmod{J_n} = -J_n, \quad \conjmod{G^+_n} = G^-_n, \quad \conjmod{G^-_n} = -G^+_n, \quad \conjmod{L_n} = L_n.
\end{equation}
An even more important family of vertex algebra automorphisms of $\bpminmoduv$ is \emph{spectral flow} $\sfsymb^\ell$, $\ell \in \ZZ$, which acts on the generators' modes as
\begin{equation} \label{eq:defsf}
  \sfmod{\ell}{J_n} = J_n - 2 \ell \kappa \delta_{n,0} \wun, \quad \sfmod{\ell}{G^\pm_n}
  = G^\pm_{n\mp \ell}, \quad \sfmod{\ell}{L_n} = L_n - \ell J_n + \ell^2 \kappa \delta_{n,0} \wun.
\end{equation}
Note that this is not a vertex operator algebra automorphism for all $\ell \ne 0$ as it does not preserve $L(z)$.

As usual, twisting the $\bpminmoduv$-action on modules by these automorphisms gives autoequivalences, which we shall also denote by $\conjsymb$ and $\sfsymb^\ell$, on the category $\wcatuv$ of weight $\bpminmoduv$-modules with \fdim\ weight spaces and its twisted version $\twcatuv$.  Moreover, we can extend $\ell$ to $\ZZ+\frac{1}{2}$ so as to obtain spectral flow equivalences between $\wcatuv$ and $\twcatuv$.  For more details, we refer to \cite{FehCla20}.

\subsection{\bp\ weight modules} \label{sec:bpsimple}

It is useful to distinguish certain classes of $\bpminmoduv$-modules in $\wcatuv$, in particular the \hw\ and \rhw\ ones.  We recall the definitions for completeness.
\begin{definition} \label{def:weight}
	\leavevmode
	\begin{itemize}
		\item A vector $v$ in a twisted or untwisted $\bpminmoduv$-module $\Mod{M}$ is a \emph{weight vector} of \emph{weight} $(j,\Delta)$ if it is a simultaneous eigenvector of $J_0$ and $L_0$ with eigenvalues $j$ and $\Delta$, respectively.  The nonzero simultaneous eigenspaces of $J_0$ and $L_0$ are called the \emph{weight spaces} of $\Mod{M}$.  If $\Mod{M}$ has a basis of weight vectors and each weight space is \fdim, then $\Mod{M}$ is a \emph{weight module}.
		\item A vector in an untwisted $\bpminmoduv$-module is a \emph{\hwv} if it is a simultaneous eigenvector of $J_0$ and $L_0$ that is annihilated by all modes with positive index.  An untwisted $\bpminmoduv$-module generated by a single \hwv\ is called an \emph{untwisted \hwm}.
		\item A vector in a twisted $\bpminmoduv$-module is a \emph{\hwv} if it is a simultaneous eigenvector of $J_0$ and $L_0$ that is annihilated by $G^+_0$ and all modes with positive index.  A twisted $\bpminmoduv$-module generated by a single \hwv\ is called a \emph{twisted \hwm}.
		\item A vector in a twisted or untwisted $\bpminmoduv$-module is a \emph{\rhwv} if it is a simultaneous eigenvector of $J_0$ and $L_0$ that is annihilated by all modes with positive index.  A $\bpminmoduv$-module generated by a single \rhwv\ is called a \emph{\rhwm}.
	\end{itemize}
\end{definition}

Let $\kk$ be nondegenerate-admissible.  Then, we conjecture that the simple objects of the categories $\wcatuv$ and $\twcatuv$ are all spectral flows of simple \rhwms.  We also believe that these are the simple objects of the physically relevant category from which level-$\kk$ \bp\ minimal model \cfts\ may be constructed.  For these reasons, we shall restrict attention to \rhw\ $\bpminmoduv$-modules in what follows.

The classification of simple twisted and untwisted \rhw\ $\bpminmoduv$-modules was recently obtained for nondegenerate admissible levels in \cite{FehCla20} (for $\bpminmod{\uu}{2}$, $\uu \ge 3$, this classification was previously obtained in \cite{AraRat13}). Let:
\begin{itemize}
	\item $\survuv$ be the set of $\aslthree$-weights $\lambda = \lambda^I - \fracuv \lambda^F$ satisfying $\lambda^I \in \pwlat{\uu-3}$, $\lambda^F \in \pwlat{\vv-1}$ and $\lambda^F_0 \ne 0$.
\end{itemize}
Here, $\pwlat{\ell}$ denotes the dominant integral weights of $\aslthree$ whose level is $\ell$ and $[\mu_0,\mu_1,\mu_2]$ denotes the Dynkin labels of an $\aslthree$-weight $\mu$.  We note that $\survuv$ is nonempty as $\kk$ is nondegenerate-admissible.  (In fact, it would remain nonempty if we allowed $\uu\ge3$ and $\vv=2$.)  Let:
\begin{itemize}
	\item $\infwtsuv$ be the subset of $\lambda \in \survuv$ consisting of weights satisfying $\lambda^F_1 \ne 0$.
\end{itemize}
We note that $\infwtsuv$ is nonempty because $\vv\ge3$ ($\kk$ is nondegenerate-admissible).

Observe that $\infwtsuv$ admits a free $\ZZ_3$-action $\outaut$ given, at the level of the Dynkin labels of $\lambda$, by
\begin{equation} \label{eq:Z3action}
	[\lambda_0, \lambda_1, \lambda_2]
	\overset{\outaut}{\longmapsto} [\lambda_2 - \tfracuv,\lambda_0,\lambda_1 + \tfracuv]
	\overset{\outaut}{\longmapsto} [\lambda_1, \lambda_2 - \tfracuv, \lambda_0 + \tfracuv]
	\overset{\outaut}{\longmapsto} [\lambda_0, \lambda_1, \lambda_2].
\end{equation}
Given $\lambda \in \infwtsuv$, let the Dynkin labels of $\lambda^I \in \pwlat{\uu-3}$ be $\rrr = [r_0,r_1,r_2]$.  Let $\fwt{i}$, $i=0,1,2$, denote the fundamental weights of $\aslthree$ and let the Dynkin labels of $\modify{\lambda}^F = \lambda^F - \fwt{0} - \fwt{1} \in \pwlat{\vv-3}$ be $\sss = [s_0,s_1,s_2]$.  In other words, let
\begin{equation} \label{eq:definers}
	r_0 = \lambda^I_0, \quad r_1 = \lambda^I_1, \quad r_2 = \lambda^I_2 \qquad \text{and} \qquad
	s_0 = \lambda^F_0 - 1, \quad s_1 = \lambda^F_1 - 1, \quad s_2 = \lambda^F_2.
\end{equation}
Then, the $\ZZ_3$-action \eqref{eq:Z3action} becomes the cycle
\begin{equation} \label{eq:Z3action'}
	\rsparrs \overset{\outaut}{\longmapsto} \rspar{r_2&r_0&r_1\\s_2&s_0&s_1} \overset{\outaut}{\longmapsto}
	\rspar{r_1&r_2&r_0\\s_1&s_2&s_0} \overset{\outaut}{\longmapsto} \rsparrs.
\end{equation}

We shall therefore frequently parametrise weights $\lambda \in \infwtsuv$ by $\rrr$ and $\sss$, or by the labels $r_i$ and $s_i$, $i=0,1,2$:
\begin{equation} \label{eq:rspar}
	\lambda = \wtnewparrs = \wtparrs = \sum_{i=0}^2 r_i \fwt{i} - \tfracuv \brac[\Big]{\fwt{0} + \fwt{1} + \sum_{i=0}^2 s_i \fwt{i}}.
\end{equation}
Extending this parametrisation to $\survuv$ means extending the allowed range of $s_0$, $s_1$ and $s_2$ to include $\vv-2$, $-1$ and $\vv-2$, respectively (but still subject to $s_0 + s_1 + s_2 = \vv-3$).

The main classification results of \cite{FehCla20} are summarised in the following two \lcnamecrefs{thm:classuntwi}.
\begin{theorem}[{\cite[Thm.~4.9]{FehCla20}}] \label{thm:classuntwi}
	For $\kk$ nondegenerate-admissible, the simple untwisted relaxed highest-weight $\bpminmoduv$-modules are, up to isomorphism, the highest-weight modules $\ihw{\lambda}$, $\lambda \in \survuv$, whose highest weights $(j,\Delta)$ are given by
	\begin{equation} \label{eq:bphwdata}
		j(\lambda) = \frac{\lambda_1 - \lambda_2}{3} \quad \text{and} \quad
		\Delta(\lambda) = \frac{(\lambda_1-\lambda_2)^2 - 3(\lambda_1+\lambda_2) \brac[\big]{2(\kk+1)-\lambda_1-\lambda_2}}{12(\kk+3)}.
	\end{equation}
	These modules are all pairwise nonisomorphic.
\end{theorem}
Define the \emph{top space} of a twisted (untwisted) $\bpminmoduv$-module to be the subspace spanned by the states of minimal conformal weight.  If the set of $J_0$-eigenvalues of the top space coincides with a single coset of $\CC/\ZZ$, then we shall refer to the twisted $\bpminmoduv$-module as being \emph{top-dense}.
\begin{theorem}[{\cite[Thms.~4.9 and 4.20]{FehCla20}}] \label{thm:classtwi}
	For $\kk$ nondegenerate-admissible, the simple twisted relaxed highest-weight $\bpminmoduv$-modules are, up to isomorphism:
	\begin{itemize}
    \item The \hwms\ $\twihw{\lambda} \cong \sfmod{1/2}{\ihw{\lambda}}$, $\lambda \in \survuv$, whose highest weights $(j^\twist,\Delta^\twist)$ are given by
    \begin{equation} \label{eq:bptwhwdata}
			j^\twist(\lambda) = j(\lambda) + \kappa \quad \text{and} \quad
			\Delta^\twist(\lambda) = \Delta(\lambda) + \frac{\lambda_1 - \lambda_2}{6} + \frac{\kappa}{4}.
		\end{equation}
		Such a module has an infinite-dimensional top space if and only if $\lambda \in \infwtsuv$.
		\item The conjugates $\conjmod{\twihw{j,\Delta}}$ of the \hwms\ with \infdim\ top spaces, hence $\lambda \in \infwtsuv$.
		\item The top-dense modules $\twrhw{[j],[\lambda]}$, where $[\lambda] \in \infwtsuv / \ZZ_3$ and $[j] \in (\CC/\ZZ) \setminus \set[\big]{[j^\twist(\outaut^i(\lambda))] \st i \in \ZZ_3}$.  The set of $J_0$-eigenvalues of $\twrhw{[j],[\lambda]}$ coincides with $[j]$ whilst the conformal weight of its top space is $\Delta^\twist(\lambda)$.
	\end{itemize}
\end{theorem}
\noindent For $\uu\ge3$, the simple twisted and untwisted $\bpminmod{\uu}{2}$-modules are all \hw, consistent with the fact that these \voas\ are rational \cite{AraRat13}.  We remark that the conjugate of a twisted \hw\ $\bpminmoduv$-module with a \fdim\ top space is again \hw.

Each family of simple top-dense \rhw\ $\bpminmoduv$-modules, corresponding to a fixed $[\lambda] \in \infwtsuv / \ZZ^3$ and parametrised by $[j] \in \CC/\ZZ$, has three ``gaps'' corresponding to the $[j^\twist(\outaut^i(\lambda))]$, $i \in \ZZ^3$.  It was shown in \cite[Thm.~4.24]{FehCla20} that these gaps in fact also correspond to top-dense $\bpminmoduv$-modules, albeit nonsimple ones.  Each of these ``gap modules'' may be taken to be indecomposable, with two possible choices related through contragredient duals.  Alternatively, the choice is unique if one insists on semisimplicity.

As we will be concerned with the modular properties of the characters of these twisted $\bpminmoduv$-modules, it does not matter which choice we make for the gap modules.  For later convenience, we shall choose them to be indecomposable with a twisted \hw\ submodule; equivalently, so that $G^-_0$ acts injectively on them.  They will be denoted using the same notation $\twrhw{[j],[\lambda]}$ as their simple cousins, where $[j] = [j^\twist(\outaut^i(\lambda))]$, $i \in \ZZ_3$.

To streamline notation here and below, we shall also frequently write $\twrhw{\lambda}$ instead of $\twrhw{[j^\twist(\lambda)],[\lambda]}$ for these nonsemisimple ``gap modules''.  Note that this notation breaks the $\outaut$-orbit symmetry for the nonsimple top-dense modules: $\twrhw{\lambda} \cong \twrhw{\mu}$ if and only if $\lambda = \mu$ in $\infwtsuv$.  Another convenient alternative notation for what follows is writing
\begin{equation} \label{eq:defihwpar}
	\ihw{\lambda} = \ihwnewparrs = \ihwparrs  \quad \text{and} \quad
	\twihw{\lambda} = \twihwnewparrs = \twihwparrs, \quad \text{when}\ \lambda = \wtnewparrs = \wtparrs.
\end{equation}
We shall likewise write $\twrhw{[j],[\lambda]} = \twrhwnewparrs{[j]} = \twrhwparrs{[j]}$ and $\twrhw{\lambda} = \twrhwnewparrs{} = \twrhwparrs{}$ when convenient.

With this notation, the structure of the gap modules may be summarised as follows.
\begin{proposition}[{\cite[Thm.~4.24]{FehCla20}}] \label{prop:gapdecomp}
	Let $\kk$ be nondegenerate-admissible and let $\wtnewparrs \in \infwtsuv$ (so $\rrr \in \pwlat{\uu-3}$ and $\sss \in \pwlat{\vv-3}$).  Then, the following sequence is exact and nonsplit:
	\begin{equation}
		\dses{\twihwparrs}{}{\twrhwparrs{}}{}{\conjmod{\twihwpar{r_0 & r_2 & r_1 \\ s_0 & s_2 & s_1}}}.
	\end{equation}
\end{proposition}

\subsection{Spectral flow orbits} \label{sec:bpspecflow}

Given any $\bpminmoduv$-module $\Mod{M} \in \wcatuv$, its spectral flow $\sfmod{\ell}{\Mod{M}}$ is another $\bpminmoduv$-module in either $\wcatuv$ or $\twcatuv$, depending on whether $\ell \in \ZZ$ or $\ZZ + \frac{1}{2}$, respectively.  Consider therefore the orbit, under spectral flow, of a fixed \hw\ $\bpminmoduv$-module $\ihw{\lambda}$.  Almost all of the (twisted) modules in this orbit will fail to be positive-energy, meaning that the conformal weights of their states will be unbounded below.  Those that are positive-energy will be \hw\ or conjugate \hw.  We will find it useful to distinguish spectral flow orbits according to how many (twisted) \hwms\ it contains.
\begin{proposition}[{\cite[Thm.~4.15]{FehCla20}}] \label{whenarespecflowshw}
	Let $\kk$ be nondegenerate-admissible and take $\rrr$ and $\sss$ so that $\wtnewparrs \in \survuv$.  Then:
	\begin{itemize}
		\item $\sfmod{}{\ihwnewparrs}$ is \hw\ if and only if $s_1=-1$, in which case $\sfmod{}{\ihwpar{r_0&r_1&r_2\\s_0&-1&s_2}} \cong \ihwpar{r_2&r_0&r_1\\s_2&s_0-1&0}$.
		\item $\sfmod{-1}{\ihwnewparrs}$ is \hw\ if and only if $s_2=0$, in which case $\sfmod{-1}{\ihwpar{r_0&r_1&r_2\\s_0&s_1&0}} \cong \ihwpar{r_1&r_2&r_0\\s_1+1&-1&s_0}$.
		\item $\sfmod{2}{\ihwnewparrs}$ is \hw\ if and only if $\sss=[0,-1,\vv-2]$, in which case $\sfmod{2}{\ihwpar{r_0&r_1&r_2\\0&-1&\vv-2}} \cong \ihwpar{r_1&r_2&r_0\\0&\vv-3&0}$.
		\item $\sfmod{-2}{\ihwnewparrs}$ is \hw\ if and only if $\sss=[0,\vv-3,0]$, in which case $\sfmod{-2}{\ihwpar{r_0&r_1&r_2\\0&\vv-3&0}} \cong \ihwpar{r_2&r_0&r_1\\0&-1&\vv-2}$.
		\item For $\abs{\ell} \in \ZZ_{\ge3}$, $\sfmod{\ell}{\ihwparrs}$ is never \hw\ (since $\vv \ge 3$).
	\end{itemize}
\end{proposition}
\noindent As $\twihw{\lambda} \cong \sfmod{1/2}{\ihw{\lambda}}$ (\cref{thm:classtwi}), the results of \cref{whenarespecflowshw} remain valid when $\ihw{}$ is replaced throughout by $\twihw{}$.

It follows from \cref{whenarespecflowshw} that, for $\kk$ nondegenerate-admissible, the spectral flow orbit of a simple \hw\ $\bpminmoduv$-module always contains exactly one simple twisted \hwm\ with an \infdim\ top space and exactly one simple twisted conjugate \hwm\ with an \infdim\ top space.
\begin{definition} \label{def:orbittypes} 
	Let $\kk$ be nondegenerate-admissible.  We say that $\lambda \in \survuv$ is \emph{type-$n$} whenever the spectral flow orbit $\set*{\sfmod{\ell}{\ihw{\lambda}} \st \ell \in \frac{1}{2} \ZZ}$ contains precisely $n$ \hw\ $\bpminmoduv$-modules.  In this case, we shall also refer to the spectral flow orbit of $\ihw{\lambda}$, as well as any twisted or untwisted module isomorphic to one in the orbit, as being of type-$n$.
\end{definition}
\noindent Of course, a type-$n$ spectral flow orbit also contains $n$ twisted \hw\ $\bpminmoduv$-modules, only one of which has an \infdim\ top space.

\begin{corollary} \label{cor:typereps}
	Let $\kk$ be nondegenerate-admissible.  Then, every type-$n$ module is isomorphic to a unique $\bpminmoduv$-module of the form $\sfmod{\ell}{\ihwnewparrs}$, for some $\ell \in \frac{1}{2} \ZZ$, where $\wtnewparrs \in \survuv$ satisfies one of the following conditions:
	\begin{center}
		\begin{tabular}{C|C|C}
			n=1 & n=2 & n=3 \\
			\hline
			s_1 \ne -1\ \text{and}\ s_2 \ne 0 &
			s_1 = -1\ \text{and}\ s_2 \ne 0, \vv-2 &
			s_1 = -1\ \text{and}\ s_2 = \vv-2
		\end{tabular}
		.
	\end{center}
\end{corollary}
\noindent We visualise the type-$n$ spectral flow orbits in \cref{fig:sforbits}.  The representatives chosen in \cref{cor:typereps} are the leftmost for each type in this \lcnamecref{fig:sforbits}.
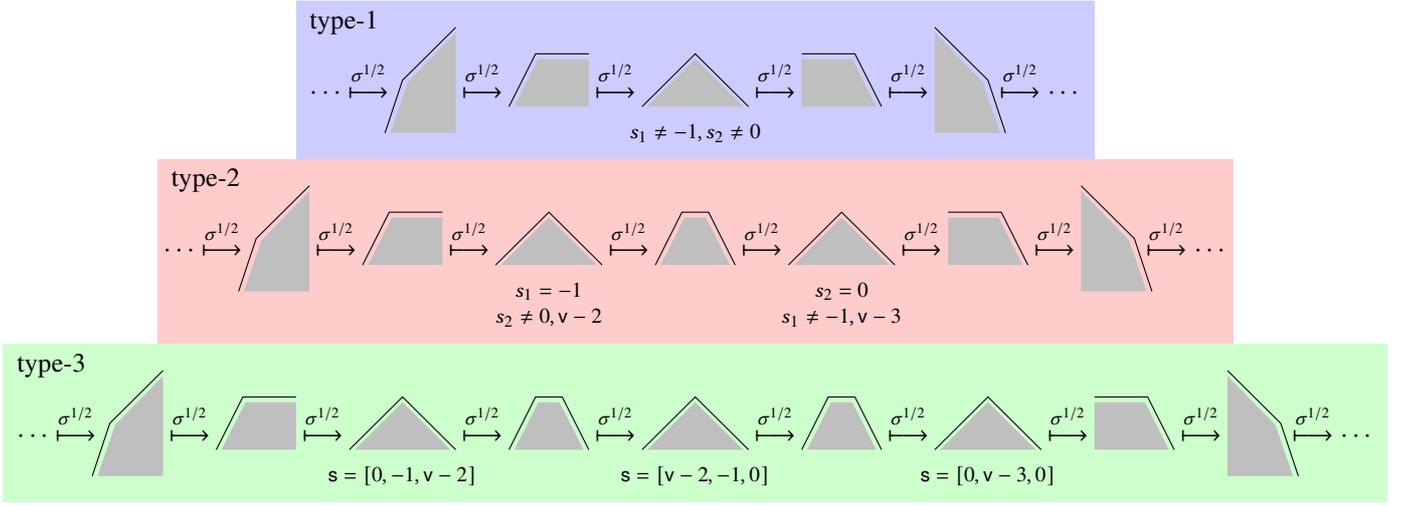
\begin{figure}
	\makebox[\textwidth]{
		\begin{tikzpicture}[scale=0.7]
		\begin{scope}
		\fill[blue!20] (7.5,2) rectangle (-7.5,-1);
		\end{scope}
		\begin{scope} [shift={(0,-3)}]
		\fill[red!20] (10.1,2) rectangle (-10.1,-1.5);
		\end{scope}
		\begin{scope} [shift={(0,-6.5)}]
		\fill[green!20] (13,2) rectangle (-13,-1);
		\end{scope}
			\begin{scope}[shift={(-6.5,0)}]
				\node at (0,0.5) {$\cdots \lma$};
			\end{scope}
			\begin{scope}[shift={(-5.5,0)}]
				\draw (1,1.5) -- (0,0.5) -- (-0.33,-0.5);
				\fill[lightgray] (1,1.4) -- (0.1,0.47) -- (-0.23,-0.5) -- (1,-0.5);
			\end{scope}
			\begin{scope}[shift={(-4,0)}]
				\node at (0,0.5) {$\lma$};
			\end{scope}
			\begin{scope}[shift={(-3,0)}]
				\draw (1,1) -- (0,1) -- (-0.5,0);
				\fill[lightgray] (1,0.9) -- (0.05,0.9) -- (-0.4,0) -- (1,0);
			\end{scope}
			\begin{scope}[shift={(-1.5,0)}]
				\node at (0,0.5) {$\lma$};
			\end{scope}
			\begin{scope}[shift={(0,0)}]
				\draw (-1,0) -- (0,1) -- (1,0);
				\fill[lightgray] (-0.9,0) -- (0,0.9) -- (0.9,0);
				\node at (0,-0.5) {\footnotesize $s_1 \ne -1, s_2 \ne 0$};
			\end{scope}
			\begin{scope}[shift={(1.5,0)}]
				\node at (0,0.5) {$\lma$};
			\end{scope}
			\begin{scope}[shift={(3,0)}]
				\draw (-1,1) -- (0,1) -- (0.5,0);
				\fill[lightgray] (-1,0.9) -- (-0.05,0.9) -- (0.4,0) -- (-1,0);
				\node at (-0.3,0.5) {};
			\end{scope}
			\begin{scope}[shift={(4,0)}]
				\node at (0,0.5) {$\lma$};
			\end{scope}
			\begin{scope}[shift={(5.5,0)}]
				\draw (-1,1.5) -- (0,0.5) -- (0.33,-0.5);
				\fill[lightgray] (-1,1.4) -- (-0.1,0.47) -- (0.23,-0.5) -- (-1,-0.5);
			\end{scope}
			\begin{scope}[shift={(6.5,0)}]
				\node at (0,0.5) {$\lma \cdots$};
			\end{scope}
			\begin{scope}[shift={(-2.75,-3)}]
				\begin{scope}[shift={(-6.5,0)}]
					\node at (0,0.5) {$\cdots \lma$};
				\end{scope}
				\begin{scope}[shift={(-5.5,0)}]
					\draw (1,1.5) -- (0,0.5) -- (-0.33,-0.5);
					\fill[lightgray] (1,1.4) -- (0.1,0.47) -- (-0.23,-0.5) -- (1,-0.5);
				\end{scope}
				\begin{scope}[shift={(-4,0)}]
					\node at (0,0.5) {$\lma$};
				\end{scope}
				\begin{scope}[shift={(-3,0)}]
					\draw (1,1) -- (0,1) -- (-0.5,0);
					\fill[lightgray] (1,0.9) -- (0.05,0.9) -- (-0.4,0) -- (1,0);
				\end{scope}
				\begin{scope}[shift={(-1.5,0)}]
					\node at (0,0.5) {$\lma$};
				\end{scope}
				\begin{scope}[shift={(0,0)}]
					\draw (-1,0) -- (0,1) -- (1,0);
					\fill[lightgray] (-0.9,0) -- (0,0.9) -- (0.9,0);
					\node at (0,-0.75) {\footnotesize $\begin{matrix} s_1=-1 \\ s_2 \ne 0, \vv-2 \end{matrix}$};
				\end{scope}
				\begin{scope}[shift={(1.5,0)}]
					\node at (0,0.5) {$\lma$};
				\end{scope}
				\begin{scope}[shift={(2.75,0)}]
					\draw (-0.75,0) -- (-0.25,1) -- (0.25,1) -- (0.75,0);
					\fill[lightgray] (-0.65,0) -- (-0.2,0.9) -- (0.2,0.9) -- (0.65,0);
					\node at (0,0.5) {};
				\end{scope}
				\begin{scope}[shift={(4,0)}]
					\node at (0,0.5) {$\lma$};
				\end{scope}
				\begin{scope}[shift={(5.5,0)}]
					\draw (-1,0) -- (0,1) -- (1,0);
					\fill[lightgray] (-0.9,0) -- (0,0.9) -- (0.9,0);
					\node at (0,-0.75) {\footnotesize $\begin{matrix} s_2=0 \\ s_1 \ne -1, \vv-3 \end{matrix}$};
				\end{scope}
				\begin{scope}[shift={(7,0)}]
					\node at (0,0.5) {$\lma$};
				\end{scope}
				\begin{scope}[shift={(8.5,0)}]
					\draw (-1,1) -- (0,1) -- (0.5,0);
					\fill[lightgray] (-1,0.9) -- (-0.05,0.9) -- (0.4,0) -- (-1,0);
					\node at (-0.3,0.5) {};
				\end{scope}
				\begin{scope}[shift={(9.5,0)}]
					\node at (0,0.5) {$\lma$};
				\end{scope}
				\begin{scope}[shift={(11,0)}]
					\draw (-1,1.5) -- (0,0.5) -- (0.33,-0.5);
					\fill[lightgray] (-1,1.4) -- (-0.1,0.47) -- (0.23,-0.5) -- (-1,-0.5);
				\end{scope}
				\begin{scope}[shift={(12,0)}]
					\node at (0,0.5) {$\lma \cdots$};
				\end{scope}
			\end{scope}
			\begin{scope}[shift={(-5.5,-6.5)}]
				\begin{scope}[shift={(-6.5,0)}]
					\node at (0,0.5) {$\cdots \lma$};
				\end{scope}
				\begin{scope}[shift={(-5.5,0)}]
					\draw (1,1.5) -- (0,0.5) -- (-0.33,-0.5);
					\fill[lightgray] (1,1.4) -- (0.1,0.47) -- (-0.23,-0.5) -- (1,-0.5);
				\end{scope}
				\begin{scope}[shift={(-4,0)}]
					\node at (0,0.5) {$\lma$};
				\end{scope}
				\begin{scope}[shift={(-3,0)}]
					\draw (1,1) -- (0,1) -- (-0.5,0);
					\fill[lightgray] (1,0.9) -- (0.05,0.9) -- (-0.4,0) -- (1,0);
				\end{scope}
				\begin{scope}[shift={(-1.5,0)}]
					\node at (0,0.5) {$\lma$};
				\end{scope}
				\begin{scope}[shift={(0,0)}]
					\draw (-1,0) -- (0,1) -- (1,0);
					\fill[lightgray] (-0.9,0) -- (0,0.9) -- (0.9,0);
					\node at (0,-0.5) {\footnotesize $\sss = [0,-1,\vv-2]$};
				\end{scope}
				\begin{scope}[shift={(1.5,0)}]
					\node at (0,0.5) {$\lma$};
				\end{scope}
				\begin{scope}[shift={(2.75,0)}]
					\draw (-0.75,0) -- (-0.25,1) -- (0.25,1) -- (0.75,0);
					\fill[lightgray] (-0.65,0) -- (-0.2,0.9) -- (0.2,0.9) -- (0.65,0);
					\node at (-0.05,0.5) {};
				\end{scope}
				\begin{scope}[shift={(4,0)}]
					\node at (0,0.5) {$\lma$};
				\end{scope}
				\begin{scope}[shift={(5.5,0)}]
					\draw (-1,0) -- (0,1) -- (1,0);
					\fill[lightgray] (-0.9,0) -- (0,0.9) -- (0.9,0);
					\node at (0,-0.5) {\footnotesize $\sss = [\vv-2,-1,0]$};
				\end{scope}
				\begin{scope}[shift={(7,0)}]
					\node at (0,0.5) {$\lma$};
				\end{scope}
				\begin{scope}[shift={(8.25,0)}]
					\draw (-0.75,0) -- (-0.25,1) -- (0.25,1) -- (0.75,0);
					\fill[lightgray] (-0.65,0) -- (-0.2,0.9) -- (0.2,0.9) -- (0.65,0);
					\node at (0,0.5) {};
				\end{scope}
				\begin{scope}[shift={(9.5,0)}]
					\node at (0,0.5) {$\lma$};
				\end{scope}
				\begin{scope}[shift={(11,0)}]
					\draw (-1,0) -- (0,1) -- (1,0);
					\fill[lightgray] (-0.9,0) -- (0,0.9) -- (0.9,0);
					\node at (0,-0.5) {\footnotesize $\sss = [0,\vv-3,0]$};
				\end{scope}
				\begin{scope}[shift={(12.5,0)}]
					\node at (0,0.5) {$\lma$};
				\end{scope}
				\begin{scope}[shift={(14,0)}]
					\draw (-1,1) -- (0,1) -- (0.5,0);
					\fill[lightgray] (-1,0.9) -- (-0.05,0.9) -- (0.4,0) -- (-1,0);
					\node at (-0.3,0.5) {};
				\end{scope}
				\begin{scope}[shift={(15,0)}]
					\node at (0,0.5) {$\lma$};
				\end{scope}
				\begin{scope}[shift={(16.5,0)}]
					\draw (-1,1.5) -- (0,0.5) -- (0.33,-0.5);
					\fill[lightgray] (-1,1.4) -- (-0.1,0.47) -- (0.23,-0.5) -- (-1,-0.5);
				\end{scope}
				\begin{scope}[shift={(17.5,0)}]
					\node at (0,0.5) {$\lma \cdots$};
				\end{scope}
			\end{scope}
			\node at (-6.6,1.6) {\textcolor{black}{type-$1$}};
			\node at (-9.2,-1.4) {\textcolor{black}{type-$2$}};
			\node at (-12.1,-4.9) {\textcolor{black}{type-$3$}};
		\end{tikzpicture}
	}
\caption{A picture of the weights of the three types of spectral flow orbits through a simple \hw\ $\bpminmoduv$-module for $\kk$ nondegenerate-admissible.  The $J_0$-eigenvalue increases from left to right, whilst the $L_0$-eigenvalue increases from top to bottom.  The conditions stated for the $s$-labels constrain the highest weight $\lambda = \wtnewparrs \in \survuv$ of the corresponding untwisted module.} \label{fig:sforbits}
\end{figure}

Note that the vacuum module $\ihw{\kk \fwt{0}} = \ihwpar{\uu-3&0&0\\\vv-2&-1&0}$ is always an untwisted type-$3$ module.  In fact, when $\vv=3$, all the simple twisted and untwisted \hw\ $\bpminmoduv$-modules are type-$3$.  On the other hand, for $\vv>3$, there are $\bpminmoduv$-modules of every type.

We conclude with a brief study of spectral flows of conjugate \hw\ $\bpminmoduv$-modules, specifically those that appear in the short exact sequences of \cref{prop:gapdecomp}.
\begin{lemma}[{\cite[Prop.~4.13]{FehCla20}}] \label{lem:conjhw}
	Let $\kk$ be nondegenerate-admissible and choose $\rrr$ and $\sss$ so that $\wtnewparrs \in \survuv$.  Then,
	\begin{equation}
		\conjmod{\ihwparrs} \cong \ihwpar{r_0 & r_2 & r_1 \\ s_0 & s_2-1 & s_1+1} \quad \text{and} \quad
		\conjmod{\twihwpar{r_0 & r_1 & r_2 \\ s_0 & -1 & s_2}} \cong \twihwpar{r_2 & r_1 & r_0 \\ s_2 & -1 & s_0}.
	\end{equation}
\end{lemma}
\noindent We remark that if $s_1 \ne -1$, so that $\lambda = \wtparrs \in \infwtsuv$, then $\twihw{\lambda}$ has an \infdim\ top space.  Its conjugate is therefore not \hw.
\begin{proposition} \label{prop:ses}
	Let $\kk$ be nondegenerate-admissible and choose $\wtnewparrs \in \survuv$ leftmost in its orbit, as pictured in \cref{fig:sforbits}.  Then, we have the following nonsplit short exact sequence:
	\begin{equation}
		\dses{\sfmod{}{\ihwpar{r_0 & r_1 & r_2 \\ s_0 & s_1+1 & s_2-1}}}{}{\sfmod{1/2}{\twrhwpar{}{r_0 & r_1 & r_2 \\ s_0 & s_1+1 & s_2-1}}}{}{\ihwparrs}.
	\end{equation}
	Here, $\wtpar{r_0 & r_1 & r_2 \\ s_0 & s_1+1 & s_2-1} \in \infwtsuv$ is the rightmost in its orbit.  It is type-$n$ under the following conditions:
	\begin{center}
		\begin{tabular}{C|C|C}
			n=1 & n=2 & n=3 \\
			\hline
			s_2 \ne 1 &
			s_1 \ne \vv-4\ \text{and}\ s_2 = 1 &
			\sss_1 = [0,\vv-4,1]
		\end{tabular}
		.
	\end{center}
\end{proposition}
\begin{proof}
	We apply the exact functor $\sfsymb^{1/2}$ to the nonsplit short exact sequence of \cref{prop:gapdecomp} and compute that
	\begin{equation}
		\sfsymb^{1/2} \conjmod{\twihwpar{r_0&r_2&r_1\\s_0&s_2&s_1}}
		\cong \conjsymb \sfmod{-1/2}{\twihwpar{r_0&r_2&r_1\\s_0&s_2&s_1}}
		\cong \conjmod{\ihwpar{r_0&r_2&r_1\\s_0&s_2&s_1}}
		\cong \ihwpar{r_0&r_1&r_2\\s_0&s_1-1&s_2+1},
	\end{equation}
	using \cref{lem:conjhw}.  Shifting $s_1 \to s_1+1$ and $s_2 \to s_2-1$, the proof is completed by noting that $\wtpar{r_0&r_1&r_2\\s_0&s_1+1&s_2-1} \in \infwtsuv$, so it must be rightmost in its orbit (see \cref{fig:sforbits}).
\end{proof}

\section{Inverse \qhr\ for \bp\ algebras} \label{sec:voaemb}

The universal affine vertex operator algebra $\uslvoak$ has three nonisomorphic \qhrs\ corresponding to the three nilpotent orbits of $\slthree$: $\uslvoak$ itself, the \bp\ algebra $\ubpvoak$ and the regular W-algebra $\uWvoak$, which we shall refer to as the Zamolodchikov algebra.  When $\kk$ is nondegenerate-admissible, $\uWvoak$ is not simple \cite{MizDet89,WatDet89}.  In this case, the simple quotient shall be denoted by $\Wminmoduv$.

For these levels, there is a relationship \cite{AdaRea20} between the minimal models $\bpminmoduv$ and $\Wminmoduv$ that will be crucial for our modularity studies.  We consider this relationship to be an instance of a kind of inverse to \qhr\ \cite{SemInv94,AdaRea17}, though now this refers to inverting an as yet unformulated reduction from $\bpminmoduv$ and $\Wminmoduv$, in the spirit of the ``reduction by stages'' of \cite{MorQua15}.  In this \lcnamecref{sec:voaemb}, we review this relationship and some of its representation-theoretic consequences.

\subsection{$\Wthree$ minimal models} \label{sec:w3voa}

We begin with the Zamolodchikov algebras and their representation theories, when the level $\kk$ is nondegenerate-admissible.
\begin{definition} \label{def:W3}
	The \emph{universal Zamolodchikov algebra} $\uWvoak$ is the vertex algebra strongly and freely generated by fields $T(z)$ and $W(z)$ with the following \opes:
	\begin{equation} \label{eq:ope:W3}
		\begin{gathered}
	    T(z)T(w) \sim \frac{\Wcck \wun}{2(z-w)^4} + \frac{2T(w)}{(z-w)^2} + \frac{\partial T(w)}{(z-w)}, \qquad
	    T(z)W(w) \sim \frac{3 W(w)}{(z-w)^2} + \frac{\partial W(w)}{(z-w)}, \\
	    W(z)W(w) \sim \frac{2 \Lambda(w)}{(z-w)^2} + \frac{\partial \Lambda(w)}{(z-w)}
				+ A_{\kk} \sqbrac*{\frac{\Wcck \wun}{3(z-w)^6} + \frac{2 T(w)}{(z-w)^4} + \frac{\partial T(w)}{(z-w)^3}
				+ \frac{\frac{3}{10}\partial^2 T(w)}{(z-w)^2}  + \frac{\frac{1}{15}\partial^3 T(w)}{(z-w)}}.
		\end{gathered}
	\end{equation}
	Here, we set
	\begin{equation}
		\Wcck = -\frac{2(3\kk+5)(4\kk+9)}{\kk+3}, \quad
		\Lambda(z) = \no{T(z)T(z)} - \frac{3}{10}\partial^2 T(z) \quad \text{and} \quad
		A_{\kk} = -\frac{(3\kk+4)(5\kk+12)}{2(\kk+3)} = \frac{22+5\Wcck}{16}.
	\end{equation}
\end{definition}
\noindent We shall refer to the $\Wminmoduv$ as the $\Wthree$ minimal models, assuming that $\kk$ is nondegenerate-admissible.  These models are all rational and $C_2$-cofinite \cite{AraAss15,AraRat15}.  Note that the central charge is invariant under exchanging $\uu$ and $\vv$:
\begin{equation}
	\Wccuv = -\frac{2(3\uu-4\vv)(4\uu-3\vv)}{\uu\vv} = 2 - \frac{24(\uu-\vv)^2}{\uu\vv}.
\end{equation}
As the defining \opes\ \eqref{eq:ope:W3} only depend on $\kk$ through $\Wcck$, it follows that $\Wminmod{\uu}{\vv} = \Wminmod{\vv}{\uu}$.

We remark that we have employed a nonstandard normalisation for $W(z)$ in \cref{def:W3}, namely we have multiplied the standard definition of \cite{ZamInf85} by $\sqrt{A_{\kk}}$ in order to cancel the poles that arise when $\Wcck = -\frac{22}{5}$, hence $(\uu,\vv)=(3,5)$ or $(5,3)$.  In fact, $W$ and $\Lambda$ are null at this central charge, hence are zero in $\Wminmod{3}{5} = \Wminmod{5}{3}$.  In fact, the $\Wthree$ minimal model $\Wminmod{3}{5}$ coincides with the Virasoro minimal model $\virminmod{2}{5}$ of the same central charge.

The classification of simple $\Wminmoduv$-modules was obtained in \cite{FatCon87}.  These modules are \hw\ with one-dimensional top spaces.  Writing $T(z) = \sum_{n \in \ZZ} T_n z^{-n-2}$ and $W(z) = \sum_{n \in \ZZ} W_n z^{-n-3}$, a \hwv\ is then a simultaneous eigenvector of $T_0$ and $W_0$ that is annihilated by the $T_n$ and $W_n$ with $n>0$.  Here, we adapt the parametrisation of the highest weights given in \cite{BouWSym93}.

Recall from \cref{sec:bpsimple} that each $\lambda = \wtnewparrs \in \infwtsuv$ is specified by triples $\rrr = [r_0,r_1,r_2] \in \pwlat{\uu-3}$ and $\sss = [s_0,s_1,s_2] \in \pwlat{\vv-3}$.  Such a $\lambda$ also specifies a simple \hw\ $\Wminmoduv$-module and the eigenvalues of $T_0$ and $W_0$ on its \hwv\ are given by
\begin{subequations} \label{eq:W3eigs}
	\begin{align}
		\Delta_{\lambda} &= \Delta \rsnewparrs = \Delta \rsparrs \label{eq:W3Teig}
		= \frac{1}{3 \uu \vv} \biggl( \brac[\big]{\vv (r_1+1)-\uu (s_1+1)} \brac[\big]{\vv (r_2+1)-\uu (s_2+1)} \biggr. \\
		&\mspace{150mu} \biggl. + \brac[\big]{\vv (r_1+1)-\uu (s_1+1)}^2 + \brac[\big]{\vv (r_2+1)-\uu (s_2+1)}^2 - (\uu-\vv)^2 \biggr) \notag \\ \text{and} \quad
		w_{\lambda} &= w \rsnewparrs = w \rsparrs \label{eq:W3Weig}
		= \frac{\brac[\big]{\vv (r_0-r_1) - \uu (s_0-s_1)} \brac[\big]{\vv (r_0-r_2) - \uu (s_0-s_2)} \brac[\big]{\vv (r_1-r_2) - \uu (s_1-s_2)}}{3 (3 \uu \vv)^{3/2}},
	\end{align}
\end{subequations}
respectively.  As these eigenvalues are invariant under the free $\ZZ_3$-action \eqref{eq:Z3action'} defined by $\outaut$, the simple \hw\ $\Wminmoduv$-modules are actually parametrised by $\infwtsuv / \ZZ_3$ and so we shall denote them by $\Wihw{[\lambda]}$ or, if more convenient, by $\Wihwnewparrs$ or $\Wihwparrs$.

Similarly, the conformal weight \eqref{eq:W3Teig} is invariant under the (nonfree) $\ZZ_2$-action
\begin{equation} \label{eq:W3conj}
	\rsparrs \longleftrightarrow \rspar{r_0&r_2&r_1\\s_0&s_2&s_1},
\end{equation}
whilst \eqref{eq:W3Weig} changes sign.  This then corresponds to the conjugation automorphism, $T(z) \leftrightarrow T(z)$ and $W(z) \leftrightarrow -W(z)$, of $\Wminmoduv$.  We therefore get an additional isomorphism corresponding to \eqref{eq:W3conj} if $w_{\lambda} = 0$ (when $\Wihw{[\lambda]}$ is self-conjugate).  But, \eqref{eq:W3Weig} shows that this happens if and only if two of the pairs $(r_0,s_0)$, $(r_1,s_1)$ and $(r_2,s_2)$ coincide, in which case the conjugation isomorphism is already accounted for by one of the isomorphisms corresponding to the $\ZZ_3$-action \eqref{eq:Z3action'}.  We therefore conclude that the isomorphism classes of the simple $\Wminmoduv$-modules are classified by $\infwtsuv / \ZZ_3$.

The fact that the simple $\Wminmoduv$-modules and the families of ``top-dense'' $\bpminmoduv$-modules are parametrised in the same fashion suggests that there is a relationship between these modules.  The rest of this \lcnamecref{sec:voaemb} is devoted to reviewing this relationship, following \cite{AdaRea20}.

\subsection{The half-lattice vertex algebra} \label{sec:halflatvoa}

To describe the relationship between $\bpminmoduv$ and $\Wminmoduv$, we need to introduce a ``half-lattice'' \voa\ \cite{BerRep02}.  For this, we follow \cite[Sec.~3]{AdaRea20} except that our conventions require a different conformal structure.

Consider the abelian Lie algebra $\alg{h} = \text{span}_{\CC} \set{c,d}$, equipped with the symmetric bilinear form $\inner{\cdot\,}{\cdot}$ defined by
\begin{equation}
	\inner{c}{c} = \inner{d}{d} = 0 \quad \text{and} \quad \inner{c}{d} = 2.
\end{equation}
The group algebra $\CC[\ZZ c] = \text{span}_{\CC} \{e^{nc} \vert\ n \in \ZZ \}$ has the structure of an $\alg{h}$-module according to the formula
\begin{equation}
	h(e^{nc})= n\langle h, c \rangle e^{nc}.
\end{equation}
Denote by $\VOA{H}$ the Heisenberg vertex algebra defined by $\alg{h}$ and $\inner{\cdot\,}{\cdot}$.
\begin{definition} \label{def:halflattice}
The \emph{half lattice vertex algebra} $\halflattice$ is the lattice vertex algebra $\VOA{H} \otimes \CC[\ZZ c]$ where the action of $h \in \alg{h}$ on $\CC[\ZZ c]$ is identified with the action of the zero mode $h_0$ of $h(z) \in H$.
\end{definition}
\noindent A set of (strong) generating fields for $\halflattice$ is then $\set*{c(z), d(z), \ee^{mc}(z) \st m \in \ZZ}$.  The \opes\ of these fields are easily determined:
\begin{equation} \label{eq:ope:halflattice}
	\begin{aligned}
		c(z)c(w) &\sim 0, & c(z)d(w) &\sim \frac{2 \, \wun}{(z-w)^2}, & d(z)d(w) &\sim 0, \\
		c(z)\ee^{mc}(w) &\sim 0, & d(z)\ee^{mc}(w) &\sim \frac{2m \, \ee^{mc}(w)}{z-w}, & \ee^{mc}(z) \ee^{nc}(w) &\sim  0.
	\end{aligned}
\end{equation}
For what follows, we introduce a convenient orthogonal basis for the Heisenberg fields in $\halflattice$ given by
\begin{equation}
	a(z) = -\kappa c(z) + \frac{1}{2}d(z) \quad \text{and} \quad b(z) = \kappa c(z) + \frac{1}{2}d(z),
\end{equation}
where $\kappa$ was defined in \eqref{eq:defkappa}.  Note that $\inner{a}{a} = -2\kappa$ and $\inner{b}{b} = 2\kappa$.

This half lattice vertex algebra admits a two-parameter family of energy-momentum fields given by
\begin{equation}
	t(z) = \frac{1}{2} \no{c(z)d(z)} + \alpha \pd c(z) + \beta \pd d(z), \quad \alpha, \beta \in \CC;
\end{equation}
the corresponding central charge is $2-48\alpha \beta$.  We equip $\halflattice$ with the conformal structure given by $\alpha = -\frac{3}{2} \kappa$ and $\beta = \frac{3}{4}$, so that $t(z) = \frac{1}{2} \no{c(z)d(z)} + \frac{3}{2} \pd a(z)$.  At the nondegenerate admissible levels we are interested in, the central charge of $\halflattice$ now simplifies to
\begin{equation} \label{eq:ccaddup}
	\lccuv = 2 + 54 \kappa = -1 + \frac{6(3\uu-4\vv)}{\vv} \qquad \Rightarrow \qquad
	\bpccuv = \lccuv + \Wccuv.
\end{equation}
The latter identity is in fact the reason for choosing $t(z)$ as we did.  With respect to $t(z)$, both $a(z)$ and $b(z)$ have conformal weight $1$ (though $a$ is not quasiprimary) whilst that of $\ee^{mc}(z)$ is $-\frac{3m}{2}$.

We are interested in the positive-energy (indecomposable) weight modules of $\halflattice$, meaning those on which the $h_0$, with $h \in \alg{h}$, act semisimply and $t_0$ has eigenvalues that are bounded below.  (Here, we write $t(z) = \sum_{n \in \ZZ} t_n z^{-n-2}$ as usual.)  These may be induced \cite{BerRep02} from the $\ZZ c$-modules generated by (certain) elements $\ee^h \in \CC[\alg{h}]$ on which $h' \in \alg{h}$ acts as $h' \cdot \ee^h = \inner{h'}{h} \, \ee^h$.  The following is adapted from \cite{AdaRea20} to accommodate our choice of conformal structure.
\begin{proposition}[{\cite[Prop.~3.4]{AdaRea20}}] \label{prop:lvoamod}
	The (twisted) weight $\halflattice$-module generated from $\ee^{rb + jc}$ is positive-energy if and only if $r=\frac{3}{2}$.  In this case, the twisted $\halflattice$-module is simple and the minimal $t_0$-eigenvalue is $\frac{9}{4} \kappa$.
\end{proposition}
\noindent The eigenvalue of $b_0$ on $\ee^{3b/2 + jc}$ is $j + 3 \kappa$.  We therefore define $\lmod{[j]}$, $[j] \in \CC/\ZZ$, to be the simple positive-energy weight $\halflattice$-module generated by $\ee^{3b/2 + (j-3\kappa) c}$ so that the $b_0$-eigenvalues of $\lmod{[j]}$ coincide with $[j]$.  The notation reflects the fact that the isomorphism class of this module only depends on $[j]$ rather than $j$ itself.  We remark that $\ee^{\pm c}_0$ acts injectively on every $\lmod{[j]}$.

\subsection{Inverse \qhr} \label{sec:emd}

The inverse \qhr\ relevant to the present work amounts to embedding the \bp\ minimal model \voa\ $\bpminmoduv$ in the tensor product of $\halflattice$ and the minimal model $\Wminmoduv$, then using this embedding to construct the top-dense $\bpminmoduv$-modules.  This embedding and construction was recently detailed in \cite{AdaRea20}.  Here, we review their main results, adapted to our choice of conformal structure (we also twist their embedding by the conjugation automorphism \eqref{eq:defconj} in order to prioritise \hw\ $\bpminmoduv$-modules over their conjugates).
\begin{theorem}[{\cite[Thms.~3.6 and 6.2]{AdaRea20}}] \label{thm:embedding}
	For $\kk$ nondegenerate-admissible, there exists a vertex operator algebra embedding $\bpminmoduv \ira \Wminmoduv \otimes \halflattice$ given by
	\begin{equation} \label{eq:embedding}
		\begin{gathered}
			J(z) \longmapsto b(z), \qquad L(z) \longmapsto T(z) + t(z), \qquad G^-(z) \longmapsto \ee^{-c}(z), \\
			G^+(z) \longmapsto \no{\brac*{\frac{3(\uu-\vv)}{\vv} \pd a(z) a(z) - a(z)^3 - \frac{(\uu-\vv)^2}{\vv^2} \pd^2 a(z)
				+ \fracuv T(z) a(z) - \frac{\uu(\uu-\vv)}{2\vv^2} \pd T(z) - \sqrt{\frac{\uu^3}{3\vv^3}} W(z)} \ee^c(z)}.
		\end{gathered}
	\end{equation}
	Moreover, such an embedding does not exist when $\uu \ge 2$ and $\vv = 1$ or $2$.
\end{theorem}

\begin{theorem}[{\cite[Thms.~5.12 and 6.3]{AdaRea20}}] \label{thm:construction}
	Let $\kk$ be nondegenerate-admissible.  Then, for each $[\lambda] \in \infwtsuv / \ZZ_3$ and $[j] \in \CC/\ZZ$:
	\begin{itemize}
		\item $\Wihw{[\lambda]} \otimes \lmod{[j]}$ is an indecomposable top-dense $\bpminmoduv$-module on which $G^-_0$ acts injectively.
		\item Every nonzero $\bpminmoduv$-submodule of $\Wihw{[\lambda]} \otimes \lmod{[j]}$ has nonzero intersection with its top space.
		\item If $[j]$ is not in the $\outaut$-orbit of $[j^\twist(\lambda)]$, then $\Wihw{[\lambda]} \otimes \lmod{[j]}$ is a simple $\bpminmoduv$-module.
	\end{itemize}
\end{theorem}

Armed with this information, it is now straightforward to identify these restrictions as $\bpminmoduv$-modules.
\begin{proposition} \label{prop:identification}
	Let $\kk$ be nondegenerate-admissible, $[\lambda] \in \infwtsuv / \ZZ_3$ and $[j] \in \CC/\ZZ$.  Then,
	\begin{equation} \label{eq:identification}
		\Wihw{[\lambda]} \otimes \lmod{[j]} \cong \twrhw{[j],[\lambda]}.
	\end{equation}
\end{proposition}
\begin{proof}
	Note that the $\Wihw{[\lambda]} \otimes \lmod{[j]}$ are completely specified by their top spaces (\cref{thm:construction}), as are the $\twrhw{[j],[\lambda]}$.  It therefore suffices to show that the top spaces of each coincide as modules over the twisted Zhu algebra of $\bpminmoduv$.  The classification of such modules \cite[Thm.~3.22]{FehCla20} shows that this will follow if the $J_0$-, $L_0$- and $\Omega$-eigenvalues all match.  Here, $\Omega$ is a ``cubic Casimir'' of the twisted Zhu algebra that may be identified with
	\begin{equation} \label{eq:defOmega}
		\Omega = G^+_0 G^-_0 + G^-_0 G^+_0 + 2J_0^3 + J_0 - 2J_0 \brac*{\fracuv L_0 + \frac{(\uu-2\vv)(2\uu-3\vv)}{8\vv^2} \wun}.
	\end{equation}

	Checking this matching is immediate for $J_0$.  For $L_0 = T_0 + t_0$, it amounts to verifying that
	\begin{equation}
		\Delta_{\lambda} + \frac{9\kappa}{4} = \Delta^\twist(\lambda).
	\end{equation}
	The $\Omega$-check is likewise straightforward, though tedious.  We only mention that the action on the top space of $\Wihw{[\lambda]} \otimes \lmod{[j]}$ is obtained from \eqref{eq:embedding} and \eqref{eq:defOmega}, whilst the action on the top space of $\twrhw{[j],[\lambda]}$ was computed in \cite[Eq.~(4.16)]{FehCla20}.
\end{proof}
\noindent Recall from \cref{sec:bpsimple} that we chose to define the nonsimple $\twrhw{[j],[\lambda]}$ so that $G^-_0$ would always act injectively.  The reason why is simply that it makes the identification \eqref{eq:identification} true for all cosets $[j]$ rather than for all but three.

\section{Characters and modularity} \label{sec:char}

Having thoroughly reviewed the representation theory of the \bp\ minimal models at nondegenerate admissible levels and the construction of their top-dense modules via inverse \qhr, we are well placed to investigate characters and their modular properties.  For this, we shall employ the standard module formalism developed in \cite{CreLog13,RidVer14} with certain spectral flows of the top-dense modules $\twrhw{[j],[\lambda]}$, $[j] \in \RR/\ZZ$, playing the role of the standard modules.  However, this identification is complicated by the fact that there are twisted and untwisted modules to consider, even though the two sectors are related by spectral flow equivalences.  As we shall see, this complication is conveniently overcome by (temporarily) changing the conformal structure of $\bpminmoduv$.

\subsection{Characters for standard modules} \label{sec:stchar}

We begin by recalling the usual notion of character for $\bpminmoduv$-modules, decorated with an additional factor involving $\kappa$ that will be convenient for our modular studies.  For a $\bpminmoduv$-module $\Mod{M}$, we define its character to be
\begin{equation} \label{eq:defbpchar}
	\bpchm{\Mod{M}} = \yy^{\kappa} \traceover{\Mod{M}} \left( \zz^{J_0} \qq^{L_0 - \bpccuv/24} \right),
\end{equation}
where $\yy = \ee^{2 \pi \ii \theta}$, $\zz = \ee^{2 \pi \ii \zeta}$ and $\qq = \ee^{2 \pi \ii \tau}$.  We remark that this character does not always distinguish inequivalent simple modules.  In particular, it does not keep track of the eigenvalue of the ``cubic Casimir'' $\Omega$ mentioned in the proof of \cref{prop:identification}.  We will overcome this deficiency in the next \lcnamecref{sec:stopfs}.

Our hypothesis, for $\kk$ nondegenerate-admissible, is that the standard modules of $\bpminmoduv$ are spectral flows of the top-dense $\bpminmoduv$-modules $\twrhw{[j],[\lambda]}$ (with $[j] \in \RR / \ZZ$ and $[\lambda] \in \infwtsuv / \ZZ_3$).  However, this places the standard modules in the twisted module category $\twcatuv$ whilst the vacuum module belongs to the untwisted module category $\wcatuv$.  This is inconvenient for Verlinde considerations (though not insurmountable, see for example \cite{CanFus15b,RidAdm17}), hence we shall modify the conformal structure of the \voa\ $\bpminmoduv$ so as to reimagine the $\twrhw{[j],[\lambda]}$ as untwisted modules.

In fact, $\bpminmoduv$ admits a one-parameter family of conformal structures given by
\begin{equation} \label{eq:deftildeL}
	\modify{L}(z) = L(z) + \alpha \pd J(z), \quad \alpha \in \CC;
\end{equation}
the corresponding central charges are $\tbpccuv = \bpccuv - 24 \alpha^2 \kappa$.  Choosing another conformal structure means regrading any weight $\bpminmoduv$-module by the eigenvalue of $\modify{L}_0 = L_0 - \alpha J_0$.  The following modified definition for characters is thus natural:
\begin{equation} \label{eq:deftwbfchar}
	\tbpchm{\Mod{M}}
	= \yy^{\kappa} \traceover{\Mod{M}} \left( \zz^{J_0} \qq^{\modify{L}_0 - \tbpccuv/24} \right)
	= \bpch{\Mod{M}}{\theta + \alpha^2 \tau}{\zeta - \alpha \tau}{\tau}.
\end{equation}
Of course, modifying the conformal grading also results in a modified notion of positive-energy modules and \rhwms.
\begin{proposition}
	Let $\kk$ be nondegenerate-admissible and assume that $\alpha \in \frac{1}{2} \ZZ$.  Then,
	\begin{equation} \label{eq:deftilderhw}
		\tilderhw{[j],[\lambda]} = \sfmod{\alpha}{\twrhw{[j-2 \alpha \kappa],[\lambda]}}
	\end{equation}
	is a \rhwm\ with respect to $\modify{L}(z)$.
\end{proposition}
\begin{proof}
	It follows from \eqref{eq:defsf} and \eqref{eq:deftildeL} that
	\begin{equation} \label{eq:sftildeL0}
		\sfmod{\ell}{\modify{L}_0}
		= L_0 - (\ell+\alpha) J_0 + \ell (\ell+2\alpha) \kappa \wun
		= \modify{L}_0 - \ell J_0 + \ell (\ell+2\alpha) \kappa \wun.
	\end{equation}
	If $v_j$ denotes a \rhwv\ of $\twrhw{[j],[\lambda]}$ of $J_0$-eigenvalue $j$, then
	\begin{equation}
		\modify{L}_0 \sfmod{\ell}{v_j}
		= \sfsymb^{\ell} \brac*{(L_0 + (\ell-\alpha) J_0 + \ell(\ell-2\alpha) \kappa \wun) v_j}
		= \brac*{\Delta^\twist(\lambda) + (\ell-\alpha) j + \ell(\ell-2\alpha) \kappa} \sfmod{\ell}{v_j},
	\end{equation}
	hence the $\modify{L}_0$-eigenvalue is $j$-independent if and only if $\ell = \alpha$.
\end{proof}
\noindent Note that the shift in $j$ on the \rhs\ of \eqref{eq:deftilderhw} ensures that the $J_0$-eigenvalues of $\tilderhw{[j],[\lambda]}$ coincide with the coset $[j] \in \CC/\ZZ$.

To convert the $\twrhw{[j],[\lambda]}$ into untwisted modules $\tilderhw{[j],[\lambda]}$, we therefore need to choose $\alpha \in \ZZ+\frac{1}{2}$.  For simplicity, we shall specialise to $\alpha = \frac{1}{2}$ in what follows.  With this choice, $\bpminmoduv$ is $\ZZ$-graded by $\modify{L}_0$: the conformal weights of $G^+$ and $G^-$ are $1$ and $2$, respectively.  We shall take the \emph{standard modules} to be the $\sfmod{\ell}{\tilderhw{[j],[\lambda]}}$ with $\ell \in \ZZ$, $[j] \in \RR / \ZZ$ and $[\lambda] \in \infwtsuv / \ZZ_3$.

In what follows, we shall make much more use of spectral flow.  For brevity, we will therefore sometimes denote the action of the spectral flow functor $\sfsymb^{\ell}$ on a $\bpminmoduv$-module $\Mod{M}$ by a superscript:  $\sfmod{\ell}{\Mod{M}} = \Mod{M}^{\ell}$.  With this notation, our first task is to compute the characters of the $\tilderhw{[j],[\lambda]}^{\ell}$.  We shall do so by using \cref{prop:identification} to compute the characters of the $\twrhw{[j],[\lambda]}$.  This requires the characters of the $\Wminmoduv$-modules $\Wihw{[\lambda]}$ and the $\halflattice$-modules $\lmod{[j]}$:
\begin{equation} \label{eq:defWPichars}
	\Wchm{\Wihw{[\lambda]}} = \traceover{\Wihw{[\lambda]}} \qq^{T_0 - \Wccuv/24} \quad \text{and} \quad
	\lchm{\lmod{[j]}} = \traceover{\lmod{[j]}} \brac*{\zz^{b_0} \qq^{t_0 - \lccuv/24}}.
\end{equation}
Being modules over a lattice \voa, the $\lmod{[j]}$ have easily computed characters.
\begin{proposition} \label{prop:lch}
	For all $[j] \in \CC/\ZZ$, we have
	\begin{equation}
		\lchm{\lmod{[j]}} = \frac{\zz^j}{\eta(\tau)^2} \sum_{m \in \ZZ} \zz^m,
	\end{equation}
	where $\eta(\tau) = \qq^{1/24} \prod_{n=1}^{\infty}(1-\qq^n)$ is the Dedekind eta function.
\end{proposition}
\noindent Explicit formulae for the characters of the $\Wihw{[\lambda]}$ may be found in many places, for example \cite{BouSom91,FreCha92}.  We shall not need them, noting merely that \cref{prop:identification} immediately gives
\begin{equation} \label{eq:rhwmchar}
	\bpchm{\twrhw{[j],[\lambda]}}
	= \yy^{\kappa} \Wchm{\Wihw{[\lambda]}} \lchm{\lmod{[j]}}
	= \frac{\yy^{\kappa} \zz^j \Wchm{\Wihw{[\lambda]}}}{\eta(\tau)^2} \sum_{m \in \ZZ} \zz^m.
\end{equation}
\begin{lemma} \label{lem:sfmods}
	Given any $\bpminmoduv$-module $\Mod{M}$ (that possesses a character) and $\ell \in \frac{1}{2} \ZZ$, we have
	\begin{equation}
		\begin{aligned}
			\bpchm{\Mod{M}^{\ell}} &= \bpch{\Mod{M}}{\theta + 2 \ell \zeta + \ell^2 \tau}{\zeta + \ell \tau}{\tau} \\ \text{and} \quad
			\tbpchm{\Mod{M}^{\ell}} &= \tbpch{\Mod{M}}{\theta + 2 \ell \zeta + \ell(\ell+1) \tau}{\zeta + \ell \tau}{\tau}.
		\end{aligned}
	\end{equation}
\end{lemma}
\begin{proof}
	The first character identity follows easily from \eqref{eq:defsf}:
	\begin{align}
		\bpchm{\Mod{M}^{\ell}}
		&= \traceover{\sfmod{\ell}{\Mod{M}}} \brac*{\yy^{\kappa} \zz^{J_0} \qq^{L_0 - \bpccuv/24}}
		= \traceover{\Mod{M}} \brac*{\yy^{\kappa} \zz^{\sfmod{-\ell}{J_0}} \qq^{\sfmod{-\ell}{L_0} - \bpccuv/24}} \notag \\
		&= \traceover{\Mod{M}} \brac*{\yy^{\kappa} \zz^{J_0 + 2 \kappa \ell \wun} \qq^{L_0 + \ell J_0 + \kappa \ell^2 \wun - \bpccuv/24}}
		= \charac{\Mod{M}}{\theta + 2 \ell \zeta + \ell^2 \tau}{\zeta + \ell \tau}{\tau}. \notag
	\end{align}
	The second follows in the same way, but using \eqref{eq:sftildeL0} with $\alpha = \frac{1}{2}$.
\end{proof}
\begin{proposition} \label{prop:stchars}
	Let $\kk$ be nondegenerate-admissible.  Then, for all $\ell \in \frac{1}{2} \ZZ$, $[j] \in \CC/\ZZ$ and $[\lambda] \in \infwtsuv / \ZZ_3$, the standard characters have the form
	\begin{equation} \label{eq:stchars}
		\tbpchm{\tilderhw{[j],[\lambda]}^{\ell}}
		= \ee^{2\pi\ii\kappa\brac[\big]{\theta - \ell(\ell+1)\tau}} \frac{\Wchm{\Wihw{[\lambda]}}}{\eta(\tau)^2}
			\sum_{m \in \ZZ} \ee^{2\pi\ii m(j+2\kappa\ell)} \delta(\zeta+\ell\tau-m).
	\end{equation}
\end{proposition}
\begin{proof}
	First combine \eqref{eq:deftilderhw} and \eqref{eq:rhwmchar} with \cref{lem:sfmods} to see that
	\begin{align}
		\bpchm{\tilderhw{[j],[\lambda]}^{\ell}}
		&= \bpchm{\sfmod{\ell+1/2}{\twrhw{[j-\kappa],[\lambda]}}} \\
		&= \frac{\yy^{\kappa} \zz^{j+2\ell\kappa} \qq^{(\ell+1/2)j + (\ell^2-1/4)\kappa} \Wchm{\Wihw{[\lambda]}}}{\eta(\tau)^2}
			\sum_{m \in \ZZ} \zz^m \qq^{(\ell+1/2)m}. \notag
	\end{align}
	Now use \eqref{eq:deftwbfchar} to conclude that
	\begin{equation}
		\tbpchm{\tilderhw{[j],[\lambda]}^{\ell}}
		= \frac{\yy^{\kappa} \zz^{j+2\ell\kappa} \qq^{\ell j + \ell(\ell-1)\kappa} \Wchm{\Wihw{[\lambda]}}}{\eta(\tau)^2} \sum_{m \in \ZZ} \zz^m \qq^{\ell m}.
	\end{equation}
	To conclude, use the standard identity $\sum_{m \in \ZZ} \ee^{2\pi\ii mx} = \sum_{m \in \ZZ} \delta(x-m)$.
\end{proof}

\subsection{One-point functions for standard modules} \label{sec:stopfs}

As appealing as the standard character formula \eqref{eq:stchars} is, the result has a highly undesirable feature: the standard characters are not linearly independent.  This means that characters cannot distinguish isomorphism classes of simple $\bpminmoduv$-modules and so any Verlinde computations relying on them will give ambiguous answers.

The root cause of this failure of linear independence is the well known fact that the $\Wminmoduv$-characters are not linearly independent either: the definition \eqref{eq:defWPichars} ignores the eigenvalue of $W_0$.  As the conjugation automorphism \eqref{eq:W3conj} of $\Wminmoduv$ preserves $T_0$-eigenvalues but negates $W_0$-eigenvalues, conjugate $\Wminmoduv$-modules will always have the same character.  The simple characters will therefore be linearly dependent whenever $\Wminmoduv$ admits a \hwv\ with a nonzero $W_0$-eigenvalue.

This issue was recently resolved in \cite{AraMod19} by considering \emph{one-point functions} instead of characters.  Here, the definition of the character is ``upgraded'' by inserting the zero mode of some $u \in \Wminmoduv$:
\begin{equation} \label{eq:W3opfs}
	\Wopfm{\Wihw{[\lambda]}} = \traceover{\Wihw{[\lambda]}} \brac*{u_0 \qq^{T_0 - \Wccuv/24}}.
\end{equation}
Because $\Wminmoduv$ is rational and $C_2$-cofinite \cite{AraAss15,AraRat15}, these one-point functions are linearly independent for generic choices of $u$ \cite{ZhuMod96}.  In particular, as $\Wihw{[\lambda]}$ is a simple \hwm, completely specified by the eigenvalues of $T_0$ and $W_0$ on the \hwv, we have the desired linear independence when $u=W$.

In fact, this conclusion needs a minor refinement because it may happen that $W$ is zero in $\Wminmoduv$.  From the \opes\ \eqref{eq:ope:W3} of the universal Zamolodchikov algebra, we see that $W$ is null in $\uWvoak$ (hence zero in $\Wminmoduv$) if and only if $\Wcck = 0$ or $A_{\kk} = 0$.  But, $\Wcck = 0$ if and only if $(\uu,\vv) = (3,4),(4,3)$ and $\Wminmod{3}{4} = \Wminmod{4}{3}$ is the trivial (one-dimensional) \voa.  Similarly, $A_{\kk} = 0$ if and only if $(\uu,\vv) = (3,5),(5,3)$ and $\Wminmod{3}{5} = \Wminmod{5}{3}$ is the Virasoro minimal model $\virminmod{2}{5}$.  It follows that when $W=0$, the characters of the minimal model are linearly independent, so we may take $u = \wun$ in \eqref{eq:W3opfs}.  For all other $\Wthree$ minimal models, we take $u = W$.

We can similarly upgrade the definition of $\bpminmoduv$-characters to one-point functions as follows:
\begin{equation} \label{eq:bpopfs}
	\begin{aligned}
		\bpopfm{\Mod{M}} &= \yy^{\kappa} \traceover{\Mod{M}} \left( u_0 \zz^{J_0} \qq^{L_0 - \bpccuv/24} \right), \\
		\tbpopfm{\Mod{M}} &= \yy^{\kappa} \traceover{\Mod{M}} \left( u_0 \zz^{J_0} \qq^{\modify{L}_0 - \tbpccuv/24} \right),
	\end{aligned}
	\qquad u \in \bpminmoduv.
\end{equation}
The question is now if there is a choice of $u$ guaranteeing linear independence.  As $\bpminmoduv$ is neither rational nor $C_2$-cofinite when $\kk$ is nondegenerate-admissible \cite{AdaRea20,FehCla20}, this is not immediately clear.

Our end goal for these one-point functions is, however, the modular properties when $\Mod{M}$ is a standard module.  By \cref{prop:identification}, a standard module of $\bpminmoduv$ is always a $\Wminmoduv \otimes \halflattice$-module, something that is not true for general $\bpminmoduv$-modules.  We may therefore take $u$ to be an element of $\Wminmoduv \otimes \halflattice$ and know that the one-point functions \eqref{eq:bpopfs}, with $\Mod{M}$ standard, are well defined.  In particular, we may choose $u = \wun \otimes \wun$ when $(\uu,\vv) \in \set[\big]{(3,4),(4,3),(3,5),(5,3)}$ and $u = W \otimes \wun$ otherwise.  It is now clear how to lift \cref{prop:stchars} to linearly independent one-point functions.
\begin{proposition} \label{prop:stopfs}
	Let $\kk$ be nondegenerate-admissible.  Then, for all $\ell \in \frac{1}{2} \ZZ$, $[j] \in \CC/\ZZ$ and $[\lambda] \in \infwtsuv / \ZZ_3$, we have
	\begin{equation} \label{eq:stopfs}
		\tbpopfm{\tilderhw{[j],[\lambda]}^{\ell}}
		= \ee^{2\pi\ii\kappa\brac[\big]{\theta - \ell(\ell+1)\tau}} \frac{\Wopfm{\Wihw{[\lambda]}}}{\eta(\tau)^2}
			\sum_{m \in \ZZ} \ee^{2\pi\ii m(j+2\kappa\ell)} \delta(\zeta+\ell\tau-m).
	\end{equation}
	Moreover, if we take $u=\wun$ when $(\uu,\vv) \in \set[\big]{(3,4),(4,3),(3,5),(5,3)}$ and $u = W$ otherwise, then these standard one-point functions are linearly independent.
\end{proposition}
\noindent Note the slight abuse of notation in writing $u$ instead of $u \otimes \wun$ on the \lhs\ of \eqref{eq:stopfs}.

\subsection{Modularity of standard one-point functions} \label{sec:stmod}

The S-transforms of the $\Wminmoduv$-characters were first obtained in \cite{FreCha92}, though the issue with the linear dependence of the characters was not resolved until recently \cite{AraMod19}.  Since $u=\wun$ or $W$ is a Virasoro \hwv\ of conformal weight $\Delta_u=0$ or $3$, respectively, the S-transform of the $\Wminmoduv$ one-point functions takes the following simple form \cite{ZhuMod96}:
\begin{equation} \label{eq:W3Stransform}
	\Wopf{\Wihw{[\lambda]}}{-\frac{1}{\tau}}{\frac{u}{\tau^{\Delta_u}}}
	= \sum_{[\lambda'] \in \infwtsuv / \ZZ_3} \WSmatrix{[\lambda]}{[\lambda']} \Wopfm{\Wihw{[\lambda']}}.
\end{equation}
The explicit form of the $\Wminmoduv$ S-matrix $\WSmatrix{[\lambda]}{[\lambda']}$ is given in \cref{thm:W3Smatrix}.

Define the following transformations on the parameter space $\left( \theta \middle\bracevert \zeta \middle\bracevert \tau \ ; u \right)$:
\begin{equation} \label{eq:STcoordtransforms}
	\mods \colon \left( \theta \middle\bracevert \zeta \middle\bracevert \tau  \ ; \ u \right)
	\longmapsto \left( \theta-\frac{\zeta^2}{\tau}-\frac{\zeta}{\tau}+\zeta
		\middle\bracevert \frac{\zeta}{\tau}
		\middle\bracevert -\frac{1}{\tau}
		\ ; \ \frac{u}{\tau^{\Delta_u}} \right) \quad \text{and} \quad
	\modt \colon \left( \theta \middle\bracevert \zeta \middle\bracevert \tau \ ; \ u \right)
	\longmapsto \left( \theta \middle\bracevert \zeta \middle\bracevert \tau+1 \ ; \ u \right).
\end{equation}
That this defines an $\SLG{SL}{2}(\ZZ)$-action is a straightforward computation:
\begin{equation}
	\mods^2 = (\mods \modt)^3 = \modc \colon \left( \theta \middle\bracevert \zeta \middle\bracevert \tau  \ ; \ u \right)
	\longmapsto \left( \theta + 2 \zeta \middle\bracevert -\zeta \middle\bracevert \tau  \ ; \ (-1)^{\Delta_u} u \right).
\end{equation}
Obviously, $\modc$ squares to the identity as required.
\begin{theorem} \label{thm:stopfS}
	Let $\kk$ be nondegenerate-admissible.  Then, for each $\ell \in \ZZ$, $[j] \in \RR/\ZZ$ and $[\lambda] \in \infwtsuv / \ZZ_3$, the S-transform of the one-point function of $\tilderhw{[j],[\lambda]}^{\ell}$ is given by
	\begin{equation} \label{eq:stransform}
		\mods \set*{\tchnoargs{\tilderhw{[j],[\lambda]}^{\ell}}}
		= \frac{\abs{\tau}}{-\ii \tau} \sum_{\ell' \in \ZZ} \int_{\RR/\ZZ} \sum_{[\lambda'] \in \infwtsuv / \ZZ_3}
			\Smatrix{\ell,[j],[\lambda]}{\ell',[j'],[\lambda']} \tchnoargs{\tilderhw{[j'],[\lambda']}^{\ell'}} \, \dd [j'],
	\end{equation}
	where the entries of the ``S-matrix'' (integral kernel) are
	\begin{equation} \label{eq:smatrix}
		\Smatrix{\ell,[j],[\lambda]}{\ell',[j'],[\lambda']}
		= \WSmatrix{[\lambda]}{[\lambda']} \ee^{-2 \pi \ii \brac[\big]{2 \kappa \ell \ell' + \ell (j'-\kappa) + (j-\kappa) \ell'}}.
	\end{equation}
\end{theorem}
\begin{proof}
	Our strategy is to evaluate and simplify both sides of \eqref{eq:stransform}.  Starting with the \lhs, we have
	\begin{align}
		\mods &\set*{\tbpopfm{\tilderhw{[j],[\lambda]}^{\ell}}}
		= \tbpopf{\tilderhw{[j],[\lambda]}^{\ell}}
			{\theta-\frac{\zeta^2}{\tau}-\frac{\zeta}{\tau}+\zeta}{\frac{\zeta}{\tau}}{-\frac{1}{\tau}}{\frac{u}{\tau^{\Delta_u}}} \\
		&\qquad = \exp \sqbrac*{2 \pi \ii \kappa \brac[\Big]{\theta - \frac{\zeta^2}{\tau} - \frac{\zeta}{\tau} + \zeta + \frac{\ell(\ell+1)}{\tau}}}
			\frac{\Wopf{\Wihw{[\lambda]}}{-\frac{1}{\tau}}{\frac{u}{\tau^{\Delta_u}}}}{-\ii \tau \eta(\tau)^2}
			\sum_{m \in \ZZ} \ee^{2 \pi \ii m (j+2\kappa\ell)} \delta \left( \frac{\zeta}{\tau} - \frac{\ell}{\tau} - m\right) \notag
		\intertext{(using \cref{prop:stopfs} and the well known S-transform of Dedekind's eta function)}
		&\qquad = \frac{\abs{\tau}}{-\ii \tau \eta(\tau)^2}
			\sum_{[\lambda']} \WSmatrix{[\lambda]}{[\lambda']} \Wopfm{\Wihw{[\lambda']}} \notag \\
		&\mspace{200mu} \cdot \exp \sqbrac*{2 \pi \ii \kappa \brac[\Big]{\theta - \frac{\zeta^2}{\tau} - \frac{\zeta}{\tau} + \zeta + \frac{\ell(\ell+1)}{\tau}}}
			\sum_{m \in \ZZ} \ee^{2 \pi \ii (j+2\kappa\ell) m} \delta \left( \zeta - \ell - m\tau \right) \notag \\
		&\qquad = \frac{\abs{\tau}}{-\ii \tau \eta(\tau)^2}
			\sum_{[\lambda']} \WSmatrix{[\lambda]}{[\lambda']} \chnoargs{\Wihw{[\lambda']}} \, \ee^{2 \pi \ii \kappa (\theta+\ell)}
			\sum_{m \in \ZZ} \ee^{-2 \pi \ii \brac[\big]{(j-\kappa) m + \kappa m(m+1) \tau}} \delta \left( \zeta + m\tau - \ell \right) \notag
	\end{align}
	(using \eqref{eq:W3Stransform} and the properties of the delta function).  Here, and below, the $[\lambda']$-sums run over $\infwtsuv / \ZZ_3$.  Inserting \eqref{eq:smatrix} into the right-hand side, similar manipulations result in the same answer:
	\begin{align}
		\frac{\abs{\tau}}{-\ii \tau} &\sum_{\ell' \in \ZZ} \int_{\RR/\ZZ} \sum_{[\lambda']}
			\Smatrix{\ell,[j],[\lambda]}{\ell',[j'],[\lambda']} \tchnoargs{\tilderhw{[j'],[\lambda']}^{\ell'}} \, \dd [j'] \\
		&= \frac{\abs{\tau}}{-\ii \tau \eta(\tau)^2}
			\sum_{[\lambda']} \WSmatrix{[\lambda]}{[\lambda']} \chnoargs{\Wihw{[\lambda']}} \notag \\
		&\mspace{50mu}\cdot \sum_{\ell' \in \ZZ} \int_{\RR/\ZZ} \ee^{-2 \pi \ii \brac[\big]{2 \kappa \ell \ell' + \ell (j'-\kappa) + (j-\kappa) \ell'}}
			\ee^{2 \pi \ii \kappa \brac[\big]{\theta - \ell'(\ell'+1)\tau}}
			\sum_{m \in \ZZ} \ee^{2 \pi i (j'+2\kappa\ell') m} \delta \left( \zeta + \ell'\tau - m\right) \, \dd [j'] \notag \\
		&= \frac{\abs{\tau}}{-\ii \tau \eta(\tau)^2}
			\sum_{[\lambda']} \WSmatrix{[\lambda]}{[\lambda']} \chnoargs{\Wihw{[\lambda']}}
			\sum_{\ell' \in \ZZ} \ee^{-2 \pi i \brac[\big]{(j-\kappa) \ell' - \kappa \ell}}
			\ee^{2 \pi i \kappa \brac[\big]{\theta - \ell'(\ell'+1)\tau}} \delta \left( \zeta + \ell'\tau - \ell \right) \notag \\
		&= \frac{\abs{\tau}}{-\ii \tau \eta(\tau)^2}
			\sum_{[\lambda']} \WSmatrix{[\lambda]}{[\lambda']} \chnoargs{\Wihw{[\lambda']}} \, \ee^{2 \pi \ii \kappa (\theta+\ell)}
			\sum_{m \in \ZZ} \ee^{-2 \pi \ii \brac[\big]{(j-\kappa) m + \kappa m(m+1) \tau}} \delta \left( \zeta + m\tau - \ell \right). \qedhere \notag
	\end{align}
\end{proof}
\noindent We remark that the residual factor of $\abs{\tau} / (-\ii \tau)$ in \eqref{eq:stransform} may also be absorbed by further adjusting the coordinate modular transformation \eqref{eq:STcoordtransforms}.  This adjustment will not be detailed here, but the interested reader may refer to \cite{RidBos14} for a similar example.  We also note that the explicit formula for the (diagonal) T-matrix of the standard one-point functions is very easy to derive.  As we shall not need this formula, it is likewise omitted.

The ``matrix elements'' $\Smatrix{\ell, [j], [\lambda]}{\ell', [j'], [\lambda']}$ are manifestly symmetric because the $\WSmatrix{[\lambda]}{[\lambda']}$ are.  It is also easy to check that the $\bpminmoduv$ ``S-matrix'' is unitary and its square represents conjugation, properties which again follow from those of the $\Wminmoduv$ S-matrix.

\section{The \bp\ minimal models $\bpminmod{\uu}{3}$} \label{sec:bpu3}

We have determined a set of standard modules for the \bp\ minimal models $\bpminmoduv$, computed their linearly independent one-point functions and determined the consequent modular S-transforms.  According to the standard module formalism of \cite{CreLog13,RidVer14}, the other simple (untwisted) $\bpminmoduv$-modules may be resolved in terms of the nonsimple standard modules
\begin{equation} \label{eq:defatypstandards}
	\sfmod{\ell+1/2}{\twrhw{\lambda}} = \sfmod{\ell+1/2}{\twrhw{[j^\twist(\lambda)],[\lambda]}}
	= \sfmod{\ell}{\tilderhw{[j^{\twist}(\lambda) + \kappa],[\lambda]}}
	= \tilderhw{\lambda}^{\ell}, \quad \ell \in \ZZ,\ \lambda \in \infwtsuv.
\end{equation}
In this \lcnamecref{sec:bpu3}, we shall derive these resolutions and determine the consequent modularity of the remaining simple modules when $\kk$ is nondegenerate-admissible with $\vv=3$.  The more technically demanding generalisation to $\vv>3$ will be discussed in \cref{sec:bpuv}.

The motivation for initially restricting to $\vv=3$ is purely to present the analysis with a minimum of complications.  In particular, every \hw\ $\bpminmod{\uu}{3}$-module is type-$3$ (\cref{sec:bpspecflow}).  As we shall see, this means that the resolutions of these modules all have the same form (up to spectral flow), significantly reducing the number of cases that need to be considered.  Another related simplification is that for $\vv=3$, $\lambda \in \infwtsuv$ corresponds to $\sss = [0,0,0]$.

\subsection{Resolutions} \label{sec:bpu3res}

We begin with the short exact sequence of \cref{prop:ses}.  The highest weight of the quotient is required to be the leftmost in its orbit as pictured in \cref{fig:sforbits}.  For $\vv=3$, the highest weight is type-$3$ and so the leftmost has $\sss = [0,-1,1]$.  The short exact sequence is thus
\begin{equation} \label{ses:v=3}
	\dses{\ihwpar{r_0 & r_1 & r_2 \\ 0&0&0}^1}{}{\trhwpar{r_0 & r_1 & r_2 \\ 0&0&0}}{}{\ihwpar{r_0 & r_1 & r_2 \\ 0 & -1 & 1}}.
\end{equation}
The highest weight of the submodule (without spectral flow) is in $\infwts{\uu}{3}$, hence it is the rightmost in its orbit.  As the orbit is type-$3$, it is obtained from the leftmost by spectrally flowing twice.  By \cref{whenarespecflowshw}, we thus have
\begin{equation}
	\ihwpar{r_0 & r_1 & r_2 \\ 0&0&0} \cong \ihwpar{r_2 & r_0 & r_1 \\ 0 & -1 & 1}^2.
\end{equation}
We can therefore splice the exact sequence \eqref{ses:v=3} with that obtained by applying $\sfsymb^3$ to the corresponding exact sequence with quotient $\ihwpar{r_2 & r_0 & r_1 \\ 0 & -1 & 1}$.  Iterating this, we arrive at the desired resolution.
\begin{proposition} \label{prop:resv=3}
	Let $\kk$ be admissible with $\vv=3$.  Then, every simple \hw\ $\bpminmod{\uu}{3}$-module is resolved by the nonsimple standard modules as follows:
	\begin{subequations}
		\begin{gather}
			\cdots \lra \trhwpar{r_0 & r_1 & r_2 \\ 0&0&0}^9 \lra \trhwpar{r_1 & r_2 & r_0 \\ 0&0&0}^6 \lra
				\trhwpar{r_2 & r_0 & r_1 \\ 0&0&0}^3 \lra \trhwpar{r_0 & r_1 & r_2 \\ 0&0&0} \lra \ihwpar{r_0 & r_1 & r_2 \\ 0 & -1 & 1} \lra 0, \\
			\cdots \lra \trhwpar{r_1 & r_2 & r_0 \\ 0&0&0}^{10} \lra \trhwpar{r_2 & r_0 & r_1 \\ 0&0&0}^7 \lra
				\trhwpar{r_0 & r_1 & r_2 \\ 0&0&0}^4 \lra \trhwpar{r_1 & r_2 & r_0 \\ 0&0&0}^1 \lra \ihwpar{r_0 & r_1 & r_2 \\ 1 & -1 & 0} \lra 0, \\
			\cdots \lra \trhwpar{r_2 & r_0 & r_1 \\ 0&0&0}^{11} \lra \trhwpar{r_0 & r_1 & r_2 \\ 0&0&0}^8 \lra
				\trhwpar{r_1 & r_2 & r_0 \\ 0&0&0}^5 \lra \trhwpar{r_2 & r_0 & r_1 \\ 0&0&0}^2 \lra \ihwpar{r_0 & r_1 & r_2 \\ 0 & 0 & 0} \lra 0.
		\end{gather}
	\end{subequations}
\end{proposition}
\noindent The other two resolutions are obtained from the first by applying one or two units of spectral flow. For reasons that will become clear shortly, we focus on the \hw\ $\bpminmod{\uu}{3}$-modules with $\sss = [1,-1,0]$.
\begin{corollary} \label{cor:chHv=3}
	Let $\kk$ be admissible with $\vv=3$.  Then, for all $\rrr \in \pwlat{\uu-3}$ and $\ell \in \frac{1}{2} \ZZ$, we have
	\begin{equation} \label{eq:type3characv=3}
		\tchnoargs{ \ihwpar{r_0 & r_1 & r_2 \\ 1 & -1 & 0}^\ell }
		= \sum_{n=0}^\infty (-1)^n \brac[\Big]{
			\tchnoargs{\trhwpar{r_1 & r_2 & r_0 \\ 0&0&0}^{\ell+9n+1}}
			- \tchnoargs{\trhwpar{r_0 & r_1 & r_2 \\ 0&0&0}^{\ell+9n+4}}
			+ \tchnoargs{\trhwpar{r_2 & r_0 & r_1 \\ 0&0&0}^{\ell+9n+7}}}.
	\end{equation}
\end{corollary}

\noindent The analogous formulae for $\sss = [0,-1,1]$ and $\sss=[0,0,0]$ are obtained by applying $-1$ and $1$ unit of spectral flow, respectively.

It follows that the (linearly independent) standard one-point functions form a topological basis for the space of all one-point functions of $\bpminmod{\uu}{3}$-modules.  The S-transforms of the \hw\ one-point functions thus follow trivially from the standard ones, computed in \cref{thm:stopfS}.

Note that the $r$-labels of the three summands appearing on the \rhs\ of \eqref{eq:type3characv=3} are related by the $\ZZ_3$-action.  This allows us to rewrite \eqref{eq:type3characv=3} in the following alternative form:
\begin{equation} \label{eq:type3characv=3'}
	\tchnoargs{\ihwpar{r_0 & r_1 & r_2 \\ 1 & -1 & 0}^\ell}
	= \sum_{n=0}^\infty (-1)^n \brac[\Big]{
		\tchnoargs{\trhwnewpar{\outaut^{-1}(\rrr),0}^{\ell+9n+1}}
		- \tchnoargs{\trhwnewpar{\rrr,0}^{\ell+9n+4}}
		+ \tchnoargs{\trhwnewpar{\outaut(\rrr),0}^{\ell+9n+7}}}.
\end{equation}
Here, $0$ is being used as a shorthand for the $s$-triple $[0,0,0]$.  We shall also find it convenient to introduce notation for the $J_0$-eigenvalue of a highest weight with $\sss=[1,-1,0]$:
\begin{equation} \label{eq:j(r)}
	j(\rrr) \equiv j \brac[\big]{\wtnewpar{\rrr,[1,-1,0]}} = \tfrac{1}{3} (r_1-r_2).
\end{equation}
\begin{theorem} \label{thm:type3modularityv=3}
	Let $\kk$ be admissible with $\vv=3$ and take $\rrr \in \pwlat{\uu-3}$.  Then for all $\ell \in \ZZ$, the S-transform of the one-point function of $\ihwpar{r_0 & r_1 & r_2 \\ 1 & -1 & 0}^\ell$ is given by
	\begin{equation}
		\mods \set*{\tchnoargs{\ihwpar{r_0 & r_1 & r_2 \\ 1 & -1 & 0}^\ell}}
		= \frac{\abs{\tau}}{-\ii \tau} \sum_{\ell' \in \ZZ} \int_{\RR/\ZZ} \sum_{[\lambda'] \in \infwts{\uu}{3} / \ZZ_3}
			\Smatrix{\ell,\rrr}{\ell',[j'],[\lambda']} \tchnoargs{\tilderhw{[j'],[\lambda']}^{\ell'}} \, \dd [j'],
	\end{equation}
	where the entries of the ``\hw\ S-matrix" are given by
	\begin{equation} \label{eq:type3Smatrixv=3}
		\Smatrix{\ell,\rrr}{\ell',[j'],[\lambda']}
		= \WSmatrix{[\wtnewpar{\rrr,0}]}{[\lambda']}
			\frac{\ee^{-2 \pi \ii \brac[\big]{2 \kappa (\ell-1/2) \ell' + (\ell-1/2) (j'-\kappa) + j(\rrr) \ell'}}}{2 \cos \brac[\big]{3 \pi (j'-\kappa)}}.
	\end{equation}
\end{theorem}
\begin{proof}
	By \cref{cor:chHv=3}, the S-matrix entry $\Smatrix{\ell,\rrr}{\ell',[j'],[\lambda']}$ for the one-point function of $\ihwpar{r_0 & r_1 & r_2 \\ 1 & -1 & 0}^\ell$ may be written as an infinite linear combination of standard S-matrix entries \eqref{eq:smatrix}.  Recall that $\tilderhw{\mu} = \tilderhw{[j(\mu)+2\kappa],[\mu]}$ 	and note that the $\mu$ corresponding to the standard one-point functions on the \rhs\ of \eqref{eq:type3characv=3} or \eqref{eq:type3characv=3'} all belong to the same class $[\wtnewpar{\rrr,0}]$ in $\infwts{\uu}{3} / \ZZ_3$, since $0$ is obviously $\outaut$-invariant.  Comparing each $j(\mu)$ with $j(\rrr) = j(\wtnewpar{\rrr,[1,-1,0]})$ then gives
	\begin{equation} \label{eq:inflincomb}
		\Smatrix{\ell,\rrr}{\ell',[j'],[\lambda']} = \sum_{n=0}^\infty (-1)^n \brac*{
			\Smatrix{\ell+9n+1,[j(\rrr)-2\kappa], [\wtnewpar{\rrr,0}]}{\ell',[j'],[\lambda']}
			- \Smatrix{\ell+9n+4,[j(\rrr)+\kappa-\frac{1}{2}], [\wtnewpar{\rrr,0}]}{\ell',[j'],[\lambda']}
			+ \Smatrix{\ell+9n+7,[j(\rrr)+4\kappa], [\wtnewpar{\rrr,0}]}{\ell',[j'],[\lambda']}},
	\end{equation}
	where we have also used the fact that $\uu = 9(\kappa+\frac{1}{2})$ (since $\vv=3$).  From \eqref{eq:smatrix}, we obtain
	\begin{equation}
		\begin{aligned}
			\Smatrix{\ell+9n+1,[j(\rrr)-2\kappa], [\wtnewpar{\rrr,0}]}{\ell',[j'],[\lambda']}
			&= \WSmatrix{[\wtnewpar{\rrr,0}]}{[\lambda']} \ee^{-2 \pi \ii \brac[\big]{2 \kappa (\ell+9n+1) \ell' + (\ell+9n+1) (j'-\kappa) + (j(\rrr)-3\kappa) \ell'}}, \\
			\Smatrix{\ell+9n+4,[j(\rrr)+\kappa-\frac{1}{2}], [\wtnewpar{\rrr,0}]}{\ell',[j'],[\lambda']}
			&= \ee^{-6 \pi \ii (j'-\kappa)} \Smatrix{\ell+9n+1,[j(\rrr)-2\kappa], [\wtnewpar{\rrr,0}]}{\ell',[j'],[\lambda']} \\ \text{and} \qquad
			\Smatrix{\ell+9n+7,[j(\rrr)+4\kappa], [\wtnewpar{\rrr,0}]}{\ell',[j'],[\lambda']}
			&= \ee^{-12 \pi \ii (j'-\kappa)} \Smatrix{\ell+9n+1,[j(\rrr)-2\kappa], [\wtnewpar{\rrr,0}]}{\ell',[j'],[\lambda']}.
		\end{aligned}
	\end{equation}
	Substituting into \eqref{eq:inflincomb} now gives the desired result:
	\begin{align}
		\Smatrix{\ell,\rrr}{\ell',[j'],[\lambda']}
		&= \brac*{1 - \ee^{-6 \pi \ii (j'-\kappa)} + \ee^{-12 \pi \ii (j'-\kappa)}} \, \WSmatrix{[\wtnewpar{\rrr,0}]}{[\lambda']}
			\sum_{n=0}^\infty (-1)^n \ee^{-2 \pi \ii \brac[\big]{2 \kappa (\ell+9n+1) \ell' + (\ell+9n+1) (j'-\kappa) + (j(\rrr)-3\kappa) \ell'}} \\
		&= \frac{1 + \ee^{-18 \pi \ii (j'-\kappa)}}{1 + \ee^{-6 \pi \ii (j'-\kappa)}} \, \WSmatrix{[\wtnewpar{\rrr,0}]}{[\lambda']}
			\ee^{-2 \pi \ii \brac[\big]{2 \kappa (\ell+1) \ell' + (\ell+1) (j'-\kappa) + (j(\rrr)-3\kappa) \ell'}} \frac{1}{1+\ee^{-18 \pi \ii (j'-\kappa)}} \notag \\
		&= \WSmatrix{[\wtnewpar{\rrr,0}]}{[\lambda']} \frac{\ee^{-2 \pi \ii \brac[\big]{2 \kappa (\ell-1/2) \ell' + (\ell-1/2) (j'-\kappa) + j(\rrr) \ell'}}}{2 \cos \brac[\big]{3 \pi (j'-\kappa)}}. \notag \qedhere
	\end{align}
\end{proof}

Of particular importance for Grothendieck fusion rule computations are the S-matrix elements corresponding to the vacuum module $\ihw{\kk \fwt{0}} = \ihwpar{\uu-3 & 0 & 0 \\ 1 & -1 & 0}$.  These will be given the special notation $\vacSmatrix{\ell',[j'],[\lambda']} = \Smatrix{0,[\uu-3,0,0]}{\ell',[j'],[\lambda']}$.
\begin{corollary} \label{cor:vacmodularityv=3}
	Let $\kk$ be admissible with $\vv=3$.  Then,
	\begin{equation} \label{eq:vacsmatrixv=3}
		\vacSmatrix{\ell',[j'],[\lambda']}
		= \vacWSmatrix{[\lambda']} \frac{\ee^{2 \pi \ii \kappa \ell'} \ee^{\pi \ii (j'-\kappa)}}{2 \cos \brac[\big]{3 \pi (j'-\kappa)}}, \qquad
		\vacWSmatrix{[\lambda']} = \WSmatrix{[\wtnewpar{[\uu-3,0,0],0}]}{[\lambda']}.
	\end{equation}
\end{corollary}
\noindent Note that $\Wihwnewpar{[\uu-3,0,0],0}$ is the vacuum module of $\Wminmod{\uu}{3}$ because \eqref{eq:W3eigs} gives $\Delta \rspar{\uu-3&0&0\\0&0&0} = w \rspar{\uu-3&0&0\\0&0&0} = 0$.

As with the analysis of the admissible-level $\sltwo$ minimal models reported in \cite{CreMod13}, the vacuum S-matrix element \eqref{eq:vacsmatrixv=3} diverges when $\tilderhw{[j'],[\lambda']}^{\ell'}$ is nonsimple.  To see this, recall that $\tilderhw{[j'],[\lambda']}^{\ell'}$ is nonsimple when $[j'] = [j^\twist\brac[\big]{(\outaut^i(\lambda')}+\kappa]$ for some $i \in \ZZ_3$.  As $\outaut^i(\lambda') \in \infwts{\uu}{3}$, it is given by $\wtnewpar{\rrr,0}$ for some $\rrr \in \pwlat{\uu-3}$.  However, $j^\twist \left( \wtnewpar{\rrr,0} \right) + \kappa = \tfrac{1}{3}\left(r_1-r_2\right) - \tfrac{1}{2}$, so the denominator of \eqref{eq:vacsmatrixv=3} becomes $\cos \brac[\big]{\pi(r_1 - r_2) - \frac{3\pi}{2}} = 0$ when $\tilderhw{[j'],[\lambda']}^{\ell'}$ is nonsimple.

\subsection{Grothendieck Fusion Rules} \label{sec:bpu3fus}

One of the most beautiful results in rational \cft\ is the Verlinde formula, discovered by Verlinde \cite{VerFus88} and proven by Huang \cite{HuaVer05,HuaVer08}.  It expresses the fusion coefficients, which are nonnegative integers, in terms of the entries of the modular S-matrix, which are algebraic numbers in general.  This formula does not apply to nonrational theories such as the \bp\ minimal models studied here, but there is a conjectural extension that has been successfully tested in a wide range of examples.  This is the \emph{standard Verlinde formula} of \cite{CreLog13,RidVer14}.

We present this formula in the following conjecture for all \bp\ minimal models with nondegenerate admissible levels $\kk$.  Note however that it computes not the fusion coefficients but the \emph{Grothendieck} fusion coefficients, these being the structure constants of the Grothendieck group of the category of standard modules, equipped with (the image of) the fusion product.  As characters (and one-point functions) are blind to the difference between a module and the direct sum of its composition factors, these coefficients are all that one could hope to access using modularity.

Of course, to consistently equip the Grothendieck group with the fusion product, one needs to know that fusing with a standard module defines an exact functor.  This appears to be very difficult to establish, so we shall have to conjecture that it does hold.  In fact, we believe that a slightly stronger statement is true: the category of standard modules is rigid.  Assuming this, the standard Verlinde conjecture is as follows.
\begin{conjecture} \label{conj:SVF}
	Let $\kk$ be admissible-nondegenerate.  Then, for $\ell,\ell' \in \ZZ$, $[j],[j'] \in \RR/\ZZ$ and $[\lambda], [\lambda'] \in \infwtsuv / \ZZ_3$, the Grothendieck fusion rules of the standard $\bpminmoduv$-modules are given by
	\begin{subequations} \label{eq:SVF}
		\begin{equation} \label{eq:SVF1}
			\Gfusion{\tilderhw{[j],[\lambda]}^{\ell}}{\tilderhw{[j'],[\lambda']}^{\ell'}}
			= \sum_{\ell'' \in \ZZ} \int_{\RR/\ZZ} \sum_{[\lambda''] \in \infwtsuv / \ZZ_3}
				\fuscoeff{\ell, [j], [\lambda]}{\ell', [j'], [\lambda']}{\ell'', [j''], [\lambda'']} \Gr{\tilderhw{[j''],[\lambda'']}^{\ell''}} \, \dd [j''],
		\end{equation}
		where the Grothendieck fusion coefficients are given by
		\begin{equation} \label{eq:SVF2}
			\fuscoeff{\ell, [j], [\lambda]}{\ell', [j'], [\lambda']}{\ell'', [j''], [\lambda'']}
			= \sum_{m \in \ZZ} \int_{\RR/\ZZ} \sum_{[\mu] \in \infwtsuv / \ZZ_3} \frac{\Smatrix{\ell, [j], [\lambda]}{m, [k], [\mu]} \
				\Smatrix{\ell', [j'], [\lambda']}{m, [k], [\mu]} \ \brac*{\Smatrix{\ell'', [j''], [\lambda'']}{m, [k], [\mu]}}^*}{\vacSmatrix{m, [k], [\mu]}} \, \dd [k].
		\end{equation}
	\end{subequations}
	Here, the asterisk indicates complex conjugation.
\end{conjecture}
\noindent The results obtained in the remainder of this \lcnamecref{sec:bpu3} will implicitly assume that this \lcnamecref{conj:SVF} holds.
We now apply the standard Verlinde formula to compute the Grothendieck fusion rules of the standard $\bpminmod{\uu}{3}$-modules.  First, note that substituting the factorisation
\begin{equation}
	\Smatrix{\ell, [j], [\lambda]}{m, [k], [\mu]} = \ee^{-2 \pi \ii \ell (2 \kappa \ell' + j'-\kappa)} \Smatrix{0, [j], [\lambda]}{m, [k], [\mu]}
\end{equation}
into \eqref{eq:SVF2} results in
\begin{equation} \label{eq:shortcut}
	\fuscoeff{\ell, [j], [\lambda]}{\ell', [j'], [\lambda']}{\ell'', [j''], [\lambda'']}
	= \fuscoeff{0, [j], [\lambda]}{0, [j'], [\lambda']}{\ell''-\ell'-\ell, [j''], [\lambda'']}.
\end{equation}
It follows that the Grothendieck fusion rule for $\Gfusion{\tilderhw{[j],[\lambda]}^{\ell}}{\tilderhw{[j'],[\lambda']}^{\ell'}}$ may be obtained by applying $\sfsymb^{\ell+\ell'}$ to the rule for $\Gfusion{\tilderhw{[j],[\lambda]}}{\tilderhw{[j'],[\lambda']}}$.  We shall exploit this ``conservation of spectral flow'' to simplify all our Grothendieck fusion rule computations.  In fact, \eqref{eq:shortcut} also extends from $\ell, \ell' \in \ZZ$ to $\ell, \ell' \in \frac{1}{2} \ZZ$.

\begin{theorem} \label{thm:FRstxstv=3}
	Let $\kk$ be admissible with $\vv=3$.  Then for all $\ell, \ell' \in \frac{1}{2} \ZZ$, $[j], [j'] \in \RR / \ZZ$ and $[\lambda], [\lambda'] \in \infwts{\uu}{3} / \ZZ_3$, the Grothendieck fusion rules of the standard $\bpminmod{\uu}{3}$-modules are
	\begin{equation} \label{eq:FRstxstv=3}
    \Gfusion{\tilderhw{[j],[\lambda]}^{\ell}}{\tilderhw{[j'],[\lambda']}^{\ell'}}
    = \sum_{[\lambda''] \in \infwts{\uu}{3} / \ZZ_3} \Wfuscoeff{[\lambda]}{[\lambda']}{[\lambda'']}
    \brac*{\Gr{\tilderhw{[j+j'-4\kappa],[\lambda'']}^{\ell+\ell'+2}} + \Gr{\tilderhw{[j+j'+2\kappa],[\lambda'']}^{\ell+\ell'-1}}}.
	\end{equation}
\end{theorem}
\begin{proof}
	We apply the standard Verlinde formula \eqref{eq:SVF2} with $\ell = \ell' = 0$, substituting \eqref{eq:smatrix} and \eqref{eq:vacsmatrixv=3}:
	\begin{align}
		\fuscoeff{0, [j], [\lambda]}{0, [j'], [\lambda']}{\ell'', [j''], [\lambda'']}
		&= \sum_{[\mu]} \frac{\WSmatrix{[\lambda]}{[\mu]} \WSmatrix{[\lambda']}{[\mu]} \brac*{\WSmatrix{[\lambda'']}{[\mu]}}^*}{\vacWSmatrix{[\mu]}} \\
		&\qquad \cdot \sum_{m \in \ZZ} \ee^{-2 \pi \ii (j+j'-j''-2\kappa\ell'') m}
			\int_{\RR/\ZZ} \ee^{2 \pi \ii (\ell''-1/2) (k-\kappa)} 2 \cos \brac[\big]{3 \pi (k-\kappa)} \, \dd [k] \notag \\
		&= \Wfuscoeff{[\lambda]}{[\lambda']}{[\lambda'']} \delta \brac[\big]{[j''] - [j+j'-2\kappa\ell'']} \brac*{\delta_{\ell'',2} + \delta_{\ell'',-1}} \notag \\
		&= \Wfuscoeff{[\lambda]}{[\lambda']}{[\lambda'']} \brac[\Big]{\delta \brac[\big]{[j''] - [j+j'-4\kappa]} \delta_{\ell'',2}
			+ \delta \brac[\big]{[j''] - [j+j'+2\kappa]} \delta_{\ell'',-1}}. \notag
	\end{align}
	Substituting this result into \eqref{eq:SVF1} and applying $\sfsymb^{\ell+\ell'}$ recovers \eqref{eq:FRstxstv=3}.
\end{proof}

A peculiar feature of \eqref{eq:FRstxstv=3} is the asymmetry in the shifts of the spectral flow indices and $J_0$-eigenvalues.  This is a consequence of the asymmetry in the vacuum S-matrix entries \eqref{eq:vacsmatrixv=3} and derives from the fact that we have chosen an \emt\ $\modify{L}(z)$ that treats the conformal weights of $G^+$ and $G^-$ asymmetrically.  As the Grothendieck fusion rules clearly cannot depend on how we grade our modules, we may use the definition \eqref{eq:deftilderhw} of the $\tilderhw{[j],[\lambda]}$ to recast these rules in terms of the $\twrhw{[j],[\lambda]}$:
\begin{align} \label{eq:FRstxstv=3'}
	&\Gfusion{\sfmod{\ell}{\twrhw{[j],[\lambda]}}}{\sfmod{\ell'}{\twrhw{[j'],[\lambda']}}}
	= \Gfusion{\tilderhw{[j+\kappa],[\lambda]}^{\ell-1/2}}{\tilderhw{[j'+\kappa],[\lambda']}^{\ell'-1/2}} \\
	&\mspace{100mu} = \sum_{[\lambda''] \in \infwts{\uu}{3} / \ZZ_3} \Wfuscoeff{[\lambda]}{[\lambda']}{[\lambda'']}
     \brac*{\Gr{\tilderhw{[j+j'-2\kappa],[\lambda'']}^{\ell+\ell'+1}} + \Gr{\tilderhw{[j+j'+4\kappa],[\lambda'']}^{\ell+\ell'-2}}} \notag \\
  &\mspace{100mu} = \sum_{[\lambda''] \in \infwts{\uu}{3} / \ZZ_3} \Wfuscoeff{[\lambda]}{[\lambda']}{[\lambda'']}
     \brac*{\Gr{\sfmod{\ell+\ell'+3/2}{\twrhw{[j+j'-3\kappa],[\lambda'']}}} + \Gr{\sfmod{\ell+\ell'-3/2}{\twrhw{[j+j'+3\kappa],[\lambda'']}}}}. \notag
\end{align}
Here, the expected symmetry in the spectral flow indices and $J_0$-eigenvalues is restored.

Having established the standard-by-standard Grothendieck fusion rules, it is now a matter of straightforward computation with \cref{cor:chHv=3} to compute the remaining fusion rules.  For this, we recall that every \hw\ $\bpminmod{\uu}{3}$-module is the spectral flow of one whose highest weight corresponds to $\sss=[1,-1,0]$.
\begin{corollary} \label{cor:FR3xstv=3}
	Let $\kk$ be admissible with $\vv=3$.  Then for all $\ell, \ell' \in \frac{1}{2} \ZZ$, $[j'] \in \RR / \ZZ$, $\rrr \in \pwlat{\uu-3}$, and $[\lambda'] \in \infwts{\uu}{3} / \ZZ_3$, we have the following Grothendieck fusion rules:
	\begin{equation} \label{eq:FR3xstv=3}
    \Gfusion{\ihwpar{r_0 & r_1 & r_2 \\ 1 & -1 & 0}^\ell}{\tilderhw{[j'],[\lambda']}^{\ell'}}
    = \sum_{[\lambda''] \in \infwts{\uu}{3} / \ZZ_3} \Wfuscoeff{[\wtnewpar{\rrr,0}]}{[\lambda']}{[\lambda'']}
	    \Gr{\tilderhw{[j(\rrr)+j'],[\lambda'']}^{\ell+\ell'}}.
	\end{equation}
\end{corollary}
\begin{proof}
	Because the standard one-point functions are linearly independent, \eqref{eq:type3characv=3} lifts to the following identity in the Grothendieck group:
	\begin{equation}
		\Gr{\ihwpar{r_0 & r_1 & r_2 \\ 1 & -1 & 0}} = \sum_{n=0}^\infty (-1)^n \brac[\Big]{
			\Gr{\trhwpar{r_1 & r_2 & r_0 \\ 0&0&0}^{9n+1}} - \Gr{\trhwpar{r_0 & r_1 & r_2 \\ 0&0&0}^{9n+4}} + \Gr{\trhwpar{r_2 & r_0 & r_1 \\ 0&0&0}^{9n+7}}}.
	\end{equation}
	As in the proof of \cref{thm:type3modularityv=3}, we rewrite the standard modules in the form required by the standard-by-standard rules:
	\begin{equation} \label{eq:hw=sum_st}
		\Gr{\ihwpar{r_0 & r_1 & r_2 \\ 1 & -1 & 0}} = \sum_{n=0}^\infty (-1)^n \brac[\Big]{
			\Gr{\tilderhw{[j(\rrr)-2\kappa],[\wtnewpar{\rrr,0}]}^{9n+1}} - \Gr{\tilderhw{[j(\rrr)+\kappa-1/2],[\wtnewpar{\rrr,0}]}^{9n+4}}
			+ \Gr{\tilderhw{[j(\rrr)+4\kappa],[\wtnewpar{\rrr,0}]}^{9n+7}}}.
	\end{equation}
	Substituting into the \lhs\ of \eqref{eq:FR3x3v=3} and applying the standard-by-standard rules \eqref{eq:FRstxstv=3}, almost every term cancels and we arrive at the desired answer.
\end{proof}

A more direct, but less instructive, route to these \hw-by-standard Grothendieck fusion rules is to use \cref{thm:stopfS,thm:type3modularityv=3} to directly apply the standard Verlinde formula \eqref{eq:SVF2}.  The ``symmetrised'' version of the Grothendieck fusion rule \eqref{eq:FR3xstv=3} is also easily deduced:
\begin{equation} \label{eq:FR3xstv=3'}
	\Gfusion{\sfmod{\ell}{\ihwpar{r_0 & r_1 & r_2 \\ 1 & -1 & 0}}}{\sfmod{\ell'}{\twrhw{[j'],[\lambda']}}}
	= \sum_{[\lambda''] \in \infwts{\uu}{3} / \ZZ_3} \Wfuscoeff{[\wtnewpar{\rrr,0}]}{[\lambda']}{[\lambda'']}
    \Gr{\sfmod{\ell+\ell'}{\twrhw{[j(\rrr)+j'],[\lambda'']}}}.
\end{equation}

For the \hw-by-\hw\ rules, it will be useful to recall from \cref{thm:W3fusion} that $\Wminmod{\uu}{3}$ fusion coefficients may be expressed in terms of fusion coefficients for the rational $\slthree$ minimal model $\slminmod{\uu}{1} = \sslvoa{\uu-3}$:
\begin{equation} \label{eq:W3fuscoefffactorisationv=3}
	\Wfuscoeff{[\lambda]}{[\lambda']}{[\lambda'']} = \slfuscoeff{\uu-3}{\rrr}{\rrr'}{\rrr''}.
\end{equation}
Here, we should choose representatives $\lambda \in [\lambda]$ so that $\finite{\rrr} = [r_1,r_2] \in \finite{\rlat}$, the root lattice of $\slthree$ (and similarly for the primed representatives).  Since
\begin{equation} \label{eq:repsuniquev=3}
	\finite{\outaut(\rrr)} - \finite{\rrr} = \uu \fwt{1} \bmod{\finite{\rlat}},
\end{equation}
$\uu \notin 3\ZZ$ implies that such representatives always exist and are unique.

\begin{corollary} \label{cor:FR3x3v=3}
	Let $\kk$ be admissible with $\vv=3$.  Then for all $\ell, \ell' \in \frac{1}{2} \ZZ$ and all $\rrr, \rrr' \in \pwlat{\uu-3}$, we have the following Grothendieck fusion rules:
	\begin{equation} \label{eq:FR3x3v=3}
    \Gfusion{\ihwpar{r_0 & r_1 & r_2 \\ 1 & -1 & 0}^{\ell}}{\ihwpar{r'_0 & r'_1 & r'_2 \\ 1 & -1 & 0}^{\ell'}}
    = \sum_{\rrr'' \in \pwlat{\uu-3}} \slfuscoeff{\uu-3}{\rrr}{\rrr'}{\rrr''} \Gr{\ihwpar{r''_0 & r''_1 & r''_2 \\ 1 & -1 & 0}^{\ell+\ell'}}.
	\end{equation}
\end{corollary}
\begin{proof}
	This time, we substitute the primed version of \eqref{eq:hw=sum_st} and apply \eqref{eq:FR3xstv=3} to get
	\begin{multline} \label{eq:3x3exp}
		\Gfusion{\ihwpar{r_0 & r_1 & r_2 \\ 1 & -1 & 0}^\ell}{\ihwpar{r'_0 & r'_1 & r'_2 \\ 1 & -1 & 0}^{\ell'}}
		= \sum_{[\wtnewpar{\rrr'',0}]} \Wfuscoeff{[\wtnewpar{\rrr,0}]}{[\wtnewpar{\rrr',0}]}{[\wtnewpar{\rrr'',0}]} \\
		\cdot \sum_{n=0}^{\infty} (-1)^n \brac[\Big]{
			\Gr{\tilderhw{[j(\rrr)+j(\rrr')-2\kappa],[\wtnewpar{\rrr'',0}]}^{9n+1}} - \Gr{\tilderhw{[j(\rrr)+j(\rrr')+\kappa-1/2],[\wtnewpar{\rrr'',0}]}^{9n+4}}
			+ \Gr{\tilderhw{[j(\rrr)+j(\rrr')+4\kappa],[\wtnewpar{\rrr'',0}]}^{9n+7}}}.
	\end{multline}
	We therefore have to show that for each $[\wtnewpar{\rrr'',0}] \in \infwts{\uu}{3} / \ZZ_3$, the sum over $n$ is $\Gr{\ihwpar{r''_0 & r''_1 & r''_2 \\ 1 & -1 & 0}}$ for some unique $\rrr'' \in \pwlat{\uu-3}$.  There are of course only three candidates for $\rrr''$ as the $\ZZ_3$-orbit is fixed.  However, they are further constrained by requiring that $[j(\rrr'')] = [j(\rrr) + j(\rrr')]$.

	To show that this constraint picks exactly one representative of the $\ZZ_3$-orbit, recall from \eqref{eq:j(r)} that $j(\rrr) \in \frac{1}{3} \ZZ$.  On the other hand, an easy calculation gives $j \brac[\big]{\outaut(\rrr)} - j(\rrr) \in \ZZ + \frac{\uu}{3}$.  	Since $\uu \notin 3\ZZ$, it follows that the three elements of the $\ZZ_3$-orbit have distinct charges modulo $1$.  There thus exists a unique $\rrr''$ that corresponds to a weight in the required $\ZZ_3$-orbit and satisfies $[j(\rrr'')] = [j(\rrr) + j(\rrr')]$.

	It only remains to replace the $\Wminmod{\uu}{3}$ fusion coefficients in \eqref{eq:3x3exp} by $\slminmod{\uu}{1}$ ones.  We may choose the representative $\rrr''$ to satisfy $\finite{\rrr''} \in \finite{\rlat}$, but we cannot assume that $\rrr$ or $\rrr'$ satisfy the analogous constraints.  Thus, \eqref{eq:sl3fuscoeffoutaut} gives
	\begin{equation}
		\Wfuscoeff{[\wtnewpar{\rrr,0}]}{[\wtnewpar{\rrr',0}]}{[\wtnewpar{\rrr'',0}]}
		= \slfuscoeff{\uu-3}{\outaut^m(\rrr)}{\outaut^n(\rrr')}{\rrr''}
		= \slfuscoeff{\uu-3}{\rrr}{\rrr'}{\outaut^{-m-n}(\rrr'')},
	\end{equation}
	for some $m,n \in \ZZ_3$.  This fusion coefficient is zero unless $\finite{\rrr}+\finite{\rrr'}-\finite{\outaut^{-m-n}(\rrr'')} \in \finite{\rlat}$, by the Kac--Walton formula \eqref{eq:KacWalton}.  It can therefore be nonzero for at most one $-m-n \in \ZZ_3$, by \eqref{eq:repsuniquev=3}.  We may therefore replace the sum in \eqref{eq:3x3exp} by one over all $\rrr'' \in \pwlat{\uu-3}$, dropping the constraint $[j(\rrr'')] = [j(\rrr) + j(\rrr')]$, because the $\slminmod{\uu}{1}$ fusion coefficient vanishes when this constraint is not met.
\end{proof}

To illustrate this result, we compute the Grothendieck fusion of $\ihwpar{0 & \uu-3 & 0 \\ 1 & -1 & 0}$ with another \hwm.  Let $\mathsf{0} = [\uu-3,0,0]$, so that $[0,\uu-3,0] = \outaut(\mathsf{0})$.  \eqref{eq:FR3x3v=3} and \eqref{eq:sl3fuscoeffoutaut} now give
\begin{align} \label{eq:simplecurrent}
	\Gfusion{\ihwpar{0 & \uu-3 & 0 \\ 1 & -1 & 0}}{\ihwpar{r'_0 & r'_1 & r'_2 \\ 1 & -1 & 0}}
	&= \sum_{\rrr'' \in \pwlat{\uu-3}} \slfuscoeff{\uu-3}{\outaut(\mathsf{0})}{\rrr'}{\rrr''} \Gr{\ihwpar{r''_0 & r''_1 & r''_2 \\ 1 & -1 & 0}}
	= \sum_{\rrr'' \in \pwlat{\uu-3}} \slfuscoeff{\uu-3}{\mathsf{0}}{\rrr'}{\outaut^{-1}(\rrr'')} \Gr{\ihwpar{r''_0 & r''_1 & r''_2 \\ 1 & -1 & 0}} \\
	&= \sum_{\rrr'' \in \pwlat{\uu-3}} \delta_{\rrr'',\outaut(\rrr')} \Gr{\ihwpar{r''_0 & r''_1 & r''_2 \\ 1 & -1 & 0}}
	= \Gr{\ihwpar{r'_2 & r'_0 & r'_1 \\ 1 & -1 & 0}}. \notag
\end{align}
This, and another nearly identical calculation for $\ihwpar{0 & 0 & \uu-3 \\ 1 & -1 & 0}$, proves the following \lcnamecref{prop:simpcurr}.
\begin{proposition} \label{prop:simpcurr}
	Let $\kk$ be admissible with $\vv=3$.  Then, $\ihwpar{0 & \uu-3 & 0 \\ 1 & -1 & 0}$ and $\ihwpar{0 & 0 & \uu-3 \\ 1 & -1 & 0}$ are simple currents of order $3$, inverse to one another.  Their highest weights (with respect to $J_0$ and $L_0$) are
\begin{equation} \label{eq:simplecurrentdata}
	(j,\Delta) = (+\tfrac{\uu-3}{3}, \tfrac{\uu-3}{2}) \qquad \text{and} \qquad (j,\Delta) = (-\tfrac{\uu-3}{3}, \tfrac{\uu-3}{2}),
\end{equation}
respectively.
\end{proposition}

Another consequence of \eqref{eq:FR3x3v=3} is the following interesting identification, similar to that noted for nondegenerate-admissible-level $\sltwo$ minimal models in \cite[Thm.~16]{CreMod13}.
\begin{proposition} \label{prop:fusringisov=3}
	Let $\kk$ be admissible with $\vv=3$.  Then, the fusion subring of $\bpminmod{\uu}{3}$-modules generated by the $\ihwpar{r_0&r_1&r_2\\1&-1&0} \in \surv{\uu}{3}$, is isomorphic to the fusion ring of the affine \voa\ $\sslvoa{\uu-3}$.
\end{proposition}
\begin{proof}
	By \eqref{eq:FR3x3v=3}, the Grothendieck fusion subring generated by the $\Gr{\ihwpar{r_0&r_1&r_2\\1&-1&0}}$ is clearly isomorphic to the fusion ring of $\sslvoa{\uu-3}$.  An isomorphism consists of identifying $\ihwpar{r_0&r_1&r_2\\1&-1&0}$ with the simple \hw\ $\slminmod{\uu}{1}$-module $\slihw{\rrr}$ whose highest weight is $\rrr = [r_0,r_1,r_2]$.  To show that this gives an isomorphism of fusion rings, we only need to show that the $\ihwpar{r_0&r_1&r_2\\1&-1&0}$ generate a semisimple fusion subring.

	For this, consider the Grothendieck fusion rules \eqref{eq:FR3x3v=3} in which $\rrr = [\uu-4,1,0]$.  The $\slminmod{\uu}{1}$ fusion coefficients that appear may be computed using the Kac--Walton formula \eqref{eq:KacWalton}:
	\begin{equation}
		\fusion{\slihw{[\uu-4,1,0]}}{\slihw{[r'_0,r'_1,r'_2]}}
		\cong \slihw{[r'_0-1,r'_1+1,r'_2]} \oplus \slihw{[r'_0+1,r'_1,r'_2-1]} \oplus \slihw{[r'_0,r'_1-1,r'_2+1]}.
	\end{equation}
	Here, the modules appearing on the \rhs\ are understood to be $0$ if the $r$-labels do not define a weight in $\pwlat{\uu-3}$.  It follows that
	\begin{equation} \label{eq:genFRv=3}
		\Gfusion{\ihwpar{\uu-4 & 1 & 0 \\ 1 & -1 & 0}}{\ihwpar{r'_0 & r'_1 & r'_2 \\ 1 & -1 & 0}}
		= \Gr{\ihwpar{r'_0-1 & r'_1+1 & r'_2 \\ 1 & -1 & 0}} + \Gr{\ihwpar{r'_0+1 & r'_1 & r'_2-1 \\ 1 & -1 & 0}} + \Gr{\ihwpar{r'_0 & r'_1-1 & r'_2+1 \\ 1 & -1 & 0}},
	\end{equation}
	with the same proviso on the modules if the $r$-labels fall outside of $\pwlat{\uu-3}$.

	We now verify that the $\modify{L}_0$-eigenvalues of the \hwvs\ of any two of the modules appearing on the \rhs\ differ by nonintegers.  These modules therefore admit no nonsplit extensions, hence \eqref{eq:genFRv=3} lifts to the genuine fusion rule
	\begin{equation} \label{eq:genFRv=3'}
		\fusion{\ihwpar{\uu-4 & 1 & 0 \\ 1 & -1 & 0}}{\ihwpar{r'_0 & r'_1 & r'_2 \\ 1 & -1 & 0}}
		\cong \ihwpar{r'_0-1 & r'_1+1 & r'_2 \\ 1 & -1 & 0} \oplus \ihwpar{r'_0+1 & r'_1 & r'_2-1 \\ 1 & -1 & 0} \oplus \ihwpar{r'_0 & r'_1-1 & r'_2+1 \\ 1 & -1 & 0}.
	\end{equation}
	This conclusion also holds for $\rrr = [\uu-4,0,1]$.  The proof is then completed by noting that $\ihwpar{\uu-4 & 1 & 0 \\ 1 & -1 & 0}$ and $\ihwpar{\uu-4 & 0 & 1 \\ 1 & -1 & 0}$ generate the desired fusion subring of $\bpminmod{\uu}{3}$, proving that this subring is semisimple.
\end{proof}

\subsection{Examples: $\bpminmod{4}{3}$ and $\bpminmod{5}{3}$} \label{sec:bpu3ex}

We conclude our study of the $\vv=3$ \bp\ minimal models by revisiting two of the examples that we discussed in \cite{FehCla20}.

\subsubsection{$\bpminmod{4}{3}$}

This case was first analysed in \cite{AdaCla19}, where it was identified as the $\ZZ_3$-orbifold of the bosonic ghost system.  Our approach is in many respects the opposite of theirs.  The corresponding minimal model $\bpminmod{4}{3}$ has level $\kk = -\frac{5}{3}$ and central charge $\bpcc{4,3}=-1$ with respect to the conformal vector $L$.  Here, there are $9$ untwisted highest-weight modules that are arranged into $3$ spectral flow orbits as follows:
\begin{equation} \label{eq:43hwmods}
	\begin{gathered}
		\cdots \xrightarrow{\sfsymb} \ihwpar{1&0&0 \\ 0&-1&1} \xrightarrow{\sfsymb} \ihwpar{0&1&0 \\ 1&-1&0}
			\xrightarrow{\sfsymb} \ihwpar{0&0&1 \\ 0&0&0} \xrightarrow{\sfsymb} \cdots , \\
		\cdots \xrightarrow{\sfsymb} \ihwpar{0&0&1 \\ 0&-1&1} \xrightarrow{\sfsymb} \ihwpar{1&0&0 \\ 1&-1&0}
			\xrightarrow{\sfsymb} \ihwpar{0&1&0 \\ 0&0&0} \xrightarrow{\sfsymb} \cdots , \\
		\cdots \xrightarrow{\sfsymb} \ihwpar{0&1&0 \\ 0&-1&1} \xrightarrow{\sfsymb} \ihwpar{0&0&1 \\ 1&-1&0}
			\xrightarrow{\sfsymb} \ihwpar{1&0&0 \\ 0&0&0} \xrightarrow{\sfsymb} \cdots . \\
	\end{gathered}
\end{equation}
The corresponding highest weights $(j,\Delta)$ are easily computed using \cref{thm:classuntwi}:
\begin{equation}
	\begin{gathered}
		\cdots \xrightarrow{\sfsymb} (\tfrac{4}{9},\tfrac{1}{9}) \xrightarrow{\sfsymb} (\tfrac{1}{3},\tfrac{1}{2})
			\xrightarrow{\sfsymb} (-\tfrac{7}{9},\tfrac{5}{18}) \xrightarrow{\sfsymb} \cdots , \\
		\cdots \xrightarrow{\sfsymb} (\tfrac{1}{9},-\tfrac{1}{18}) \xrightarrow{\sfsymb} (0,0)
			\xrightarrow{\sfsymb} (-\tfrac{1}{9},-\tfrac{1}{18}) \xrightarrow{\sfsymb} \cdots , \\
		\cdots \xrightarrow{\sfsymb} (\tfrac{7}{9},\tfrac{5}{18}) \xrightarrow{\sfsymb} (-\tfrac{1}{3},\tfrac{1}{2})
			\xrightarrow{\sfsymb} (-\tfrac{4}{9},\tfrac{1}{9}) \xrightarrow{\sfsymb} \cdots . \\
	\end{gathered}
\end{equation}
The vacuum module and simple currents are $\ihw{\kk \fwt{0}} = \ihwpar{1&0&0 \\ 1&-1&0}$, $\ihwpar{0&1&0 \\ 1&-1&0}$ and $\ihwpar{0&0&1 \\ 1&-1&0}$.  This information completely determines the fusion rules of the \hwms, for example
\begin{equation}
	\ihwpar{1&0&0 \\ 0&-1&1}{} \fuse \ihwpar{1&0&0 \\ 0&0&0}
	\cong \sfmod{-1}{\ihwpar{0&1&0 \\ 1&-1&0}} \fuse \sfmod{}{\ihwpar{0&0&1 \\ 1&-1&0}}
	\cong \ihwpar{0&1&0 \\ 1&-1&0} \fuse \ihwpar{0&0&1 \\ 1&-1&0}
	\cong \ihwpar{1&0&0 \\ 1&-1&0}.
\end{equation}

We turn next to the (twisted) \rhw\ $\bpminmod{4}{3}$-modules $\twrhw{[j],[\lambda]}$, where $[\lambda] \in \infwts{4}{3} / \ZZ_3$ and $[j] \in \RR/\ZZ$.  As $\infwts{4}{3} / \ZZ_3$ has only one element ($\Wminmod{4}{3} \cong \CC$), we shall simplify the notation by dropping the dependence on $[\lambda]$: $\twrhw{[j]} \equiv \twrhw{[j],[\lambda]}$. The condition on $[j]$ for $\twrhw{[j]}$ to be simple is $j \ne \frac{1}{6}, \frac{1}{2}, \frac{5}{6} \pmod{1}$, by \cref{thm:classtwi}.  When $j$ assumes one of these values, $\twrhw{[j]}$ is nonsemisimple with $G^-_0$ acting injectively.

As every \hw\ $\bpminmod{4}{3}$-module is the spectral flow of the vacuum module or a simple current, the ``highest-weight by relaxed'' Grothendieck fusion rules are fixed by computing
\begin{equation}
	\Gfusion{\ihwpar{0&1&0 \\ 1&-1&0}}{\twrhw{[j']}} = \Gr{\twrhw{[j'+\uu/3]}},
\end{equation}
using \eqref{eq:FR3xstv=3'}.  Note that if the relaxed module on the \lhs\ is simple, then so is that on the right.  We therefore obtain the following fusion rule:
\begin{equation} \label{eq:FRscxst43}
	\fusion{\ihwpar{0&1&0 \\ 1&-1&0}}{\twrhw{[j']}} \cong \twrhw{[j'+\uu/3]}, \qquad j' \notin \tfrac{1}{3} \ZZ + \tfrac{1}{6}.
\end{equation}

Somewhat more interesting are the ``relaxed by relaxed'' rules.  Specialising \eqref{eq:FRstxstv=3'} results in
\begin{equation}
	\Gfusion{\twrhw{[j]}}{\twrhw{[j']}} = \Gr{\sfmod{3/2}{\twrhw{[j+j'+1/6]}}} + \Gr{\sfmod{-3/2}{\twrhw{[j+j'-1/6]}}}.
\end{equation}
Comparing conformal weights for the summands on the \rhs, using \eqref{eq:defsf} and \cref{thm:classtwi}, we conclude that the corresponding fusion product is generically semisimple:
\begin{equation}
	\fusion{\twrhw{[j]}}{\twrhw{[j']}} \cong \sfmod{3/2}{\twrhw{[j+j'+1/6]}} \oplus \sfmod{-3/2}{\twrhw{[j+j'-1/6]}}, \qquad j+j' \notin \tfrac{1}{3} \ZZ.
\end{equation}
When $j+j' \in \frac{1}{3} \ZZ$, we conjecture that the fusion product is nonsemisimple.  In fact, we expect that the results are staggered $\bpminmod{4}{3}$-modules that serve as projective covers of the vacuum module and the simple currents.  Exploring this conjecture is however well beyond the scope of this investigation.

The simple current extension of $\bpminmod{4}{3}$ corresponding to $\ihwpar{0&1&0 \\ 1&-1&0}$ and $\ihwpar{0&0&1 \\ 1&-1&0}$ is the \voa\ $\VOA{B}$ whose vacuum module decomposes as
\begin{equation}
	\VOA{B} = \ihwpar{0&1&0\\1&-1&0} \oplus \ihwpar{1&0&0\\1&-1&0} \oplus \ihwpar{0&0&1\\1&-1&0}.
\end{equation}
It is easy to check that the field $\beta(z)$ of weight $(\frac{1}{3},\frac{1}{2})$ and the field $\gamma(z)$ of weight $(-\frac{1}{3},\frac{1}{2})$ generate a copy of the bosonic ghosts \voa\ in $\VOA{B}$.  In fact, as the generating fields of $\bpminmod{4}{3}$ can be expressed in terms of $\beta$ and $\gamma$, see \cite[Eq.~(5.7)]{FehCla20} for example, $\VOA{B}$ is the bosonic ghosts \voa.

Note that the simple current orbits
\begin{equation}
	\ihwpar{1&0&0 \\ 0&-1&1} \oplus \ihwpar{0&0&1 \\ 0&-1&1} \oplus \ihwpar{0&1&0 \\ 0&-1&1} \quad \text{and} \quad
	\ihwpar{0&0&1 \\ 0&0&0} \oplus \ihwpar{0&1&0 \\ 0&0&0} \oplus \ihwpar{1&0&0 \\ 0&0&0}
\end{equation}
are not $\VOA{B}$-modules as they are not $\frac{1}{2} \ZZ$-graded by conformal weight.  Indeed, an easy calculation shows that $\sfmod{\ell}{\VOA{B}}$ is an untwisted $\VOA{B}$-module (is $\frac{1}{2} \ZZ$-graded) if $\ell \in 3\ZZ$ and is a twisted $\VOA{B}$-module (is $\ZZ$-graded) if $\ell \in 3(\ZZ+\frac{1}{2})$.  This reflects the fact that the natural unit of ghost spectral flow is $\sfsymb^3$, not $\sfsymb$.

To obtain additional $\VOA{B}$-modules, we consider the spectral flow orbit of the twisted \rhw\ $\bpminmod{4}{3}$-modules.  Indeed,
\begin{equation}
	\Mod{B}_{\llbracket j \rrbracket} = \twrhw{[j-1/3]} \oplus \twrhw{[j]} \oplus \twrhw{[j+1/3]}, \quad \llbracket j \rrbracket \in \RR \big/ \tfrac{1}{3} \ZZ,
\end{equation}
is such an orbit, by \eqref{eq:FRscxst43}, and conformal weight considerations show that it is a simple twisted $\VOA{B}$-module for all $\llbracket j \rrbracket \ne \llbracket \frac{1}{6} \rrbracket$.

Fusion rules for these $\VOA{B}$-modules may be obtained from the $\bpminmod{4}{3}$ fusion rules by induction \cite{RidVer14}, see also \cite{CreTen17}.  Those involving the simple current extension $\VOA{B}$ (and its spectral flows) are obvious, so we compute only
\begin{equation}
	\fusion{\Mod{B}_{\llbracket j \rrbracket}}{\Mod{B}_{\llbracket j' \rrbracket}}
	\cong \sfmod{3/2}{\Mod{B}_{\llbracket j+j'+1/6 \rrbracket}} \oplus \sfmod{-3/2}{\Mod{B}_{\llbracket j+j'-1/6 \rrbracket}}, \qquad
	\llbracket j+j' \rrbracket \ne \llbracket 0 \rrbracket.
\end{equation}
This fusion rule is of course identical, up to rescaling charges and spectral flow indices by a factor of $3$, to the bosonic ghosts fusion rule computed in \cite[App.~A]{RidBos14}, see also \cite{AdaFus19,AllBos20}.

\subsubsection{$\bpminmod{5}{3}$}

The next least complicated example has $\kk = -\frac{4}{3}$ and $\bpcc{5,3} = \frac{3}{5}$. This minimal model has $18$ simple untwisted \hwms, arranged into $6$ spectral flow orbits. In addition to these, there are two families of twisted \rhwms:
\begin{equation}
	\twrhwpar{[j]}{2&0&0 \\ 0&0&0}, \quad \text{and} \quad \twrhwpar{[j]}{1&1&0 \\ 0&0&0}.
\end{equation}
The elements of both families are simple unless $j = \frac{1}{6}, \frac{1}{2}, \frac{5}{6} \pmod{1}$.  The conformal weights of the \rhwvs\ are $\frac{1}{8}$ and $-\frac{3}{40}$, respectively.

It was conjectured in \cite{FehCla20} that the modules $\ihwpar{0&2&0 \\ 1&-1&0}$ and $\ihwpar{0&0&2 \\ 1&-1&0}$ are simple currents of order $3$, despite not being spectral flows of the vacuum module $\ihwpar{2&0&0 \\ 1&-1&0}$.  This conjecture is now confirmed by \cref{prop:simpcurr}.  The consequent simple current extension
\begin{equation}
	\VOA{C} = \ihwpar{0&2&0 \\ 1&-1&0} \oplus \ihwpar{2&0&0 \\ 1&-1&0} \oplus \ihwpar{0&0&2 \\ 1&-1&0}
\end{equation}
was then noted to be isomorphic to the minimal \qhr\ of $\saffvoa{-3/2}{\mathfrak{g}_2}$.

Up to spectral flow and simple currents, there are thus two representative \hw\ $\bpminmod{5}{3}$-modules which we may choose to be the vacuum module and $\ihwpar{1&1&0\\1&-1&0}$.  As far as \hwms\ go, it is therefore enough to give the fusion rule for the latter with itself:
\begin{equation}
	\fusion{\ihwpar{1&1&0\\1&-1&0}}{\ihwpar{1&1&0\\1&-1&0}} \cong \ihwpar{0&2&0\\1&-1&0} \oplus \ihwpar{1&0&1\\1&-1&0}.
\end{equation}
Here, it is helpful to recall the well known fusion rules of the Virasoro minimal model $\virminmod{2}{5}$ which is isomorphic to $\Wminmod{5}{3}$.  Similarly, it is enough to specify two \hw-by-relaxed fusion rules, namely
\begin{equation}
	\begin{aligned}
		\fusion{\ihwpar{1&1&0\\1&-1&0}}{\twrhwpar{[j']}{2&0&0 \\ 0&0&0}} &\cong \twrhwpar{[j'+1/3]}{1&1&0 \\ 0&0&0} \\ \text{and} \quad
		\fusion{\ihwpar{1&1&0\\1&-1&0}}{\twrhwpar{[j]}{1&1&0 \\ 0&0&0}} &\cong \twrhwpar{[j'+1/3]}{2&0&0 \\ 0&0&0} \oplus \twrhwpar{[j'+1/3]}{1&1&0 \\ 0&0&0},
	\end{aligned}
	\qquad j' \notin \tfrac{1}{3} \ZZ + \tfrac{1}{6},
\end{equation}
and three relaxed-by-relaxed rules:
\begin{equation}
	\begin{aligned}
		\fusion{\twrhwpar{[j]}{2&0&0 \\ 0&0&0}}{\twrhwpar{[j']}{2&0&0 \\ 0&0&0}}
		&\cong \twrhwpar{[j+j'-1/6]}{2&0&0 \\ 0&0&0}^{3/2} \oplus \twrhwpar{[j+j'+1/6]}{2&0&0 \\ 0&0&0}^{-3/2}, \\
		\fusion{\twrhwpar{[j]}{2&0&0 \\ 0&0&0}}{\twrhwpar{[j']}{1&1&0 \\ 0&0&0}}
		&\cong \twrhwpar{[j+j'-1/6]}{1&1&0 \\ 0&0&0}^{3/2} \oplus \twrhwpar{[j+j'+1/6]}{1&1&0 \\ 0&0&0}^{-3/2}, \\
		\fusion{\twrhwpar{[j]}{1&1&0 \\ 0&0&0}}{\twrhwpar{[j']}{1&1&0 \\ 0&0&0}}
		&\cong \twrhwpar{[j+j'-1/6]}{2&0&0 \\ 0&0&0}^{3/2} \oplus \twrhwpar{[j+j'-1/6]}{1&1&0 \\ 0&0&0}^{3/2} \\
		&\qquad \oplus \twrhwpar{[j+j'+1/6]}{2&0&0 \\ 0&0&0}^{-3/2} \oplus \twrhwpar{[j+j'+1/6]}{1&1&0 \\ 0&0&0}^{-3/2}.
	\end{aligned}
	\qquad j+j' \notin \tfrac{1}{3} \ZZ.
\end{equation}

For completeness, we briefly report the easily derived fusion rules of the simple current extension $\VOA{C}$.  Let
\begin{equation}
	\begin{aligned}
		\Mod{C} &= \ihwpar{1&1&0 \\ 1&-1&0} \oplus \ihwpar{0&1&1 \\ 1&-1&0} \oplus \ihwpar{1&0&1 \\ 1&-1&0}, \\
		\Mod{C}_{\llbracket j \rrbracket} \rspar{2&0&0\\0&0&0}
		&= \twrhwpar{[j-1/3]}{2&0&0\\0&0&0} \oplus \twrhwpar{[j]}{2&0&0\\0&0&0} \oplus \twrhwpar{[j+1/3]}{2&0&0\\0&0&0} \\ \text{and} \quad
		\Mod{C}_{\llbracket j \rrbracket} \rspar{1&1&0\\0&0&0}
		&= \twrhwpar{[j-1/3]}{1&1&0\\0&0&0} \oplus \twrhwpar{[j]}{1&1&0\\0&0&0} \oplus \twrhwpar{[j+1/3]}{1&1&0\\0&0&0}.
	\end{aligned}
\end{equation}
As before, $\llbracket j \rrbracket \in \RR \big/ \frac{1}{3} \ZZ$ and the $\Mod{C}_{\llbracket j \rrbracket} \rspar{2&0&0\\0&0&0}$ and $\Mod{C}_{\llbracket j \rrbracket} \rspar{1&1&0\\0&0&0}$ are simple (and relaxed) $\VOA{C}$-modules unless $\llbracket j \rrbracket = \llbracket \frac{1}{6} \rrbracket$.  The fusion rules are then
\begin{subequations}
	\begin{gather}
		\fusion{\Mod{C}}{\Mod{C}} \cong \VOA{C} \oplus \Mod{C}, \\
		\fusion{\Mod{C}}{\Mod{C}_{\llbracket j \rrbracket} \rspar{2&0&0\\0&0&0}}
		\cong \Mod{C}_{\llbracket j \rrbracket} \rspar{1&1&0\\0&0&0}, \quad
		\fusion{\Mod{C}}{\Mod{C}_{\llbracket j \rrbracket}} \rspar{1&1&0\\0&0&0}
		\cong \Mod{C}_{\llbracket j \rrbracket} \rspar{2&0&0\\0&0&0} \oplus \Mod{C}_{\llbracket j \rrbracket} \rspar{1&1&0\\0&0&0},
		\qquad \llbracket j \rrbracket \ne \llbracket \tfrac{1}{6} \rrbracket,
	\end{gather}
	and
	\begin{equation}
		\begin{aligned}
			\fusion{\Mod{C}_{\llbracket j \rrbracket} \rspar{2&0&0\\0&0&0}}{\Mod{C}_{\llbracket j' \rrbracket} \rspar{2&0&0\\0&0&0}}
			&\cong \Mod{C}_{\llbracket j+j'-1/6 \rrbracket} \rspar{2&0&0\\0&0&0}^{3/2} \oplus \Mod{C}_{\llbracket j+j'+1/6 \rrbracket} \rspar{2&0&0\\0&0&0}^{-3/2}, \\
			\fusion{\Mod{C}_{\llbracket j \rrbracket} \rspar{2&0&0\\0&0&0}}{\Mod{C}_{\llbracket j' \rrbracket} \rspar{1&1&0\\0&0&0}}
			&\cong \Mod{C}_{\llbracket j+j'-1/6 \rrbracket} \rspar{1&1&0\\0&0&0}^{3/2} \oplus \Mod{C}_{\llbracket j+j'+1/6 \rrbracket} \rspar{1&1&0\\0&0&0}^{-3/2}, \\
			\fusion{\Mod{C}_{\llbracket j \rrbracket} \rspar{1&1&0\\0&0&0}}{\Mod{C}_{\llbracket j' \rrbracket} \rspar{1&1&0\\0&0&0}}
			&\cong \Mod{C}_{\llbracket j+j'-1/6 \rrbracket} \rspar{2&0&0\\0&0&0}^{3/2} \oplus \Mod{C}_{\llbracket j+j'-1/6 \rrbracket} \rspar{1&1&0\\0&0&0}^{3/2} \\
			&\mspace{30mu} \oplus \Mod{C}_{\llbracket j+j'+1/6 \rrbracket} \rspar{2&0&0\\0&0&0}^{-3/2}
				\oplus \Mod{C}_{\llbracket j+j'+1/6 \rrbracket} \rspar{1&1&0\\0&0&0}^{-3/2},
		\end{aligned}
		\qquad \llbracket j+j' \rrbracket \ne \llbracket 0 \rrbracket.
	\end{equation}
\end{subequations}

\section{General \bp\ minimal models} \label{sec:bpuv}

We now turn to the generalisation of our $\vv=3$ modularity results to $\vv>3$.  As mentioned previously, one difficulty is the presence of type-$1$ and -$2$ highest-weight modules: for $\vv>3$, there are always \hwms\ of every type.  However, the same strategy as before, constructing resolutions that express the one-point functions of the \hwms\ in terms of those of the standard modules, still enables the computation of modular transformations and (Grothendieck) fusion rules.  However, the technical complexity of the computations increases considerably and so we shall not be exhaustive in our investigations.

\subsection{Resolutions} \label{sec:bpres}

We begin by generalising the type-$3$ resolutions of \cref{prop:resv=3} to all types.  Since spectral flow is an exact functor, it suffices to choose a representative \hw\ $\bpminmoduv$-module in each orbit.  We therefore take $\wtnewpar{\rrr,\sss} \in \survuv$ to be as in \cref{cor:typereps}, thus the leftmost in its orbit (as pictured in \cref{fig:sforbits}).  Then, $s_2 \ne 0$ and \cref{prop:ses} gives the following short exact sequence:
\begin{equation} \label{eq:newses}
	\dses{\ihwpar{r_0&r_1&r_2 \\ s_0&s_1+1&s_2-1}^1}{}{\trhwpar{r_0&r_1&r_2 \\ s_0&s_1+1&s_2-1}}{}{\ihwnewparrs}.
\end{equation}
Note that $\ihwpar{r_0&r_1&r_2 \\ s_0&s_1+1&s_2-1}$ is rightmost in its orbit.  As long as $s_2 \ne 1$, it is type-$1$ and thus also leftmost.  We may therefore splice \eqref{eq:newses} with the corresponding sequence for $\ihwpar{r_0&r_1&r_2 \\ s_0&s_1+1&s_2-1}^1$ and repeat until the $s_2$ label has decreased to $0$ (and the \hwm\ is no longer type-$1$).  The result is the exact sequence
\begin{equation} \label{eq:partialreso}
	0 \to
	\ihwpar{r_0&r_1&r_2 \\ s_0&s_1+s_2&0}^{s_2} \to
	\trhwpar{r_0&r_1&r_2 \\ s_0&s_1+s_2&0}^{s_2-1} \to \dots \to
	\trhwpar{r_0&r_1&r_2 \\ s_0&s_1+2&s_2-2}^1 \to
	\trhwpar{r_0&r_1&r_2 \\ s_0&s_1+1&s_2-1} \to
	\ihwnewparrs \to 0.
\end{equation}
Resolving \hwms\ therefore reduces to resolving those with $s_2=0$.

Comparing again with \cref{fig:sforbits}, we see that $\ihwpar{r_0&r_1&r_2 \\ s_0&s_1+s_2&0} = \ihwpar{r_0&r_1&r_2 \\ s_0&\vv-3-s_0&0}$ is type-$2$, if $s_0 \ne 0$, and type-$3$, if $s_0=0$.  Being rightmost in its orbit, this module is therefore obtained from the leftmost by applying one or two units of spectral flow, respectively.  Appealing to \cref{whenarespecflowshw}, we have
\begin{equation}
	\ihwpar{r_0&r_1&r_2 \\ s_0&s_1+s_2&0}^{s_2} \cong
	\begin{cases*}
		 \ihwpar{r_1&r_2&r_0 \\ \vv-2-s_0&-1&s_0}^{s_2+1}&if $s_0 \ne 0$, \\
		 \ihwpar{r_2&r_0&r_1 \\ 0&-1&\vv-2}^{s_2+2}&if $s_0 = 0$.
	\end{cases*}
\end{equation}
These are now leftmost in their orbits, hence we can iteratively splice versions of \eqref{eq:partialreso} to obtain the desired resolution.

Clearly, if we start with $s_0 = 0$, then all the sequences \eqref{eq:partialreso} to be spliced together will have $s_0 = 0$ and the resolutions will only involve type-$1$ and type-$3$ highest weights.  If we start with $s_0 \ne 0$, then the sequences being spliced will likewise all have $s_0 \ne 0$ because $\ihwnewparrs$ was chosen to be leftmost in its orbit and so we cannot have $s_0 = \vv-2$.  In this case, the resolutions will only involve type-$1$ and type-$2$ highest weights.

We first record the type-$3$ resolution obtained when $s_0=0$.
\begin{proposition} \label{prop:type3reso}
	Let $\kk$ be nondegenerate-admissible.  Then, the $\bpminmoduv$-module $\ihwpar{r_0&r_1&r_2 \\ 0&s_1&s_2}$ (chosen as in \cref{fig:sforbits} to be leftmost in its spectral flow orbit) is resolved by the nonsimple standard modules as follows:
	\begin{align} \label{eq:type3reso}
    \cdots &\lra \trhwpar{r_0&r_1&r_2 \\ 0&\vv-3&0}^{s_2-1+3\vv} \lra \dots \lra
    \trhwpar{r_0&r_1&r_2 \\ 0&1&\vv-4}^{s_2+3+2\vv} \lra \trhwpar{r_0&r_1&r_2 \\ 0&0&\vv-3}^{s_2+2+2\vv} \\
    &\lra \trhwpar{r_1&r_2&r_0 \\ 0&\vv-3&0}^{s_2-1+2\vv} \lra \dots \lra
    \trhwpar{r_1&r_2&r_0 \\ 0&1&\vv-4}^{s_2+3+\vv} \lra \trhwpar{r_1&r_2&r_0 \\ 0&0&\vv-3}^{s_2+2+\vv} \notag \\
    &\lra \trhwpar{r_2&r_0&r_1 \\ 0&\vv-3&0}^{s_2-1+\vv} \lra \dots \lra
    \trhwpar{r_2&r_0&r_1 \\ 0&1&\vv-4}^{s_2+3} \lra \trhwpar{r_2&r_0&r_1 \\ 0&0&\vv-3}^{s_2+2} \notag \\
    &\lra \trhwpar{r_0&r_1&r_2 \\ 0&\vv-3&0}^{s_2-1} \lra \dots \lra
    \trhwpar{r_0&r_1&r_2 \\ 0&s_1+2&s_2-2}^1 \lra \trhwpar{r_0&r_1&r_2 \\ 0&s_1+1&s_2-1} \lra
    \ihwpar{r_0&r_1&r_2 \\ 0&s_1&s_2} \lra 0. \notag
	\end{align}
\end{proposition}
\noindent It is easy to check that this reduces to the first resolution of \cref{prop:resv=3} when $\vv=3$.  Note that a resolution for vacuum module $\ihw{\kk\fwt{0}} = \ihwpar{\uu-3&0&0\\\vv-2&-1&0} = \ihwpar{0&0&\uu-3\\0&-1&\vv-2}^1$ follows from \eqref{eq:type3reso} since spectral flow is exact.

The resolution for $s_0 \ne 0$ is somewhat more complicated and has no $\vv=3$ analogue.
\begin{proposition} \label{prop:reso}
	Let $\kk$ be nondegenerate-admissible with $\vv > 3$ and suppose that $s_0 \ne 0$.  Then, the $\bpminmoduv$-module $\ihwnewparrs$ (chosen as in \cref{fig:sforbits} to be leftmost in its spectral flow orbit) is resolved by the nonsimple standard modules as follows:
	\begin{align} \label{eq:reso}
	  \cdots &\lra \trhwpar{r_0&r_1&r_2 \\ s_0&\vv-3-s_0&0}^{3\vv+s_2-1} \lra \dots \lra
	  \trhwpar{r_0&r_1&r_2 \\ s_0&1&\vv-4-s_0}^{3\vv-s_1} \lra \trhwpar{r_0&r_1&r_2 \\ s_0&0&\vv-3-s_0}^{3\vv-s_1-1} \\
	  &\mspace{-20mu} \lra \trhwpar{r_2&r_0&r_1 \\ \vv-2-s_0&s_0-1&0}^{3\vv-s_1-3} \lra \dots \lra
	  \trhwpar{r_2&r_0&r_1 \\ \vv-2-s_0&1&s_0-2}^{2\vv+s_2+2} \lra \trhwpar{r_2&r_0&r_1 \\ \vv-2-s_0&0&s_0-1}^{2\vv+s_2+1} \notag \\
	  &\lra \trhwpar{r_1&r_2&r_0 \\ s_0&\vv-3-s_0&0}^{2\vv+s_2-1} \lra \dots \lra
	  \trhwpar{r_1&r_2&r_0 \\ s_0&1&\vv-4-s_0}^{2\vv-s_1} \lra \trhwpar{r_1&r_2&r_0 \\ s_0&0&\vv-3-s_0}^{2\vv-s_1-1} \notag \\
	  &\mspace{-20mu} \lra \trhwpar{r_0&r_1&r_2 \\ \vv-2-s_0&s_0-1&0}^{2\vv-s_1-3} \lra \dots \lra
	  \trhwpar{r_0&r_1&r_2 \\ \vv-2-s_0&1&s_0-2}^{\vv+s_2+2} \lra \trhwpar{r_0&r_1&r_2 \\ \vv-2-s_0&0&s_0-1}^{\vv+s_2+1} \notag\\
	  &\lra \trhwpar{r_2&r_0&r_1 \\ s_0&\vv-3-s_0&0}^{\vv+s_2-1} \lra \dots \lra
	  \trhwpar{r_2&r_0&r_1 \\ s_0&1&\vv-4-s_0}^{\vv-s_1} \lra \trhwpar{r_2&r_0&r_1 \\ s_0&0&\vv-3-s_0}^{\vv-s_1-1} \notag \\
	  &\mspace{-20mu} \lra \trhwpar{r_1&r_2&r_0 \\ \vv-2-s_0&s_0-1&0}^{\vv-s_1-3} \lra \dots \lra
	  \trhwpar{r_1&r_2&r_0 \\ \vv-2-s_0&1&s_0-2}^{s_2+2} \lra \trhwpar{r_1&r_2&r_0 \\ \vv-2-s_0&0&s_0-1}^{s_2+1} \notag \\
	  &\lra \trhwpar{r_0&r_1&r_2 \\ s_0&\vv-3-s_0&0}^{s_2-1} \lra \dots \lra
	  \trhwpar{r_0&r_1&r_2 \\ s_0&s_1+2&s_2-2}^1 \lra \trhwpar{r_0&r_1&r_2 \\ s_0&s_1+1&s_2-1} \lra
	  \ihwnewparrs \lra 0. \notag
	\end{align}
\end{proposition}
\noindent Happily, the resolution \eqref{eq:type3reso} may be recovered from \eqref{eq:reso} by setting $s_0=0$ (and thus deleting every second line).

With these resolutions, one-point functions for \hw\ $\bpminmoduv$-modules and their spectral flows are easy to write down as alternating sums of characters of nonsimple standard modules.
\begin{corollary} \label{cor:onepoint}
	For $\kk$ nondegenerate-admissible, the one-point function of $\ihwnewparrs$ (chosen as in \cref{fig:sforbits} to be leftmost in its spectral flow orbit) is given by
	\begin{multline} \label{eq:onepoint}
		\tchnoargs{\ihwnewparrs}
		= \sum_{m=0}^{s_2-1} (-1)^m \tchnoargs{\trhwpar{r_0&r_1&r_2 \\ s_0&s_1+m+1&s_2-m-1}^m} \\
		+ \sum_{n=0}^{\infty} \sum_{m=0}^{s_0-1} (-1)^{s_2+m+n\vv} \biggl( \tchnoargs{\trhwpar{r_1&r_2&r_0 \\ \vv-2-s_0&m&s_0-m-1}^{m+3n\vv+s_2+1}}
			+ (-1)^{\vv} \tchnoargs{\trhwpar{r_0&r_1&r_2 \\ \vv-2-s_0&m&s_0-m-1}^{m+(3n+1)\vv+s_2+1}} \biggr. \\
		\biggl. + \tchnoargs{\trhwpar{r_2&r_0&r_1 \\ \vv-2-s_0&m&s_0-m-1}^{m+(3n+2)\vv+s_2+1}} \biggr)
			- \sum_{n=1}^{\infty} \sum_{m=0}^{\vv-3-s_0} (-1)^{s_1+m+n\vv} \biggl( \tchnoargs{\trhwpar{r_2&r_0&r_1 \\ s_0&m&\vv-3-s_0-m}^{m+(3n-2)\vv-s_1-1}} \biggr. \\
		\biggl. + (-1)^{\vv} \tchnoargs{\trhwpar{r_1&r_2&r_0 \\ s_0&m&\vv-3-s_0-m}^{m+(3n-1)\vv-s_1-1}}
			+ \tchnoargs{\trhwpar{r_0&r_1&r_2 \\ s_0&m&\vv-3-s_0-m}^{m+3n\vv-s_1-1}} \biggr).
	\end{multline}
\end{corollary}

This rather formidable formula simplifies somewhat for type-$3$ modules.  As in the previous \lcnamecref{sec:bpu3}, we shall find it convenient to choose the middle module (as depicted in \cref{fig:sforbits}) as the representative of the type-$3$ spectral flow orbits.  In particular, the vacuum module is type-$3$ and of this form.
\begin{corollary} \label{cor:type3onepoint}
	For $\kk$ be nondegenerate-admissible, the one-point function of $\ihwpar{r_0&r_1&r_2 \\ \vv-2&-1&0}$ is given by
	\begin{align} \label{eq:type3onepoint}
		\tchnoargs{\ihwpar{r_0&r_1&r_2 \\ \vv-2&-1&0}}
		&= \sum_{n=0}^\infty \sum_{m=0}^{\vv-3} (-1)^{m+n\vv} \Bigl(
			\tchnoargs{\trhwpar{r_1&r_2&r_0 \\ 0&m&\vv-3-m}^{m+3n\vv+1}} \Bigr. \\
		&\qquad \Bigl. + (-1)^{\vv} \tchnoargs{\trhwpar{r_0&r_1&r_2 \\0&m&\vv-3-m}^{m+(3n+1)\vv+1}}
			+ \tchnoargs{\trhwpar{r_2&r_0&r_1 \\0&m&\vv-3-m}^{m+(3n+2)\vv+1}} \Bigr) \notag \\
		&= \sum_{n=0}^\infty \sum_{m=0}^{\vv-3} \sum_{i=0}^2 (-1)^{m+(n+i)\vv} \tchnoargs{\trhwnewpar{\outaut^{i-1}(\rrr),[0,m,\vv-3-m]}^{m+(3n+i)\vv+1}}. \notag
	\end{align}
\end{corollary}

Before turning to the modular transforms of the type-$3$ one-point functions, we generalise the notation \eqref{eq:j(r)} to $\vv \ge 3$:
\begin{equation} \label{eq:j(r,s)}
	j(\rrr,\sss) = j \brac[\big]{\wtnewparrs} = \tfrac{1}{3} (r_1-r_2) - \tfrac{\uu}{3\vv} (s_1-s_2+1), \quad
	j^{\twist}(\rrr,\sss) = j(\rrr,\sss) + \kappa = \tfrac{1}{3} (r_1-r_2) - \tfrac{\uu}{3\vv} (s_1-s_2) - \tfrac{1}{2}.
\end{equation}
For each $\lambda = \wtpar{r_0&r_1&r_2 \\ \vv-2&-1&0} \in \surv{\uu}{\vv}$, we define another convenient notation $\underline{\lambda} = \lambda + \fracuv \brac{\fwt{0}-\fwt{1}}$, noting that
\begin{equation}
	\lambda = \wtpar{r_0&r_1&r_2 \\ \vv-2&-1&0} \in \surv{\uu}{\vv} \quad \Rightarrow \quad
	\underline{\lambda} = \wtpar{r_0&r_1&r_2 \\ \vv-3&0&0} \in \infwts{\uu}{\vv}.
\end{equation}
With this, the S-transforms of the type-$3$ one-point functions are as follows.
\begin{theorem} \label{thm:type3modularity}
	Let $\kk$ be nondegenerate-admissible and let $\lambda = \wtpar{r_0&r_1&r_2 \\ \vv-2&-1&0} \in \surv{\uu}{\vv}$.  Then for all $\ell \in \ZZ$, the S-transform of the one-point function of $\sfmod{\ell}{\ihw{\lambda}} = \ihwpar{r_0&r_1&r_2 \\ \vv-2&-1&0}^\ell$ is given by
	\begin{equation}
		\mods \set*{\tchnoargs{\sfmod{\ell}{\ihw{\lambda}}}}
		= \frac{\abs{\tau}}{-\ii \tau} \sum_{\ell' \in \ZZ} \int_{\RR/\ZZ} \sum_{[\lambda'] \in \infwts{\uu}{\vv} / \ZZ_3}
			\Smatrix{\ell,\lambda}{\ell',[j'],[\lambda']} \tchnoargs{\tilderhw{[j'],[\lambda']}^{\ell'}} \, \dd [j'],
	\end{equation}
	where the entries of the ``\hw\ S-matrix" are given by
	\begin{equation} \label{eq:type3Smatrix}
		\Smatrix{\ell,\lambda}{\ell',[j'],[\lambda']}
		= \WSmatrix{[\underline{\lambda}]}{[\lambda']}
			\frac{\ee^{-2 \pi \ii \brac[\big]{2 \kappa (\ell-1/2) \ell' + (\ell-1/2) (j'-\kappa) + j(\lambda) \ell'}}}
			{2\cos \brac[\big]{3 \pi (j'-\kappa)} - \sum_{i \in \ZZ_3} 2\cos \brac[\big]{\pi a_i(j',\lambda')}}, \quad
		a_i(j,\lambda) =  (j- \kappa) + 2j^\twist \brac[\big]{\outaut^i(\lambda)}.
	\end{equation}
\end{theorem}
\begin{proof}
	Let $\rrr = [r_0,r_1,r_2]$ and $\sss = [\vv-2,-1,0]$, so that $\lambda = \wtnewparrs$.  As the relaxed modules in \eqref{eq:type3onepoint} have linearly independent one-point functions, the ``highest-weight'' S-matrix element corresponding to $\sfmod{\ell}{\ihw{\lambda}}$ and $\tilderhw{[j'],[\lambda']}^{\ell'}$ is
	\begin{multline}
		\Smatrix{\ell,\lambda}{\ell',[j'],[\lambda']} = \sum_{n=0}^\infty \sum_{m=0}^{\vv-3}(-1)^{m+n\vv} \Bigl(
			\Smatrix{\ell+m+3n\vv+1, [j^{\twist}(\outaut^{-1}(\rrr),\sss_m)+\kappa], [\wtnewpar{\outaut^{-1}(\rrr),\sss_m}]}{\ell',[j'],[\lambda']} \Bigr. \\ \Bigl.
		+ (-1)^\vv \Smatrix{\ell+m+(3n+1)\vv+1, [j^{\twist}(\rrr,\sss_m)+\kappa], [\wtnewpar{\rrr,\sss_m}]}{\ell',[j'],[\lambda']}
			+ \Smatrix{\ell+m+(3n+2)\vv+1, [j^{\twist}(\outaut(\rrr),\sss_m)+\kappa], [\wtnewpar{\outaut(\rrr),\sss_m}]}{\ell',[j'],[\lambda']} \Bigr),
	\end{multline}
	where $\sss_m = [0,m,\vv-3-m]$.  Substituting \eqref{eq:smatrix}, we extract the $n$-dependent terms and perform the sum over $n$:
	\begin{equation} \label{eq:nsum}
		\sum_{n=0}^{\infty} (-1)^{n\vv} \ee^{-2 \pi \ii \brac[\big]{3\vv(j'-\kappa) + 6v \kappa \ell'} n}
		= \frac{1}{1 - (-1)^{\vv} \ee^{-6 \pi \ii \vv (j'-\kappa)}}.
	\end{equation}
	Here, we have noted that $6\vv\kappa = 2\uu-3\vv$.  The $n$-independent remainder may be simplified by noting that
	\begin{equation}
		[j^{\twist}(\outaut(\rrr),\sss)] = [j^{\twist}(\rrr,\sss) + \tfrac{\uu}{3}]
	\end{equation}
	and applying \eqref{eq:Soutaut} and \eqref{eq:Soutaut'} to the $\Wminmoduv$ S-matrix entries.  The result of these simplifications is
	\begin{equation} \label{eq:notnsum}
		\frac{1 - (-1)^{\vv} \ee^{-6\pi\ii\vv (j'-\kappa)}}{1 - \ee^{-2\pi\ii\vv (j'-\kappa)} \ee^{2\pi\ii\vv j^{\twist}(\lambda')}}
			\sum_{m=0}^{\vv-3} (-1)^m \ee^{-2\pi\ii (\ell+m+1) (j'-\kappa)} \ee^{-4\pi\ii \kappa (\ell+m+1) \ell'}
			\ee^{-2\pi\ii j^{\twist}(\outaut^{-1}(\rrr),\sss_m) \ell'} \WSmatrix{[\wtnewpar{\outaut^{-1}(\rrr),\sss_m}]}{[\lambda']},
	\end{equation}
	where we have also noticed that $[6\vv j^{\twist}(\lambda')] = [\vv]$ for all $\lambda' \in \infwtsuv$.

	To evaluate the sum over $m$ in \eqref{eq:notnsum}, we note that
	\begin{equation}
		[j^{\twist}(\rrr,\sss_m)] = [j^{\twist}(\rrr,\sss_0) - 2 \kappa m]
	\end{equation}
	and that writing $\lambda' = \wtnewpar{\rrr',\sss'}$ gives
	\begin{equation}
		\frac{\WSmatrix{[\wtnewpar{\outaut^{-1}(\rrr),\sss_m}]}{[\lambda']}}{\WSmatrix{[\wtnewpar{\rrr,\outaut(\sss_0)}]}{[\lambda']}}
		= \frac{\WSmatrix{[\wtnewpar{\rrr,\outaut(\sss_m)}]}{[\wtnewpar{\rrr',\sss'}]}}{\WSmatrix{[\wtnewpar{\rrr,\mathsf{0}}]}{[\wtnewpar{\rrr',\sss'}]}}
		= \ee^{2\pi\ii \inner{\finite{\rrr'}+\finite{\wvec}}{\finite{\outaut(\sss_m)}}} \slthreecharac{\finite{\outaut(\sss_m)}}{\xi_{\finite{\sss'}}}
		= \ee^{2\pi\ii m \inner{\finite{\rrr'}+\finite{\wvec}}{\fwt{2}}} \slthreecharac{m \fwt{2}}{\xi_{\finite{\sss'}}},
	\end{equation}
	by \cref{prop:ratioWeylcharac}.  Here, we recall that $\outaut(\sss_m) = [\vv-3-m,0,m]$ and denote $\outaut(\sss_0) = [\vv-3,0,0]$ by $\mathsf{0}$ (as in \cref{sec:W3data}).  The factor $\slthreecharac{m \fwt{2}}{\xi_{\finite{\sss'}}}$ denotes the character of the simple highest-weight $\slthree$-module $\fslihw{m\fwt{2}}$ evaluated at the $\slthree$ weight
	\begin{equation}
		\xi_{\finite{\sss'}} = -2\pi\ii \tfracuv (\finite{\sss'}+\finite{\wvec}).
	\end{equation}
	The sum over $m$ in \eqref{eq:notnsum} thus simplifies to
	\begin{equation} \label{eq:msum}
		\ee^{-2\pi\ii (\ell+1) (j'-\kappa)} \ee^{-4\pi\ii \kappa (\ell+1) \ell'} \ee^{-2\pi\ii j^{\twist}(\outaut^{-1}(\rrr),\sss_0) \ell'}
			\WSmatrix{[\underline{\lambda}]}{[\lambda']} \sum_{m=0}^{\vv-3} x^m \slthreecharac{m \fwt{2}}{\xi_{\finite{\sss'}}}, \quad
		x = -\ee^{2\pi\ii \brac*{\inner{\finite{\rrr'} + \finite{\wvec}}{\fwt{2}}-(j'-\kappa)}},
	\end{equation}
	where we have noticed that $[\underline{\lambda}] = [\wtnewpar{\rrr,\outaut(\sss_0)}]$.  The remaining sum is evaluated in \cref{prop:sumoffundmods}, with the result being
	\begin{equation} \label{eq:sinfactors}
		\sum_{m=0}^{\vv-3} x^m \slthreecharac{m\fwt{2}}{\xi_{\finite{\sss'}}}
		= \ee^{3 \pi \ii (j'-\kappa)} \frac{1 - \ee^{-2\pi\ii\vv (j'-\kappa)} \ee^{2\pi\ii\vv j^{\twist}(\lambda')}}{8 \sin(\pi c_1) \sin(\pi c_2) \sin(\pi c_3)},
	\end{equation}
	where $c_i = (j'-\kappa) - j^{\twist}\brac[\big]{\outaut^i(\lambda')}$.

	Putting \eqref{eq:nsum}, \eqref{eq:notnsum}, \eqref{eq:msum} and \eqref{eq:sinfactors} together, we obtain
	\begin{equation}
		\Smatrix{\ell,\lambda}{\ell',[j'],[\lambda']} = \frac{\ee^{-2\pi\ii (\ell-1/2) (j'-\kappa)} \ee^{-4\pi\ii \kappa (\ell+1) \ell'}
			\ee^{-2\pi\ii j^{\twist}(\outaut^{-1}(\rrr),\sss_0) \ell'}}{8 \sin(\pi c_1) \sin(\pi c_2) \sin(\pi c_3)} \WSmatrix{[\underline{\lambda}]}{[\lambda']}.
	\end{equation}
	The proof is now completed by noting that $[j^{\twist}(\outaut^{-1}(\rrr),\sss_0)] = [j(\rrr,\sss)-3\kappa] = [j(\lambda)-3\kappa]$ and
	\begin{align}
		8 \sin &(\pi c_1) \sin(\pi c_2) \sin(\pi c_3) \\
		&= -2 \brac[\Big]{\sin \brac[\big]{\pi(c_1+c_2+c_3)} - \sin \brac[\big]{\pi(c_1-c_2+c_3)}
			- \sin \brac[\big]{\pi(c_2-c_3+c_1)} - \sin \brac[\big]{\pi(c_3-c_1+c_2)}} \notag \\
		&= -2 \brac[\Big]{\sin \brac[\big]{3\pi (j'-\kappa) + \tfrac{3\pi}{2}} - \sin \brac[\big]{\pi a_1(j',\lambda') + \tfrac{3\pi}{2}}
			- \sin \brac[\big]{\pi a_2(j',\lambda') + \tfrac{3\pi}{2}} - \sin \brac[\big]{\pi a_3(j',\lambda') + \tfrac{3\pi}{2}}} \notag \\
		&= 2\cos \brac[\big]{3\pi (j'-\kappa)} - \sum_{i \in \ZZ_3} 2\cos \brac[\big]{\pi a_i(j',\lambda')}, \notag
	\end{align}
	where the $a_i$ were given in \eqref{eq:type3Smatrix}.
\end{proof}

Observe that the denominator of the S-matrix entries \eqref{eq:type3Smatrix} only depends on $j'$ and $\lambda'$: the dependence of $\Smatrix{\ell,\lambda}{\ell',[j'],[\lambda']}$ on the type-$3$ module $\ihw{\lambda}^\ell$ is confined entirely to the exponential term and the $\Wminmoduv$ S-matrix element.  This will prove useful when calculating Grothendieck fusion rules involving type-$3$ modules.  As always, the S-matrix elements involving the vacuum module $\ihw{\kk \fwt{0}} = \ihwpar{\uu-3 & 0 & 0 \\ \vv-2 & -1 & 0}$ are of particular importance in Verlinde computations.  These will again be given the special notation $\vacSmatrix{\ell',[j'],[\lambda']} = \Smatrix{0,\kk \fwt{0}}{\ell',[j'],[\lambda']}$.
\begin{corollary} \label{cor:vacSmatrix}
	Let $\kk$ be nondegenerate-admissible. Then,
	\begin{equation} \label{eq:vacSmatrix}
		\vacSmatrix{\ell',[j'],[\lambda']}
		= \vacWSmatrix{[\lambda']} \frac{\ee^{2\pi\ii \kappa \ell'} \ee^{\pi\ii (j'-\kappa)}}
			{2\cos \brac[\big]{3\pi (j'-\kappa)} - \sum_{i \in \ZZ_3} 2\cos \brac[\big]{\pi a_i(j',\lambda')}}, \quad
		\vacWSmatrix{[\lambda']} = \WSmatrix{[\underline{\kk \fwt{0}}]}{[\lambda']}.
	\end{equation}
\end{corollary}
\noindent As the denominator of \eqref{eq:vacSmatrix} is proportional to $\sin(c_1 \pi)\sin(c_2 \pi)\sin(c_3 \pi)$, see \eqref{eq:sinfactors}, it vanishes if and only if one of the $c_i$ is an integer.  This is equivalent to having $[j'] = [j^{\twist}\brac[\big]{\outaut^i(\lambda')}+\kappa]$ for some $i \in \ZZ_3$.  We conclude that the vacuum S-matrix elements again diverge precisely when $\tilderhw{[j'],[\lambda']}^{\ell'}$ is nonsimple.

\subsection{Grothendieck fusion rules} \label{sec:bpfus}

Taking \cref{conj:SVF} for granted, we now have all the information necessary to compute Grothendieck fusion rules for the standard modules of $\bpminmod{\uu}{\vv}$.  As before, it suffices to compute the Grothendieck fusion rules of representative $\bpminmod{\uu}{\vv}$-modules from each spectral flow orbit and then use ``conservation of spectral flow'' to obtain the rest.

\begin{theorem} \label{thm:standardfusion}
	Let $\kk$ be nondegenerate-admissible.  Then for $\ell, \ell' \in \frac{1}{2}\ZZ$, $[j],[j]' \in \RR / \ZZ$ and $[\lambda], [\lambda'] \in \infwtsuv / \ZZ_3$, the Grothendieck fusion rules of the standard $\bpminmod{\uu}{\vv}$-modules are
	\begin{multline} \label{eq:standardfusion}
    \Gfusion{\tilderhw{[j],[\lambda]}^{\ell}}{\tilderhw{[j'],[\lambda']}^{\ell'}}
    = \sum_{[\lambda''] \in \infwtsuv / \ZZ_3} \Wfuscoeff{[\lambda]}{[\lambda']}{[\lambda'']}
	    \brac*{\Gr{\tilderhw{[j+j'-4\kappa],[\lambda'']}^{\ell+\ell'+2}} +  \Gr{\tilderhw{[j+j'+2\kappa],[\lambda'']}^{\ell+\ell'-1}}} \\
    + \sum_{[\lambda''] \in \infwtsuv / \ZZ_3} \sum_{i \in \ZZ_3} \brac*{
	    \Wfuscoeff{[\lambda]}{[\wtnewpar{\rrr', \sss'-\fwt{i}+\fwt{i+1}}]}{[\lambda'']} \Gr{\tilderhw{[j+j'-2\kappa],[\lambda'']}^{\ell+\ell'+1}}
	    + \Wfuscoeff{[\lambda]}{[\wtnewpar{\rrr', \sss'+\fwt{i}-\fwt{i+1}}]}{[\lambda'']} \Gr{\tilderhw{[j+j'],[\lambda'']}^{\ell+\ell'}}},
	\end{multline}
	where $\lambda' = \wtnewpar{\rrr',\sss'}$.
\end{theorem}
\begin{proof}
	As in the $\vv=3$ case, we apply the standard Verlinde formula \eqref{eq:SVF2} with $\ell = \ell' = 0$ using \eqref{eq:smatrix} and \eqref{eq:vacSmatrix}:
	\begin{multline}
		\fuscoeff{0, [j], [\lambda]}{0, [j'], [\lambda']}{\ell'', [j''], [\lambda'']}
		= \sum_{[\mu] \in \infwtsuv / \ZZ_3} \frac{\WSmatrix{[\lambda]}{[\mu]} \WSmatrix{[\lambda']}{[\mu]} \brac*{\WSmatrix{[\lambda'']}{[\mu]}}^*}{\vacWSmatrix{[\mu]}}
			\sum_{m \in \ZZ} \ee^{-2\pi\ii (j+j'-j''-2\kappa\ell'') m} \\
		\cdot \int_{\RR/\ZZ} \ee^{2\pi\ii (\ell''-1/2) (k-\kappa)}
			\brac[\bigg]{2\cos \brac[\big]{3\pi (k-\kappa)} - \sum_{i \in \ZZ_3} 2\cos \brac[\big]{\pi a_i(k,\mu)}} \, \dd [k].
	\end{multline}
	The Grothendieck fusion coefficient thus naturally splits as a sum of two contributions.  That which involves the $\mu$-independent term $2\cos \brac[\big]{3\pi (k-\kappa)}$ is identical to the $\vv=3$ coefficient computed in \cref{thm:FRstxstv=3}:
	\begin{equation}
		\Wfuscoeff{[\lambda]}{[\lambda']}{[\lambda'']} \brac[\Big]{\delta \brac[\big]{[j''] - [j+j'-4\kappa]} \delta_{\ell'',2}
			+ \delta \brac[\big]{[j''] - [j+j'+2\kappa]} \delta_{\ell'',-1}}.
	\end{equation}

	The contribution that involves the $\mu$-dependent $a_i(k,\mu)$ is more bothersome, simplifying to the form
	\begin{equation} \label{eq:bother}
		-\sum_{[\mu] \in \infwtsuv / \ZZ_3} \sum_{\eps = \pm1} \sum_{i \in \ZZ_3} \ee^{2\pi\ii\eps j^{\twist}\brac*{\outaut^i(\mu)}}
			\frac{\WSmatrix{[\lambda]}{[\mu]} \WSmatrix{[\lambda']}{[\mu]} \brac*{\WSmatrix{[\lambda'']}{[\mu]}}^*}{\vacWSmatrix{[\mu]}}
			\delta \brac[\big]{[j''] - [j+j'-(1-\eps)\kappa]} \delta_{\ell'',\frac{1}{2}(1-\eps)}.
	\end{equation}
	To evaluate this contribution, note that \eqref{eq:identjtw} and \cref{prop:ratioManySs}, with $\finite{\ttt} = \omega_2$, give (for $\eps=+1$)
	\begin{equation}
		\sum_{i \in \ZZ_3} \ee^{2\pi\ii j^{\twist}\brac*{\outaut^i(\mu)}} \WSmatrix{[\lambda']}{[\mu]}
		= -\WSmatrix{[\wtnewpar{\rrr',\sss' \otimes \omega_2}]}{[\mu]}
		= -\sum_{i \in \ZZ_3} \WSmatrix{[\wtnewpar{\rrr',\sss'+\fwt{i}-\fwt{i+1}}]}{[\mu]},
	\end{equation}
	where $\lambda' = \wtnewpar{\rrr',\sss'}$.  Similarly, $\finite{\ttt} = \omega_1$ results (for $\eps=-1$) in
	\begin{equation}
		\sum_{i \in \ZZ_3} \ee^{-2\pi\ii j^{\twist}\brac*{\outaut^i(\mu)}} \WSmatrix{[\lambda']}{[\mu]}
		= -\WSmatrix{[\wtnewpar{\rrr',\sss' \otimes \omega_1}]}{[\mu]}
		= -\sum_{i \in \ZZ_3} \WSmatrix{[\wtnewpar{\rrr',\sss'-\fwt{i}+\fwt{i+1}}]}{[\mu]}.
	\end{equation}
	Note that as $\sss' \in \pwlat{\vv-3}$, the weight $\sss'+\eps(\fwt{i}-\fwt{i+1})$ is either in $\pwlat{\vv-3}$ or it lies on a boundary of a shifted affine alcove, in which case the corresponding S-matrix entry is $0$ by \eqref{eq:WSmatrix=0}.  We may therefore evaluate the $[\mu]$-sum in \eqref{eq:bother} as
	\begin{equation}
		\sum_{\eps = \pm1} \sum_{i \in \ZZ_3} \Wfuscoeff{[\lambda]}{\wtnewpar{\rrr',\sss'+\eps(\fwt{i}-\fwt{i+1})}}{[\lambda'']}
			\delta \brac[\big]{[j''] - [j+j'-(1-\eps)\kappa]} \delta_{\ell'',\frac{1}{2}(1-\eps)},
	\end{equation}
	secure in the knowledge that the $\Wminmoduv$ fusion coefficient is understood to be $0$ whenever $\sss'+\eps(\fwt{i}-\fwt{i+1}) \notin \pwlat{\vv-3}$.
\end{proof}

Reassuringly, all the standard-by-standard Grothendieck fusion coefficients are nonnegative integers, despite the manifest subtractions in the denominator of the vacuum S-matrix entries \eqref{eq:vacSmatrix}.  As in the $\vv=3$ case, the asymmetry in spectral flow indices and $J_0$-eigenvalues can be remedied by recasting the \eqref{eq:standardfusion} in terms of the twisted modules $\twrhw{[j],[\lambda]}$:
\begin{multline} \label{eq:twrelaxedfusion}
	\Gfusion{\sfmod{\ell}{\twrhw{[j],[\lambda]}}}{\sfmod{\ell'}{\twrhw{[j'],[\lambda']}}}
	= \sum_{[\lambda''] \in \infwtsuv / \ZZ_3} \Wfuscoeff{[\lambda]}{[\lambda']}{[\lambda'']}
		\brac*{\Gr{\sfmod{\ell+\ell'+3/2}{\twrhw{[j+j'-3\kappa],[\lambda'']}}} + \Gr{\sfmod{\ell+\ell'-3/2}{\twrhw{[j+j'+3\kappa],[\lambda'']}}}} \\
	+ \sum_{[\lambda''] \in \infwtsuv / \ZZ_3} \sum_{i \in \ZZ_3} \brac*{
		\Wfuscoeff{[\lambda]}{[\wtnewpar{\rrr',\sss'-\fwt{i}+\fwt{i+1}}]}{[\lambda'']} \Gr{\sfmod{\ell+\ell'+1/2}{\twrhw{[j+j'-\kappa],[\lambda'']}}}
		+ \Wfuscoeff{[\lambda]}{[\wtnewpar{\rrr',\sss'+\fwt{i}-\fwt{i+1}}]}{[\lambda'']} \Gr{\sfmod{\ell+\ell'-1/2}{\twrhw{[j+j'+\kappa],[\lambda'']}}}}.
\end{multline}

In principle, all Grothendieck fusion rules involving a \hw\ $\bpminmoduv$-module can now be derived using the resolutions of \cref{sec:bpres}.  As we have derived the type-$3$ S-matrix coefficients in \cref{thm:type3modularity}, Grothendieck fusion coefficients involving type-$3$ \hwms\ and standard modules can be computed directly from the standard Verlinde formula \eqref{eq:SVF2}.  For example, the coefficients for the fusion of $\sfmod{\ell}{\ihw{\lambda}}$ and $\tilderhw{[j'],[\lambda']}^{\ell'}$ are given by
\begin{equation}
	\sum_{m \in \ZZ} \int_{\RR/\ZZ} \sum_{[\mu] \in \infwtsuv / \ZZ_3} \frac{\Smatrix{\ell,\lambda}{m,[k],[\mu]} \ \Smatrix{\ell',[j'],[\lambda']}{m,[k],[\mu]} \
		\brac*{\Smatrix{\ell'',[j''],[\lambda'']}{m,[k],[\mu]}}^*}{\vacSmatrix{m,[k],[\mu]}} \, \dd [k].
\end{equation}
Substituting \eqref{eq:smatrix}, \eqref{eq:type3Smatrix} and \eqref{eq:vacSmatrix}, this evaluates to
\begin{equation}
	\Wfuscoeff{[\underline{\lambda}]}{[\lambda']}{[\lambda'']} \sum_{m \in \ZZ} \int_{\RR/\ZZ}
		\ee^{-2 \pi \ii \brac[\big]{2\kappa (\ell+\ell'-\ell'') m + \left(j(\lambda)+j'-j''\right) m + (\ell+\ell'-\ell'')(k-\kappa)}} \, \dd [k]
	= \Wfuscoeff{[\underline{\lambda}]}{[\lambda']}{[\lambda'']}\delta \brac[\big]{[j''] - [j(\lambda)+j']} \delta_{\ell'',\ell+\ell'}.
\end{equation}
Of course, this can also be checked using resolutions, as we did for $\vv=3$ in \cref{cor:FR3xstv=3}.
\begin{corollary} \label{prop:FR3xst}
	Let $\kk$ be nondegenerate-admissible.  Then for $\ell,\ell' \in \frac{1}{2}\ZZ$, $[j'] \in \RR / \ZZ$, $\lambda = \wtpar{r_0 & r_1 & r_2 \\ \vv-2 & -1 & 0} \in \surv{\uu}{\vv}$ and $[\lambda'] \in \infwtsuv / \ZZ_3$, the type-$3$-by-standard Grothendieck fusion rules are
	\begin{equation} \label{eq:FR3xst}
    \Gfusion{\sfmod{\ell}{\ihw{\lambda}}}{\tilderhw{[j'],[\lambda']}^{\ell'}}
    = \sum_{[\lambda''] \in \infwts{\uu}{\vv} / \ZZ_3} \Wfuscoeff{[\underline{\lambda}]}{[\lambda']}{[\lambda'']}
	    \Gr{\tilderhw{[j(\lambda)+j'],[\lambda'']}^{\ell+\ell'}}.
	\end{equation}
\end{corollary}
\noindent The corresponding ``symmetrised'' Grothendieck fusion rules are obtained by simply replacing $\tilderhw{[j'],[\lambda']}^{\ell'}$ by $\sfmod{\ell'}{\twrhw{[j'],[\lambda']}}$ in \eqref{eq:FR3xst}.

Our next stop is the type-$3$-by-type-$3$ Grothendieck fusion rules.  In preparation for this, it will be useful to restrict \cref{prop:FR3xst} to the nonsimple $\tilderhw{[j'],[\lambda']}^{\ell'}$, that is to the $\tilderhw{\lambda'}^{\ell'} = \tilderhw{[j^{\twist}(\lambda')+\kappa],[\lambda']}^{\ell'}$:
\begin{equation} \label{eq:FR3xstred}
	\Gfusion{\sfmod{\ell}{\ihw{\lambda}}}{\tilderhw{\lambda'}^{\ell'}}
	= \sum_{[\lambda''] \in \infwts{\uu}{\vv} / \ZZ_3} \Wfuscoeff{[\underline{\lambda}]}{[\lambda']}{[\lambda'']}
		\Gr{\tilderhw{[j(\lambda)+j^{\twist}(\lambda')+\kappa],[\lambda'']}^{\ell+\ell'}}.
\end{equation}
Our first task is to show that the standard modules appearing on the \rhs\ are also nonsimple.

Without loss of generality, let $\lambda'' = \wtnewpar{\rrr'',\sss''} \in [\lambda'']$ be a representative satisfying the conditions required by \cref{thm:W3fusion}.  Since $\underline{\lambda}, \lambda' \in \infwtsuv$ are fixed in the (Grothendieck) fusion of $\ihw{\lambda}$ and $\tilderhw{\lambda'}$, the corresponding representatives of $[\underline{\lambda}]$ and $[\lambda']$ will have the form $\outaut^m(\underline{\lambda})$ and $\outaut^n(\lambda')$, respectively, for some $m,n\in\ZZ_3$.  Then, with $\underline{\lambda} = \wtnewpar{\rrr',\mathsf{0}}$ (where $\mathsf{0} \equiv [\vv-3,0,0]$) and $\lambda' = \wtnewpar{\rrr',\sss'}$, \cref{eq:sl3fuscoeffoutaut,thm:W3fusion} give
\begin{equation} \label{eq:splittingW3intosl3s}
	\Wfuscoeff{[\underline{\lambda}]}{[\lambda']}{[\lambda'']}
	= \slfuscoeff{\uu-3}{\outaut^m(\rrr)}{\outaut^n(\rrr')}{\rrr''} \slfuscoeff{\vv-3}{\outaut^m(\mathsf{0})}{\outaut^n(\sss')}{\sss''}
	= \slfuscoeff{\uu-3}{\rrr}{\rrr'}{\outaut^{-m-n}(\rrr'')} \slfuscoeff{\vv-3}{0}{\sss'}{\outaut^{-m-n}(\sss'')}
	= \slfuscoeff{\uu-3}{\rrr}{\rrr'}{\outaut^{-m-n}(\rrr'')} \delta_{\sss',\outaut^{-m-n}(\sss'')}.
\end{equation}
For the $\Wthree$ fusion coefficient to be nonzero, it therefore must be the case that $\lambda'' = \outaut^{m+n}\brac[\big]{\wtnewpar{\ttt'',\sss'}}$, for some $\ttt'' \in \pwlat{\uu-3}$.  But, if $\slfuscoeff{\uu-3}{\rrr}{\rrr'}{\outaut^{-m-n}(\rrr'')} = \slfuscoeff{\uu-3}{\rrr}{\rrr'}{\ttt''}$ is nonzero, then the Kac--Walton formula \eqref{eq:KacWalton} shows that $\finite{t''} = \finite{r} + \finite{r'} \bmod{\rlat}$, hence
\begin{equation}
	[j(\lambda)+j^\twist(\lambda')+\kappa] = [j^{\twist} \brac[\big]{\outaut^{-m-n}(\lambda'')} + \kappa] = [j^{\twist}(\ttt'',\sss') + \kappa].
\end{equation}
The standard modules on the \rhs\ of \eqref{eq:FR3xstred} are thus the nonsimple modules $\tilderhw{\wtnewpar{\ttt'',\sss'}}^{\ell+\ell'}$, where $t''$ satisfies the equation above:
\begin{equation} \label{eq:type3byred}
	\Gfusion{\sfmod{\ell}{\ihw{\lambda}}}{\tilderhw{\lambda'}^{\ell'}}
	= \sum_{\substack{\ttt'' \in \pwlat{\uu-3} \\ [j(\ttt'')] = [j(\rrr)+j(\rrr')]}} \Wfuscoeff{[\underline{\lambda}]}{[\lambda']}{[\wtnewpar{\ttt'',\sss'}]}
		\Gr{\tilderhw{\wtnewpar{\ttt'',\sss'}}^{\ell+\ell'}}.
\end{equation}
As in the proof of \cref{cor:FR3x3v=3}, the additional constraint on $\ttt''$ may be removed by converting the $\Wminmoduv$ fusion coefficient to a $\slminmod{\uu}{1}$ one.  Replacing $\ttt''$ with $\rrr''$, the final type-$3$-by-nonsimple Grothendieck fusion rule is thus
\begin{equation} \label{eq:type3byred'}
	\Gfusion{\sfmod{\ell}{\ihw{\lambda}}}{\tilderhw{\lambda'}^{\ell'}}
	= \sum_{\rrr'' \in \pwlat{\uu-3}} \slfuscoeff{\uu-3}{\rrr}{\rrr'}{\rrr''} \Gr{\tilderhw{\wtnewpar{\rrr'',\sss'}}^{\ell+\ell'}}.
\end{equation}

This leads to a straightforward computation for the type-$3$-by-type-$3$ Grothendieck fusion rules.
\begin{corollary} \label{cor:FR3x3}
	Let $\kk$ be nondegenerate-admissible.  Then for $\ell,\ell' \in \frac{1}{2}\ZZ$, $\lambda = \wtpar{r_0 & r_1 & r_2 \\ \vv-2 & -1 & 0} \in \survuv$ and $\lambda' = \wtpar{r'_0 & r'_1 & r'_2 \\ \vv-2 & -1 & 0} \in \survuv$, the Grothendieck fusion rules between type-$3$ highest-weight $\bpminmod{\uu}{\vv}$-modules are
	\begin{equation} \label{eq:FR3x3}
		\Gfusion{\sfmod{\ell}{\ihw{\lambda}}}{\sfmod{\ell}{\ihw{\lambda'}}}
    = \sum_{\rrr'' \in \pwlat{\uu-3}} \slfuscoeff{\uu-3}{\rrr}{\rrr'}{\rrr''} \Gr{\sfmod{\ell+\ell'}{\ihwpar{r''_0 & r''_1 & r''_2 \\ \vv-2 & -1 & 0}}}.
	\end{equation}
\end{corollary}
\begin{proof}
	By \eqref{eq:shortcut}, it is enough to prove \eqref{eq:FR3x3} when $\ell = \ell' = 0$.  Let $\sss'_m = [0,m,\vv-3-m]$, as in the proof of \cref{thm:type3modularity}.  Substituting \eqref{eq:type3onepoint} and then \eqref{eq:type3byred} into the \lhs\ of \eqref{eq:FR3x3}, we get
	\begin{align} \label{eq:type3ihwSTEP}
		\Gfusion{\ihw{\lambda}}{\ihw{\lambda'}}
		&= \sum_{n=0}^\infty \sum_{m=0}^{\vv-3} \sum_{i=0}^2 (-1)^{m+(n+i)\vv}
			\Gfusion{\ihw{\lambda}}{\tilderhw{\wtnewpar{\outaut^{i-1}(\rrr'),\sss'_m}}^{m+(3n+i)\vv+1}} \\
		&= \sum_{n=0}^\infty \sum_{m=0}^{\vv-3} \sum_{i=0}^2 (-1)^{m+(n+i)\vv}
			\sum_{r'' \in \pwlat{\uu-3}} \slfuscoeff{\uu-3}{\rrr}{\outaut^{i-1}(\rrr')}{\rrr''} \Gr{\tilderhw{\wtnewpar{\rrr'',\sss'_m}}^{m+(3n+i)\vv+1}} \notag \\
		&= \sum_{r'' \in \pwlat{\uu-3}} \slfuscoeff{\uu-3}{\rrr}{\rrr'}{\rrr''}
			\sum_{n=0}^\infty \sum_{m=0}^{\vv-3} \sum_{i=0}^2 (-1)^{m+(n+i)\vv} \Gr{\tilderhw{\wtnewpar{\outaut^{i-1}(\rrr''),\sss'_m}}^{m+(3n+i)\vv+1}} \notag \\
		&= \sum_{r'' \in \pwlat{\uu-3}} \slfuscoeff{\uu-3}{\rrr}{\rrr'}{\rrr''} \Gr{\ihwpar{r''_0 & r''_1 & r''_2 \\ \vv-2 & -1 & 0}}, \notag
	\end{align}
	using \eqref{eq:type3onepoint} again.
\end{proof}
\noindent As the simple \hw\ $\slminmod{\uu}{1}$-modules of highest weights $[0,\uu-3,0]$ and $[0,0,\uu-3]$ are simple currents, \cref{cor:FR3x3} implies the existence of simple currents for $\bpminmoduv$.  (This generalises \cref{prop:simpcurr} to $\vv > 3$.)
\begin{proposition} \label{prop:simpcurrALLV}
	Let $\kk$ be nondegenerate-admissible with $\uu>3$.  Then, $\ihwpar{0 & \uu-3 & 0 \\ \vv-2 & -1 & 0}$ and $\ihwpar{0 & 0 & \uu-3 \\ \vv-2 & -1 & 0}$ are simple currents of order $3$, inverse to one another.  Their highest weights (with respect to $J_0$ and $L_0$) are
	\begin{equation} \label{eq:simplecurrentdataALLV}
		(j,\Delta) = (+\tfrac{\uu-3}{3}, \tfrac{(\uu-3)(2\vv-3)}{6}) \qquad \text{and} \qquad (j,\Delta) = (-\tfrac{\uu-3}{3},\tfrac{(\uu-3)(2\vv-3)}{6}),
	\end{equation}
	respectively.
\end{proposition}
\noindent Finally, \cref{prop:fusringisov=3} also generalises to $\vv>3$.  The proof is identical, hence omitted.
\begin{proposition} \label{prop:fusringiso}
	Let $\kk$ be nondegenerate-admissible.  Then, the fusion subring of generated by the type-$3$ simple \hw\ $\bpminmoduv$-modules $\ihw{\lambda}$, $\lambda = \wtpar{r_0 & r_1 & r_2 \\ \vv-2 & -1 & 0} \in \surv{\uu}{\vv}$, is isomorphic to the fusion ring of the affine \voa\ $\sslvoa{\uu-3}$.
\end{proposition}

The computation of Grothendieck fusion rules involving type-$1$ or type-$2$ modules becomes complicated very quickly.  Those with the standards are manageable, but general highest-weight-by-highest-weight rules involve resolutions with many terms and the appropriate cancellations become hard to identify.  Our philosophy here is that one should not really expect to determine all (Grothendieck) fusion rules explicitly.  Instead, it is better to provide an algorithmic means to construct the desired rules in individual cases (the Kac--Walton formula \eqref{eq:KacWalton} is an exemplar of this philosophy).  This is what the resolutions and character formulae in \cref{prop:reso,cor:onepoint} are for.  We shall illustrate their application by computing the type-$1$ and type-$2$ Grothendieck fusion rules for $\bpminmod{3}{4}$ below.

\subsection{Example: $\bpminmod{3}{4}$} \label{sec:bpex}

Consider the \bp\ minimal model $\bpminmod{3}{4}$ with $\kk = -\frac{9}{4}$ and $\cc = -\frac{23}{2}$.  This model is denoted by $\mathcal{B}_4$ in \cite{CreCos13}.  In \cite{AdaCla19,FehCla20}, it was shown that there are $6$ untwisted (with respect to $L(z)$) simple \hwms.  We arrange them as in \cref{fig:sforbits}, adding the action of $\outaut$ to the spectral flow orbits:
\begin{equation}
	\begin{tikzpicture}[->,>=to,xscale=1.2,yscale=0.6,baseline=(c.base)]
		\node (31) at (-2,0) {$\ihwpar{0&0&0 \\ 0&-1&2}$};
		\node (32) at (0,0) {$\ihwpar{0&0&0 \\ 2&-1&0}$};
		\node (33) at (2,0) {$\ihwpar{0&0&0 \\ 0&1&0}$};
		\node (21) at (-1,2) {$\ihwpar{0&0&0 \\ 1&-1&1}$};
		\node (22) at (1,2) {$\ihwpar{0&0&0 \\ 1&0&0}$};
		\node (11) at (0,4) {$\ihwpar{0&0&0 \\ 0&0&1}$};
		\draw (31) -- node[above] {$\scriptstyle \sfsymb{}$} (32);
		\draw (32) -- node[above] {$\scriptstyle \sfsymb{}$} (33);
		\draw (21) -- node[above] {$\scriptstyle \sfsymb{}$} (22);
		\draw (11) -- node[above right] {$\scriptstyle \outaut$} (22);
		\draw (22) -- node[above right] {$\scriptstyle \outaut$} (33);
		\node at (-4,4) {type-$1$:};
		\node (c) at (-4,2) {type-$2$:};
		\node at (-4,0) {type-$3$:};
	\end{tikzpicture}
\end{equation}
In addition to these, there is a one-parameter family of untwisted \rhwms
\begin{equation}
	\tilderhw{[j]} = \tilderhw{[j],[\wtnewpar{[0,0,0],[1,0,0]}]}.
\end{equation}
These modules are simple for all $[j] \in \RR/\ZZ$ except $[j] = [0]$, $[\frac{1}{2}]$ or $[\frac{3}{4}]$.  More precisely, the nonsimple cases are
\begin{equation}
	\trhwpar{0&0&0 \\ 0&0&1}  = \tilderhw{[1/2]}, \quad
	\trhwpar{0&0&0 \\ 1&0&0}  = \tilderhw{[1/4]} \quad \text{and} \quad
	\trhwpar{0&0&0 \\ 0&1&0}  = \tilderhw{[0]}.
\end{equation}
Together with \eqref{eq:newses}, this identification leads to the following useful identities in the Grothendieck ring of $\bpminmod{3}{4}$:
\begin{equation} \label{eq:useful34}
	\begin{aligned}
		\Gr{\tilderhw{[1/2]}} &= \Gr{\ihwpar{0&0&0 \\ 0&0&1}^1} + \Gr{\ihwpar{0&0&0 \\ 2&-1&0}^{-1}}, \\
		\Gr{\tilderhw{[1/4]}} &= \Gr{\ihwpar{0&0&0 \\ 1&0&0}^1} + \Gr{\ihwpar{0&0&0 \\ 1&0&0}^{-1}}, \\
		\Gr{\tilderhw{[0]}} &= \Gr{\ihwpar{0&0&0 \\ 2&-1&0}^2} + \Gr{\ihwpar{0&0&0 \\ 0&0&1}}.
	\end{aligned}
\end{equation}

Since $\uu=3$, the type-$3$ modules are all spectral flows of the vacuum module.  Their fusion rules are thus trivial to determine.  Of the general results reported in \cref{sec:bpfus}, the interesting Grothendieck fusion rules are then the standard-by-standard ones of \cref{thm:standardfusion}.  For this, note that $\infwts{3}{4}$ has only one $\ZZ_3$-orbit and that the representatives used in $\Wminmod{3}{4}$ fusion coefficients (\cref{thm:W3fusion}) all have the form $\lambda = \wtnewpar{\rrr,\sss}$, with $\rrr=[0,0,0]$ and $\sss=[1,0,0]$, because $\uu=3$.  The fusion coefficients in \eqref{eq:standardfusion} are then
\begin{equation}
	\Wfuscoeff{[\lambda]}{[\lambda]}{[\lambda]} = \slfuscoeff{1}{[0,0]}{[0,0]}{[0,0]} = 1, \qquad
	\begin{aligned}
		\Wfuscoeff{[\lambda]}{[\wtnewpar{\rrr,\sss-\omega_i+\omega_{i+1}}]}{[\lambda]} &= \delta_{i,0} \slfuscoeff{1}{[0,0]}{[0,0]}{[0,0]} = \delta_{i,0}, \\
		\Wfuscoeff{[\lambda]}{[\wtnewpar{\rrr,\sss+\omega_i-\omega_{i+1}}]}{[\lambda]} &= \delta_{i,2} \slfuscoeff{1}{[0,0]}{[0,0]}{[0,0]} = \delta_{i,2},
	\end{aligned}
\end{equation}
since (for example) $[\wtnewpar{\rrr,\sss-\omega_i+\omega_{i+1}}] = [\wtnewpar{\rrr,[0,1,0]}] = [\lambda]$ when $i=0$ and $\sss-\omega_i+\omega_{i+1} \notin \pwlat{1}$ otherwise.  The standard-by-standard Grothendieck fusion rules are thus
\begin{subequations}
	\begin{equation} \label{eq:FRstxst34}
		\Gfusion{\tilderhw{[j]}^{\ell}}{\tilderhw{[j']}^{\ell'}} = \Gr{\tilderhw{[j+j'+1/2]}^{\ell+\ell'-1}} + \Gr{\tilderhw{[j+j']}^{\ell+\ell'}}
			+ \Gr{\tilderhw{[j+j'+1/2]}^{\ell+\ell'+1}} + \Gr{\tilderhw{[j+j']}^{\ell+\ell'+2}}.
	\end{equation}
	Consulting \eqref{eq:twrelaxedfusion}, the ``symmetrised'' versions are
	\begin{align}
		\Gfusion{\sfmod{\ell}{\twrhw{[j]}}}{\sfmod{\ell'}{\twrhw{[j']}}}
		&= \Gr{\sfmod{\ell+\ell-3/2}{\twrhw{[j+j'+1/4]}}} + \Gr{\sfmod{\ell+\ell-1/2}{\twrhw{[j+j'-1/4]}}} \\
			&\quad + \Gr{\sfmod{\ell+\ell+1/2}{\twrhw{[j+j'+1/4]}}} + \Gr{\sfmod{\ell+\ell+3/2}{\twrhw{[j+j'-1/4]}}}. \notag
	\end{align}
\end{subequations}

The remaining Grothendieck fusion rules involve type-$1$ and type-$2$ modules.  For the former, we note that \eqref{eq:newses} implies the following equalities:
Using \eqref{eq:FRstxst34}, the Grothendieck fusion of type-$1$ modules with standard modules is easily found to be
\begin{equation}
	\Gfusion{\ihwpar{0&0&0\\0&0&1}^{\ell}}{\tilderhw{[j']}^{\ell'}}
	= \Gr{\tilderhw{[j'+1/2]}^{\ell+\ell'-1}} + \Gr{\tilderhw{[j]}^{\ell+\ell'}} + \Gr{\tilderhw{[j'+1/2]}^{\ell+\ell'+1}}.
\end{equation}
With this result, it is straightforward to compute the type-$1$-by-type-$1$ Grothendieck fusion rules:
\begin{align}
	\Gfusion{\ihwpar{0&0&0\\0&0&1}^{\ell}}{\ihwpar{0&0&0\\0&0&1}^{\ell'}}
	&= \Gr{\tilderhw{[1/2]}^{\ell+\ell'-1}} + \Gr{\tilderhw{[0]}^{\ell+\ell'}} + \Gr{\tilderhw{[1/2]}^{\ell+\ell'+1}} - \Gr{\ihwpar{0&0&0\\0&0&1}^{\ell+\ell'+2}} \\
	&= 2 \Gr{\ihwpar{0&0&0\\0&0&1}^{\ell+\ell'}} + \Gr{\ihwpar{0&0&0 \\ 2&-1&0}^{\ell+\ell'-2}} + \Gr{\ihwpar{0&0&0 \\ 2&-1&0}^{\ell+\ell'}} + \Gr{\ihwpar{0&0&0 \\ 2&-1&0}^{\ell+\ell'+2}}. \notag
\end{align}
Here, we have used the first and last identity in \eqref{eq:useful34}.

The type-$2$ case requires slightly more work as \eqref{eq:newses} now relates a type-$2$ simple to another type-$2$ simple.  As a result, this requires consideration of the full resolution \eqref{eq:reso}.  However, the full resolution is in fact quite simple in this case.  Taking $\ihwpar{0&0&0\\1&-1&1}$ as the representative of the type-$2$ spectral flow orbit, we find that
\begin{equation}
	\Gr{\ihwpar{0&0&0\\1&-1&1}} = \sum_{n=0}^\infty (-1)^n \Gr{\trhwpar{0&0&0 \\ 1&0&0}^{2n}}.
\end{equation}
The Grothendieck fusion rule of this simple with a standard module thus results in an infinite alternating sum of Grothendieck images of standard modules, all but two of which cancel:
\begin{equation}
	\Gfusion{\ihwpar{0&0&0\\1&-1&1}^{\ell}}{\tilderhw{[j']}^{\ell'}}
	= \Gr{\tilderhw{[j'-1/4]}^{\ell+\ell'-1}} + \Gr{\tilderhw{[j'+1/4]}^{\ell+\ell'}}.
\end{equation}
Combining these last two results with \eqref{eq:useful34} gives the type-$2$-by-type-$2$ Grothendieck fusion rules:
\begin{equation}
	\Gfusion{\ihwpar{0&0&0\\1&0&0}^{\ell}}{\ihwpar{0&0&0\\1&-1&1}^{\ell'}}
	= \Gr{\ihwpar{0&0&0\\0&0&1}^{\ell+\ell'}} + \Gr{\ihwpar{0&0&0\\2&-1&0}^{\ell+\ell'}}.
\end{equation}
It only remains to compute the type-$1$-by-type-$2$ Grothendieck fusion rules.  The same methods result in
\begin{equation}
	\Gfusion{\ihwpar{0&0&0\\0&0&1}^{\ell}}{\ihwpar{0&0&0\\1&-1&1}^{\ell'}} = \Gr{\ihwpar{0&0&0\\1&-1&1}^{\ell+\ell'}} + \Gr{\tilderhw{[3/4]}^{\ell+\ell'-1}}.
\end{equation}
Here, we note that $\tilderhw{[3/4]}$ is simple, unlike the ``gap modules'' of \eqref{eq:useful34}.

We remark that these results can be checked by exploiting three facts.  First, the coset of $\bpminmod{3}{4}$ by the Heisenberg subalgebra generated by $J$ is the singlet algebra $W^0(1,4)$ \cite{CreCos13}.  The representation theory of the latter may then be constructed from that of the former, using the results of \cite{CreSch16}.  Second, the triplet algebra $W(1,4)$ of central charge $-\frac{25}{2}$ is an infinite-order simple current extension of $W^0(1,4)$ \cite{RidMod13}.  Again, the representation theory of the latter may be constructed from that of the former.  Finally, the fusion rules of $W(1,4)$ are well known, see \cite{FucNon03,GabFro07,TsuTen12}, and the structure of the indecomposable projectives has been established.  We shall leave this consistency check to the interested reader, noting only that this procedure may be reversed to glean information about the (genuine) fusion rules and projective indecomposables of $\bpminmod{3}{4}$.  Again, this application is left for the future.

\appendix

\section{Modularity and fusion for the $\Wthree$ minimal models} \label{sec:W3data}

The results reported here for the Grothendieck fusion rules of the \bp\ minimal models $\bpminmoduv$, with $\kk$ nondegenerate-admissible (so $\uu,\vv\ge3$), rely on the modularity and fusion rules of the rational $\Wthree$ minimal models $\Wminmoduv$.  We review these here, specialising the results obtained for general regular W-algebras in \cite{KacBra90,FreCha92,AraMod19}.  In addition, we deduce several identities satisfied by the $\Wminmoduv$ S-matrix elements that will be crucial in our $\bpminmoduv$ investigations.

\subsection{Modularity of $\Wminmoduv$ one-point functions} \label{sec:W3Smatrixstuff}

Recall that the simple $\Wminmoduv$ modules $\Wihw{[\lambda]}$ are parametrised by elements $[\lambda] \in \infwtsuv / \ZZ_3$, where the $\ZZ_3$ action was given in \eqref{eq:Z3action}.  Recall also the parametrisation $\lambda = \wtnewparrs$ given in \eqref{eq:rspar}, where $\lambda \in \infwtsuv$ means that the $\aslthree$-weights $\rrr = [r_0,r_1,r_2]$ and $\sss = [s_0,s_1,s_2]$ belong to $\pwlat{\uu-3}$ and $\pwlat{\vv-3}$, respectively.  Their projections onto the weight space of $\slthree$ will be denoted by $\finite{\rrr} = [r_1,r_2]$ and $\finite{\sss} = [s_1,s_2]$.  Let $\finite{\wvec} = [1,1]$ denote the Weyl vector of $\slthree$ and $\wgrp$ its Weyl group.

\begin{theorem}[{\cite[Thm.~4.4]{KacBra90}, \cite[Cor.~8.4]{AraMod19}}] \label{thm:W3Smatrix}
	For $\kk$ nondegenerate-admissible, the S-transform of the $\Wminmoduv$ one-point function \eqref{eq:W3opfs} is given by \eqref{eq:W3Stransform} and the S-matrix entries are given, for $[\wtnewparrs], [\wtnewpar{\rrr',\sss'}] \in \infwtsuv / \ZZ_3$, by
	\begin{equation} \label{eq:W3Smatrix}
		\WSmatrix{[\wtnewparrs]}{[\wtnewpar{\rrr',\sss'}]}
		= \frac{1}{\sqrt{3} \uu \vv}
			\ee^{2\pi\ii \brac*{\inner{\finite{\rrr}+\finite{\wvec}}{\finite{\sss'}+\finite{\wvec}} + \inner{\finite{\sss}+\finite{\wvec}}{\finite{\rrr'}+\finite{\wvec}}}}
			\sum_{w \in \wgrp} \det w \, \ee^{-2\pi\ii \frac{\vv}{\uu} \inner{w(\finite{\rrr}+\finite{\wvec})}{\finite{\rrr'}+\finite{\wvec}}}
			\sum_{w \in \wgrp} \det w \, \ee^{-2\pi\ii \frac{\uu}{\vv} \inner{w(\finite{\sss}+\finite{\wvec})}{\finite{\sss'}+\finite{\wvec}}}.
	\end{equation}
\end{theorem}

The fact that this S-matrix formula is independent of the choice of representatives of the $\ZZ_3$-orbits deserves comment.  Acting on $\rrr$ or $\sss$ by the $\ZZ_3$-generator $\outaut$ amounts to acting with an outer automorphism of $\aslthree$.  It is easy to check that on the projection onto the weight space of $\slthree$, $\outaut$ acts as follows:
\begin{equation}
	\outaut(\finite{\ttt}) = \wref{1} \wref{2}(\finite{\ttt}) + \kk(\ttt) \fwt{1}.
\end{equation}
Here, $\fwt{1} = [1,0]$ is the first fundamental weight of $\slthree$ and $\kk(\ttt)$ is the level of $\ttt$.  With this, it is now easy to verify that acting with $\outaut$ on \eqref{eq:W3Smatrix} gives
\begin{equation} \label{eq:Soutaut}
	\begin{aligned}
		\WSmatrix{[\wtnewpar{\outaut(\rrr),\sss}]}{[\wtnewpar{\rrr',\sss'}]}
		&= \ee^{-2\pi\ii\vv \inner{\fwt{1}}{\finite{\rrr'}+\finite{\wvec}}} \ee^{-\vv \inner{\fwt{1}}{\xi_{\finite{\sss'}}}}
			\WSmatrix{[\wtnewparrs]}{[\wtnewpar{\rrr',\sss'}]} \\ \text{and} \quad
		\WSmatrix{[\wtnewpar{\rrr,\outaut(\sss)}]}{[\wtnewpar{\rrr',\sss'}]}
		&= \ee^{+2\pi\ii\vv \inner{\fwt{1}}{\finite{\rrr'}+\finite{\wvec}}} \ee^{+\vv \inner{\fwt{1}}{\xi_{\finite{\sss'}}}}
			\WSmatrix{[\wtnewparrs]}{[\wtnewpar{\rrr',\sss'}]},
	\end{aligned}
	\qquad \xi_{\finite{\sss'}} = -2\pi\ii \tfracuv (\finite{\sss'}+\finite{\wvec}).
\end{equation}
Obviously, applying $\outaut$ to both $\rrr$ and $\sss$ leaves the S-matrix invariant.  We mention that the phase appearing in \eqref{eq:Soutaut} may be rewritten in the following $\bpminmoduv$-friendly form by noting that $\vv \xi_{\finite{\rrr'},\finite{\sss'}}$ belongs to the weight lattice of $\slthree$:
\begin{align} \label{eq:Soutaut'}
	\ee^{-2\pi\ii\vv \inner{\fwt{1}}{\finite{\rrr'}+\finite{\wvec}}} \ee^{-\vv \inner{\fwt{1}}{\xi_{\finite{\sss'}}}}
	&= \ee^{2\pi\ii\vv \inner{\fwt{1}-\fwt{2}}{\finite{\rrr'} + \finite{\wvec} - \fracuv (\finite{\sss'} + \finite{\wvec})}}
	= \ee^{2\pi\ii\vv \inner{\fwt{1}-\fwt{2}}{\finite{\rrr'} - \fracuv (\finite{\sss'} + \fwt{1}) + \finite{\wvec} - \fracuv \fwt{2}}} \\
	&= \ee^{2\pi\ii \brac[\big]{\vv j(\lambda') + \uu/3}}
	= (-1)^{\vv} \ee^{2\pi\ii\vv \brac[\big]{j(\lambda')+\kappa}}
	= (-1)^{\vv} \ee^{2\pi\ii\vv j^{\twist}(\lambda')}. \notag
\end{align}

Applying $\outaut$ to both $\rrr'$ and $\sss'$ also leaves the S-matrix invariant because \eqref{eq:W3Smatrix} is manifestly symmetric.  The S-matrix may also be verified to be unitary, see for example \cite[Prop.~4.4]{KacBra90}.  A similar calculation demonstrates that its square is the matrix whose $[\lambda], [\lambda']$-entry is $0$ unless $\finite{\rrr'}=[r_2,r_1]$ and $\finite{\sss'}=[s_2,s_1]$, in which case it is $1$.  Referring back to \eqref{eq:W3conj}, we see that this matrix represents conjugation, as expected.

We remark that it is sometimes useful to extend \eqref{eq:W3Smatrix} to allow arbitrary integral $\aslthree$-weights $\rrr$, $\rrr'$, $\sss$ and $\sss'$.  The \rhs\ of \eqref{eq:W3Smatrix} then exhibits various symmetries.  For example, it is straightforward to show that
\begin{equation} \label{eq:WSWeylSymmetry}
	\WSmatrix{[\wtnewpar{\rrr,w \cdot \sss}]}{[\wtnewpar{\rrr',\sss'}]} = \det w \, \WSmatrix{[\wtnewparrs]}{[\wtnewpar{\rrr',\sss'}]}, \quad w \in \wgrp,
\end{equation}
where $w \cdot \sss = w(\sss+\wvec)-\wvec$ is the usual shifted action.  It follows that if $s_i=-1$, for $i=1$ or $2$, then it is fixed by the shifted action of the $i$-th simple Weyl reflection $w_i$ and so
\begin{equation} \label{eq:WSmatrix=0}
	\WSmatrix{[\wtnewparrs]}{[\wtnewpar{\rrr',\sss'}]} = 0.
\end{equation}
Similarly, the well known decomposition of $\finite{w_0(\sss+\wvec)}$ as the Weyl reflection for the highest root $\theta$ followed by translation by $\vv \theta$ leads to \eqref{eq:WSWeylSymmetry} also holding for $w=w_0$ (and therefore for any $w$ in the affine Weyl group $\awgrp$ of $\aslthree$).  Consequently, \eqref{eq:WSmatrix=0} continues to hold if $s_0=-1$, hence $w_0 \cdot \sss = \sss$.

Summarising, it follows that the $\Wminmoduv$ S-matrix entry \eqref{eq:W3Smatrix} vanishes when $\sss$ lies on a shifted affine alcove boundary.  This obviously remains true if we swap $\sss$ and $\vv$ with $\rrr$ and $\uu$.

\subsection{Identities for $\Wminmoduv$ S-matrix elements} \label{sec:W3Smatrixidentities}

In many of the computations performed in this paper, for example those in \cref{sec:bpu3}, the explicit formula for the $\Wminmoduv$ S-matrix can largely be ignored.  However, there are instances in which we encounter sums and ratios of $\Wminmoduv$ S-matrix elements.  Here, we address some means for dealing with these instances.

We begin with a simple ratio calculation.  For any $\slthree$-weight $\finite{\ttt} = [t_1,t_2]$, denote the character of the simple highest-weight $\slthree$-module $\fslihw{\finite{\ttt}}$ by $\slthreecharacnoarg{\finite{\ttt}}$.  Let $\mathsf{0} = [\vv-3,0,0]$.
\begin{proposition} \label{prop:ratioWeylcharac}
	Let $\kk$ be nondegenerate-admissible and $[\wtnewparrs], [\wtnewpar{\rrr',\sss'}] \in \infwtsuv / \ZZ_3$.  Then,
	\begin{equation} \label{eq:ratioWeylcharac}
		\frac{\WSmatrix{[\wtnewpar{\rrr,\sss}]}{[\wtnewpar{\rrr',\sss'}]}}{\WSmatrix{[\wtnewpar{\rrr,\mathsf{0}}]}{[\wtnewpar{\rrr',\sss'}]}}
		= \ee^{2\pi\ii \inner{\finite{\sss}}{\finite{\rrr'}+\finite{\wvec}}} \slthreecharac{\finite{\sss}}{\xi_{\finite{\sss'}}}.
	\end{equation}
\end{proposition}
\begin{proof}
Substituting \eqref{eq:W3Smatrix} into the left-hand-side of \eqref{eq:ratioWeylcharac} and simplifying gives
\begin{equation}
	\frac{\WSmatrix{[\wtnewpar{\rrr,\sss}]}{[\wtnewpar{\rrr',\sss'}]}}{\WSmatrix{[\wtnewpar{\rrr,\mathsf{0}}]}{[\wtnewpar{\rrr',\sss'}]}}
	= \ee^{2\pi\ii \inner{\finite{\sss}}{\finite{\rrr'}+\finite{\rho}}}
		\frac{\sum_{w \in \wgrp} \det w \, \ee^{\inner{w(\finite{\sss}+\finite{\wvec})}{\xi_{\finite{\sss'}}}}}
		{\sum_{w \in \wgrp} \det w \, \ee^{\inner{w(\finite{\wvec})}{\xi_{\finite{\sss'}}}}}
	= \ee^{2\pi\ii\inner{\finite{\sss}}{\finite{\rrr'}+\finite{\wvec}}} \slthreecharac{\finite{\sss}}{\xi_{\finite{\sss'}}},
\end{equation}
where the final equality is the Weyl character formula.
\end{proof}
\noindent The roles of $\rrr$ and $\sss$ in Proposition \ref{prop:ratioWeylcharac} can be reversed to obtain a similar result involving the character $\slthreecharacnoarg{\finite{\rrr}}$ of $\fslihw{\finite{\rrr}}$ instead.  Both of these results can be viewed as consequences of similar results for the S-matrix of $\saffvoa{\ell}{\slthree}$ for nonnegative integer levels $\ell$, see \cite[Sec.~14.6.3]{DiFCon97} for example.

A generalisation that will prove useful in \cref{sec:bpfus} requires a choice of a dominant integral $\slthree$-weight $\finite{\ttt}$.  We define
\begin{equation} \label{eq:defStensor}
	\WSmatrix{[\wtnewpar{\rrr,\sss \otimes \finite{\ttt}}]}{[\wtnewpar{\rrr',\sss'}]}
	\equiv \sum_{\finite{\ttt'}} \WSmatrix{[\wtnewpar{\rrr,\sss+\ttt'}]}{[\wtnewpar{\rrr',\sss'}]},
\end{equation}
where the sum runs over the (finitely many) weights $\finite{\ttt'}$ of $\fslihw{\finite{\ttt}}$, with multiplicity, and $\ttt'$ denotes the level-$0$ weight of $\aslthree$ whose projection onto the weight space of $\slthree$ is $\finite{\ttt'}$.  Note that we may define this sum for any dominant integral $\slthree$-weight $\finite{\ttt}$, even if $\wtnewpar{\rrr,\sss+\ttt'} \notin \infwtsuv$, by directly substituting the \rhs\ of \eqref{eq:W3Smatrix} for the $\Wminmoduv$ S-matrix.

\begin{proposition} \label{prop:ratioManySs}
	Let $\kk$ be nondegenerate-admissible, $[\wtnewparrs], [\wtnewpar{\rrr',\sss'}] \in \infwtsuv / \ZZ_3$ and $\finite{\ttt}$ be a dominant integral $\slthree$-weight.  Then,
	\begin{equation} \label{eq:ratioManySs}
		\WSmatrix{[\wtnewpar{\rrr,\sss \otimes \finite{\ttt}}]}{[\wtnewpar{\rrr',\sss'}]}
		= \ee^{2\pi\ii\inner{\finite{\sss}}{\finite{\rrr'}+\finite{\wvec}}} \slthreecharac{\finite{\ttt}}{\xi_{\finite{\sss'}}}
		\, \WSmatrix{[\wtnewpar{\rrr,\sss}]}{[\wtnewpar{\rrr',\sss'}]}.
	\end{equation}
\end{proposition}
\begin{proof}
	Substituting \eqref{eq:W3Smatrix} into the definition \eqref{eq:defStensor} gives
	\begin{align}
		\WSmatrix{[\wtnewpar{\rrr,\sss \otimes \finite{\ttt}}]}{[\wtnewpar{\rrr',\sss'}]}
		&= \frac{1}{\sqrt{3} \uu \vv}
			\ee^{2\pi\ii \brac*{\inner{\finite{\rrr}+\finite{\wvec}}{\finite{\sss'}+\finite{\wvec}} + \inner{\finite{\sss}+\finite{\wvec}}{\finite{\rrr'}+\finite{\wvec}}}}
			\sum_{w \in \wgrp} \det w \, \ee^{-2\pi\ii \frac{\vv}{\uu} \inner{w(\finite{\rrr}+\finite{\wvec})}{\finite{\rrr'}+\finite{\wvec}}} \\
		&\quad \cdot \sum_{\finite{\ttt'}} \ee^{2\pi\ii \inner{\finite{\ttt'}}{\finite{\rrr'}+\finite{\wvec}}}
			\sum_{w \in \wgrp} \det w \, \ee^{-2\pi\ii \frac{\uu}{\vv} \inner{w(\finite{\sss}+\finite{\wvec})}{\finite{\sss'}+\finite{\wvec}}}
			\ee^{-2\pi\ii \frac{\uu}{\vv} \inner{w(\finite{\ttt'})}{\finite{\sss'}+\finite{\wvec}}}. \notag
	\end{align}
	Since the weights of $\fslihw{\finite{\ttt}}$ differ by elements of the root lattice $\finite{\rlat}$ of $\slthree$, we may replace $\finite{\ttt'}$ by $\finite{\ttt}$ in the first exponential on the second line.  Moreover, the weights of $\fslihw{\finite{\ttt}}$ are permuted by $\wgrp$ so that
	\begin{equation}
		\WSmatrix{[\wtnewpar{\rrr,\sss \otimes \finite{\ttt}}]}{[\wtnewpar{\rrr',\sss'}]}
		= \ee^{2\pi\ii \inner{\finite{\ttt}}{\finite{\rrr'}+\finite{\wvec}}}
			\sum_{\finite{\ttt'}} \ee^{-2\pi\ii \frac{\uu}{\vv} \inner{\finite{\ttt'}}{\finite{\sss'}+\finite{\wvec}}}
			\WSmatrix{[\wtnewpar{\rrr,\sss}]}{[\wtnewpar{\rrr',\sss'}]}
		= \ee^{2\pi\ii\inner{\finite{\sss}}{\finite{\rrr'}+\finite{\wvec}}} \slthreecharac{\finite{\ttt}}{\xi_{\finite{\sss'}}}
			\, \WSmatrix{[\wtnewpar{\rrr,\sss}]}{[\wtnewpar{\rrr',\sss'}]}. \qedhere
	\end{equation}
\end{proof}
\noindent Again, the roles of $\rrr$ and $\sss$ in this \lcnamecref{prop:ratioManySs} can be reversed to obtain a similar result.

We have seen some identifications of ratios of $\Wminmoduv$ S-matrix elements.  Now, we turn to the evaluation of certain sums of such elements.  In particular, the proof of \cref{thm:type3modularity} requires the computation of a finite weighted sum of S-matrix elements.  By \cref{prop:ratioWeylcharac}, this is equivalent to a sum of weighted $\slthree$ characters and the characters turn out to correspond to multiples of the fundamental weight $\fwt{2}$.  Given that $\slthreecharacnoarg{\fwt{2}} = \ee^{\fwt{2}} + \ee^{\fwt{1} - \fwt{2}} + \ee^{-\fwt{1}}$ and $\fslihw{m \fwt{2}}$ is isomorphic to the $m$-th symmetric product of $\fslihw{\fwt{2}}$, the character of $\fslihw{m\fwt{2}}$ is
\begin{equation}
	\slthreecharacnoarg{m \fwt{2}} = \chsympoly{m}{\ee^{\fwt{2}}}{\ee^{\fwt{1}-\fwt{2}}}{\ee^{-\fwt{1}}},
\end{equation}
where $h_m$ is the $m$-th complete symmetric polynomial.  The following \lcnamecref{prop:sumoffundmods} evaluates the required weighted sum.
\begin{lemma} \label{lem:sumofcompolys}
	Let $x, X_1, X_2, X_3 \in \CC$ be such that the $X_i$ are distinct and $x \ne X_i^{-1}$, for all $i = 1,2,3$.  Suppose further that $X_1^\vv = X_2^\vv = X_3^\vv$ for some $\vv \in \ZZ_{\ge3}$.  Then,
	\begin{equation} \label{eq:sumofcompolys}
		\sum_{m=0}^{\vv-3} x^m \chsympoly{m}{X_1}{X_2}{X_3} = \frac{1-x^{\vv} X_2^{\vv}}{(1-x X_1)(1-x X_2)(1-x X_3)}.
	\end{equation}
\end{lemma}
\begin{proof}
	By computing a partial fraction decomposition for the standard generating function of the complete symmetric polynomials, we arrive at the identity
	\begin{equation}
		\chsympoly{m}{X_1}{X_2}{X_3} = \frac{X_1^{m+2}}{(X_1-X_2)(X_1-X_3)} + \frac{X_2^{m+2}}{(X_2-X_1)(X_2-X_3)} + \frac{X_3^{m+2}}{(X_3-X_1)(X_3-X_2)}.
	\end{equation}
	Since $X_1^\vv = X_2^\vv = X_3^\vv$, explicit calculation now gives
	\begin{align}
		\sum_{m=0}^{\vv-3} &x^m \chsympoly{m}{X_1}{X_2}{X_3} \\
		&= \frac{X_1^2(1-x^{\vv-2}X_1^{\vv-2})}{(X_1-X_2)(X_1-X_3)(1-x X_1)}
			+ \frac{X_2^2(1-x^{\vv-2}X_2^{\vv-2})}{(X_2-X_1)(X_2-X_3)(1-x X_2)}
			+ \frac{X_3^2(1-x^{\vv-2}X_3^{\vv-2})}{(X_3-X_1)(X_3-X_2)(1-x X_3)} \notag \\
		&= \frac{X_1^2-x^{\vv-2}X_2^{\vv}}{(X_1-X_2)(X_1-X_3)(1-x X_1)}
			+ \frac{X_2^2-x^{\vv-2}X_2^{\vv}}{(X_2-X_1)(X_2-X_3)(1-x X_2)}
			+ \frac{X_3^2-x^{\vv-2}X_2^{\vv}}{(X_3-X_1)(X_3-X_2)(1-x X_3)} \notag \\
		&= \frac{1-x^{\vv} X_2^{\vv}}{(1-x X_1)(1-x X_2)(1-x X_3)}. \qedhere \notag
	\end{align}
\end{proof}

\begin{proposition} \label{prop:sumoffundmods}
	Let $\kk$ be nondegenerate-admissible, $[\lambda'] = [\wtnewpar{\rrr',\sss'}] \in \infwtsuv / \ZZ_3$, $[j'] \in \RR/\ZZ$, $\xi_{\finite{\sss'}}$ be as in \eqref{eq:ratioWeylcharac} and $x = -\ee^{2\pi\ii \left(\inner{\finite{\rrr'}+\finite{\wvec}}{\fwt{2}} - (j'-\kappa)\right)}$.  Then,
	\begin{equation}
		\sum_{m=0}^{\vv-3} x^m \slthreecharac{m\fwt{2}}{\xi_{\finite{\sss'}}}
		= \brac*{1 - \ee^{-2\pi\ii\vv (j'-\kappa )} \ee^{2\pi\ii\vv j^{\twist}(\lambda')}} \frac{\ee^{3\pi\ii (j'-\kappa)}}{8\sin(\pi c_1)\sin(\pi c_2)\sin(\pi c_3)},
	\end{equation}
	where $c_i = (j'-\kappa) - j^{\twist}\brac[\big]{\outaut^i(\lambda')}$.
\end{proposition}
\begin{proof}
	Note first that for each $\slthree$-weight $\omega$, we have
	\begin{equation}
		\brac*{\ee^{\inner{\omega}{\xi_{\finite{\sss'}}}}}^{\vv} = \ee^{-2\pi\ii\uu \inner{\omega}{\finite{\sss'}+\finite{\wvec}}},
	\end{equation}
	which is clearly invariant under shifting $\omega$ by elements of $\finite{\rlat}$.  It follows that if we set
	\begin{equation}
		X_1 = \ee^{\inner{\fwt{2}}{\xi_{\finite{\sss'}}}}, \quad
		X_2 = \ee^{\inner{\fwt{1}-\fwt{2}}{\xi_{\finite{\sss'}}}} \quad \text{and} \quad
		X_3 = \ee^{\inner{-\fwt{1}}{\xi_{\finite{\sss'}}}},
	\end{equation}
	then $X_1^{\vv} = X_2^{\vv} = X_3^{\vv}$.  Applying \cref{lem:sumofcompolys} now gives
	\begin{equation} \label{eq:usedalemma}
		\sum_{m=0}^{\vv-3} x^m \slthreecharac{m\fwt{2}}{\xi_{\finite{\sss'}}}
		= \sum_{m=0}^{\vv-3} x^m \chsympoly{m}{X_1}{X_2}{X_3} = \frac{1-x^{\vv} X_2^{\vv}}{(1-x X_1)(1-x X_2)(1-x X_3)}.
	\end{equation}
	Noting that
	\begin{subequations} \label{eq:identjtw}
		\begin{equation}
			\inner{\finite{\rrr'}+\finite{\wvec}-\tfracuv (\finite{\sss'}+\finite{\wvec})}{\fwt{2}} = \tfrac{1}{3} (r'_1+2r'_2+3) - \tfrac{\uu}{3\vv} (s'_1+2s'_2+3)
			= j^{\twist}\brac[\big]{\outaut^2(\lambda')} + \tfrac{1}{2}
		\end{equation}
		and, similarly
		\begin{equation}
			\inner{\finite{\rrr'}+\finite{\wvec}-\tfracuv (\finite{\sss'}+\finite{\wvec})}{\fwt{1}-\fwt{2}} = j^{\twist}(\lambda') + \tfrac{1}{2} \quad \text{and} \quad
			\inner{\finite{\rrr'}+\finite{\wvec}-\tfracuv (\finite{\sss'}+\finite{\wvec})}{-\fwt{1}}
			= j^{\twist}\brac[\big]{\outaut(\lambda')} + \tfrac{1}{2},
		\end{equation}
	\end{subequations}
	we obtain
	\begin{equation}
		xX_1 = \ee^{-2\pi\ii (j'-\kappa)} \ee^{2\pi\ii j^{\twist}\brac*{\outaut^2(\lambda')}}, \quad
		xX_2 = \ee^{-2\pi\ii (j'-\kappa)} \ee^{2\pi\ii j^{\twist}(\lambda')} \quad \text{and} \quad
		xX_3 = \ee^{-2\pi\ii (j'-\kappa)} \ee^{2\pi\ii j^{\twist}\brac*{\outaut(\lambda')}}.
	\end{equation}
	Substituting into \eqref{eq:usedalemma}, we arrive at the desired result by observing that $\sum_{i \in \ZZ_3} j^{\twist}\brac*{\outaut^i(\lambda')} = -\frac{3}{2}$.
\end{proof}

\subsection{Fusion rules for $\Wminmoduv$} \label{sec:W3fusion}

As the $\Wthree$ minimal models are rational and $C_2$-cofinite \cite{AraAss15,AraRat15}, their fusion coefficients may be computed from the Verlinde formula \cite{VerFus88,HuaVer08}.  However, these coefficients beautifully (and usefully!) factorise as products of fusion coefficients for rational $\slthree$ minimal models.

Recall that for $\ell \in \NN$, the simple affine \voa\ $\sslvoa{\ell} = \slminmod{\ell+3}{1}$ of level $\ell$ is rational and $C_2$-cofinite \cite{FreVer92}.  Its simple modules are the integrable \hw\ $\aslthree$-modules $\slihw{\ttt}$ whose highest weights $\ttt$ lie in $\pwlat{\ell}$.  The fusion rules of $\sslvoa{\ell}$ take the form
\begin{equation} \label{eq:FRsl3}
	\fusion{\slihw{\ttt}}{\slihw{\ttt'}} \cong \bigoplus_{\ttt'' \in \pwlat{\ell}} \slfuscoeff{\ell}{\ttt}{\ttt'}{\ttt''} \slihw{\ttt''},
\end{equation}
where the fusion coefficients $\slfuscoeff{\ell}{\ttt}{\ttt'}{\ttt''}$ are known.  We shall not try to write them out explicitly, but instead note that they may be computed in several ways including the \emph{Kac--Walton formula} \cite{KacInf90,WalFus90,WalAlg90,FucWZW90}:
\begin{equation} \label{eq:KacWalton}
	\slfuscoeff{\ell}{\ttt}{\ttt'}{\ttt''} = \sum_{\substack{w \in \affine{\grp{S}}_3 \\ w \cdot \ttt'' \in \pwlat{\ell}}}
		\det w \, \sltencoeff{\finite{\ttt}}{\finite{\ttt'}}{\finite{w \cdot \ttt''}}.
\end{equation}
Here, $\awgrp$ is the affine Weyl group of $\aslthree$, $\finite{\ttt} = [t_1,t_2]$ is the projection of $\ttt$ onto the weight space of $\slthree$, and $\sltencoeff{\finite{\ttt}}{\finite{\ttt'}}{\finite{\ttt''}}$ denotes the tensor product (Littlewood--Richardson) coefficients of the simple \fdim\ $\slthree$-modules $\fslihw{\finite{\ttt}}$:
\begin{equation}
	\fslihw{\finite{\ttt}} \otimes \fslihw{\finite{\ttt'}}
	\cong \bigoplus_{\finite{\ttt''}} \sltencoeff{\finite{\ttt}}{\finite{\ttt'}}{\finite{\ttt''}} \fslihw{\finite{\ttt''}}.
\end{equation}
We also mention that the $\sslvoa{\ell}$ fusion coefficients satisfy
\begin{equation} \label{eq:sl3fuscoeffoutaut}
	\slfuscoeff{\ell}{\outaut(\ttt)}{\ttt'}{\outaut(\ttt'')} = \slfuscoeff{\ell}{\ttt}{\outaut(\ttt')}{\outaut(\ttt'')} = \slfuscoeff{\ell}{\ttt}{\ttt'}{\ttt''},
\end{equation}
where $\outaut$ is defined in \eqref{eq:Z3action'}, see \cite[Eq.~(16.9)]{DiFCon97} for example.

With this setup, we present the factorisation of the $\Wthree$ minimal model fusion coefficients.  Recall that $\finite{\rlat}$ denotes the root lattice of $\slthree$.
\begin{theorem}[{\cite[Thm.~4.3]{FreCha92}}] \label{thm:W3fusion}
	Let $\kk$ be nondegenerate-admissible.  Then, the $\Wminmoduv$ fusion coefficients are given by
	\begin{equation} \label{eq:W3fuscoefffactorisation}
		\Wfuscoeff{[\lambda]}{[\lambda']}{[\lambda'']} = \slfuscoeff{\uu-3}{\rrr}{\rrr'}{\rrr''} \slfuscoeff{\vv-3}{\sss}{\sss'}{\sss''},
	\end{equation}
	where we choose the representatives $\lambda, \lambda', \lambda'' \in \infwtsuv$ of $[\lambda], [\lambda'], [\lambda''] \in \infwtsuv / \ZZ_3$ so that:
	\begin{itemize}
		\item If $\uu \in 3\ZZ$, then take $\finite{\sss}, \finite{\sss'}, \finite{\sss''} \in \finite{\rlat}$.
		\item If $\vv \in 3\ZZ$, then take $\finite{\rrr}, \finite{\rrr'}, \finite{\rrr''} \in \finite{\rlat}$.
		\item If $\uu, \vv \notin 3\ZZ$, then take either $\finite{\rrr}, \finite{\rrr'}, \finite{\rrr''} \in \finite{\rlat}$ or $\finite{\sss}, \finite{\sss'}, \finite{\sss''} \in \finite{\rlat}$ (it does not matter which).
	\end{itemize}
\end{theorem}
\noindent For example, the fusion coefficients for $\vv=3$ take the form $\Wfuscoeff{[\lambda]}{[\lambda']}{[\lambda'']} = \slfuscoeff{\uu-3}{\rrr}{\rrr'}{\rrr''}$, with $\finite{\rrr}, \finite{\rrr'}, \finite{\rrr''} \in \finite{\rlat}$, because in this case $\sss = \sss' = \sss'' = [0,0,0]$.  It follows that the $\Wminmod{\uu}{3}$ fusion ring coincides with the subring of the $\sslvoa{\uu-3}$ fusion ring spanned by the $\slihw{\lambda}$ with $\finite{\rrr} = [\lambda_1, \lambda_2] \in \finite{\rlat}$.  That this indeed constitutes a subring follows easily from \eqref{eq:KacWalton}.

An simple but instructive example is provided by the minimal model $\Wminmod{5}{3}$.  The $\slthree$ weights $\finite{\rrr} = [r_1,r_2]$ corresponding to $\sslvoa{2}$ are $[0,0]$, $[1,0]$, $[0,1]$, $[2,0]$, $[0,2]$ and $[1,1]$ (as usual when $\vv=3$, we only have $\finite{\sss} = [0,0]$).  Only the first and the last lie in $\finite{\rlat}$.  Because $[0,0]$ corresponds to the vacuum module of $\sslvoa{2}$, there is only one nontrivial fusion rule: $\slihw{[0,1,1]} \times \slihw{[0,1,1]}$.  To compute it, we start with the tensor product rule
\begin{equation}
	[1,1] \otimes [1,1] = [2,2] \oplus [3,0] \oplus [0,3] \oplus 2 [1,1] \oplus [0,0].
\end{equation}
It is easy to check that the level-$2$ affinisations $[-1,3,0]$ and $[-1,0,3]$ of $[3,0]$ and $[0,3]$, respectively, are preserved by the shifted action of the affine reflection $w_0 \in \affine{\grp{S}}_3$.  They can therefore never be reflected into $\pwlat{2}$, hence the Kac--Walton formula \eqref{eq:KacWalton} shows that the modules corresponding to $[3,0]$ and $[0,3]$ do not appear in the fusion rule (in fact, they are not even $\sslvoa{2}$-modules).  Similarly, the same reflection sends $[-2,2,2]$ to $[0,1,1]$, hence $[2,2]$ cancels one of the two $[1,1]$-modules.  The $\sslvoa{2}$ fusion rule is thus
\begin{equation}
	\slihw{[0,1,1]} \fuse \slihw{[0,1,1]} = \slihw{[0,1,1]} \oplus \slihw{[2,0,0]}
\end{equation}
and we thereby recover the well known fusion rules of the Yang--Lee Virasoro minimal model $\virminmod{2}{5}$.  Of course, this is expected because $\Wminmod{5}{3} \cong \virminmod{2}{5}$.

\flushleft
\providecommand{\opp}[2]{\textsf{arXiv:\mbox{#2}/#1}}
\providecommand{\pp}[2]{\textsf{arXiv:#1 [\mbox{#2}]}}

\end{document}